\documentclass[reqno]{amsproc}
\usepackage{amsmath,leftidx}
\usepackage{xspace}
\DeclareFontFamily{OT1}{pzc}{}
\DeclareFontShape{OT1}{pzc}{m}{it}%
              {<-> s * [1.2] pzcmi7t}{}
\DeclareMathAlphabet{\mathpzc}{OT1}{pzc}%
                                 {m}{it}
\usepackage{calrsfs}
\DeclareMathAlphabet{\pazocal}{OMS}{zplm}{m}{n}

\usepackage[psamsfonts]{amssymb}
\usepackage[latin1]{inputenc}
\usepackage{graphicx}
\usepackage{amsthm,amsopn,pifont} 
\usepackage{subfigure, mathtools}
\usepackage{multirow, float}
\usepackage{psfig,latexsym,graphicx, epsfig}
\usepackage[usenames,dvipsnames]{color}
\usepackage[implicit=true, 
  colorlinks=true,linkcolor=Black,citecolor=Black,urlcolor=Black]%
{hyperref}

\input xy
\xyoption{all}

\newtheorem{theorem}{Theorem}[section]
\newtheorem{lem}[theorem]{Lemma}
\newtheorem{conj}[theorem]{Conjecture}
\newtheorem{prop}[theorem]{Proposition}
\newtheorem{coro}[theorem]{Corollary}

\newtheorem*{lemma*}{Lemma}
\newtheorem*{theorem*}{Theorem}

\newtheorem*{BorsukThm*}{Theorem of Borsuk}
\newtheorem*{BEThm*}{Theorem of Buff and Epstein}
\newtheorem*{RislerThm*}{Theorem of Risler}

\theoremstyle{remark}
\newtheorem{remark}[theorem]{\bf Remark}

\theoremstyle{definition}
\newtheorem{definition}[theorem]{Definition}
\newtheorem{example}[theorem]{Example}

\title[Antipode Preserving Cubic Maps]
{Antipode Preserving Cubic Maps:\\ the Fjord Theorem}

\author[A. Bonifant]{Araceli Bonifant}
\address{Mathematics Department, University of Rhode Island, RI 02881}
\email{bonifant@math.uri.edu}
\thanks{The first author wishes to thank the Institute for Mathematical
Sciences at Stony Brook University, where she spent her sabbatical year, and
 ICERM for their support to this project.}
\author[X. Buff]{Xavier Buff}
\address{Universit\'e Paul Sabatier, Institut de Mathématiques de Toulouse,
 31062 Toulouse  Cedex, France }
\thanks{The second author wishes to thank the Clay Mathematics Institute, 
ICERM and  IUF for supporting this research.}
\email{xavier.buff@math.univ-toulouse.fr}
\author[J. Milnor]{John Milnor}
\address{Institute for Mathematical Sciences, Stony Brook University, 
Stony Brook, NY 11794-3660}
\thanks{The third author wishes to thank ICERM for their support towards this project.}
\email{jack@math.sunysb.edu}

\subjclass[2013]{37D05, 37F15, 37F10}

\def\eps{\varepsilon} 
\def\ph{{\varphi}} 

\renewcommand{\Im}{\mathop{\mathrm{Im}}} 
\newcommand\ssm{{\smallsetminus}} 

\def\Z{{\mathbb Z}} 
\def\Q{{\mathbb Q}} 
\def\R{{\mathbb R}} 
\def\C{{\mathbb C}} 
\def\bH{{\mathbb H}} 
\def\Chat{{\widehat\C}} 
\def\D{{\mathbb D}} 
\def\RZ{{\R/\Z}} 
\def\CC{\raisebox{-0.15em}{\includegraphics[width=.13in]{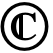}}} 
\def\binf{{\boldsymbol\infty}} 

\def\dyn{\mathcal}
\def\J{{\dyn J}} 
\def\dR{{\dyn R}} 
\def\sR{{\dyn S}} 
\def\sRlim{{\dyn S}} 

\newcommand\B{{\dyn B}} 
\def\Ba{\B_q} 
\def\Bo{\B_0} 
\def\Bi{\B_\infty} 

\def\Bv{{\B_q^{\rm vis}}} 
\def\Vq{{\dyn V}_q} 

\def\cC{{\dyn C}} 

\def\HR{{\dyn A}} 
\def\BRo{{\dyn C}_0} 
\def\BRi{{\dyn C}_\infty} 
\def\Co{{C_0}}
\def\Ci{{C_\infty}}

\def\cO{{\dyn O}} 
\def\cOo{\cO_0} 
\def\cOi{\cO_\infty}

\def\para{\pazocal}
\def\H{{\para H}} 
\def\HO{{\H_0}}
\def\HOtilde{{\widetilde \H_0}}
\def\HOrat{{\H_0^{\rm rat}}}
\def\HOpoly{{\H_0^{\rm poly}}}
\def\HOtilderat{{\widetilde \H_0^{\rm rat}}}
\def\HOtildepoly{{\widetilde \H_0^{\rm poly}}}
\def\pR{{\para R}} 
\def\tri{{\para W}}
\def\L{{\para L}} 

\def\V{{V_\bth}} 
\def\X{{X_\bth}} 
\def\bLambda{{{\Lambda}_\bth}} 
\def\bTheta{{{\Theta}_\bt}} 
\def\gap{G} 
\def\cgap{{{G}_\bth}} 
\def\cgapp{{G'_\bth}} 
\def\I{{I}} 

\def\Sector{{\bf S}}

\def\ant{{\mathpzc A}} 

\def\E{{\rm exp}}

\def\fract{{\bf frac}}
\def\floor{{\bf floor}}
\def\mult{{\bf mult}}
\def\rot{{\rm rot}}
\def\trans{{\rm transl}}

\def\cp{{\bf c}}
\def\cpo{{\cp_0}} 
\def\cpi{{\cp_\infty}} 

\def\fixq{{\boldsymbol\xi}} 
\def\fixp{{\gamma}} 
\def\fixf{{\xi}} 

\def\fb{{\mathfrak b}}
\def\bot{\fb}
\def\botq{\fb_q}

\newcommand\m{{\bf m}} 

\def\bdeta{{\boldsymbol\eta}}
\def\bzeta{{\boldsymbol\zeta}}
\def\bphi{{\boldsymbol\ph}}
\def\bpsi{{\boldsymbol\psi}}
\def\brho{{\boldsymbol\rho}}

\def\bth{{\boldsymbol\vartheta}} 
\def\bthmin{\theta_{\rm min}}
\def\bthmax{\theta_{\rm max}}

\def\q{{q_\infty}} 
\def\qsqinf{{q^2_\infty}}

\def\bt{{\bf t}} 

\def\slope{{\bf s}}

\def\mat{\perp\!\!\perp}
\def\bA{{\mathcal  A}}

\begin{document}
\begin{abstract}
This note will study a family of cubic rational maps which carry 
 antipodal points of the Riemann sphere to antipodal points. We 
focus particularly on the \textbf{\textit{fjords}},
 which are part of the central hyperbolic component
but stretch out to infinity. These serve to decompose the parameter
 plane into subsets, each of which is characterized by a corresponding
\textbf{\textit{rotation number}}.
\end{abstract}

\maketitle

\tableofcontents

\medskip

\setcounter{equation}{0}
\section{Introduction.}\label{s1}\bigskip
By a classical theorem of Borsuk (see \cite{Bor}): 

\begin{quote} {\it There exists a degree $d$ map from the $n$-sphere
to itself carrying antipodal points to antipodal points if and
    only if $d$ is odd.}
\end{quote}\smallskip

For the \textbf{\textit{Riemann sphere}},  $\Chat=\C\cup\{\infty\}
\cong~S^2$, the  \textbf{\textit{antipodal map}}  is defined to be the
 fixed point free map $\ant(z)=-1/\overline z$.
We are interested in rational maps from the Riemann sphere to itself
which carry antipodal points to antipodal points\footnote{The proof of
Borsuk's Theorem in this special case is quite easy. A rational map of degree
$d$ has $d+1$ fixed points, counted with multiplicity. If these fixed
 points occur in antipodal pairs, then $d+1$ must be even.},
 so that $\ant\circ f=f\circ\ant$. If all of the zeros $q_j$ of $f$
lie in the finite plane, then $f$ can be written uniquely as
\[f(z)=u\prod_{j=1}^d\frac{z-q_j}{1+ \overline q_j z}\qquad
{\rm with}\qquad |u|=1.\]

The simplest interesting case is in degree $d=3$. To fix ideas
we will discuss only the special case\footnote{For a different special case,
consisting of maps with only two critical points,
see \cite[\S7]{Milnor1} as well as \cite{GH}.}
where $f$ has a critical fixed point. Putting this fixed point at the origin,
we can take $q_1=q_2=0$, and write $q_3$ briefly as $q$. It will be
 convenient to take $u=-1$, so that the map takes the form
\begin{equation}\label{e0}
 f(z)=f_q(z)= z^2\frac{q-z}{1+\overline{q} z}.
\end{equation}
(Note that there is no loss of generality in choosing one particular value for
$u$, since we can always change to any other value of
$u$ by rotating the $z$-plane appropriately.)

\begin{figure}[H]
\centerline{\includegraphics[width=2.8in]{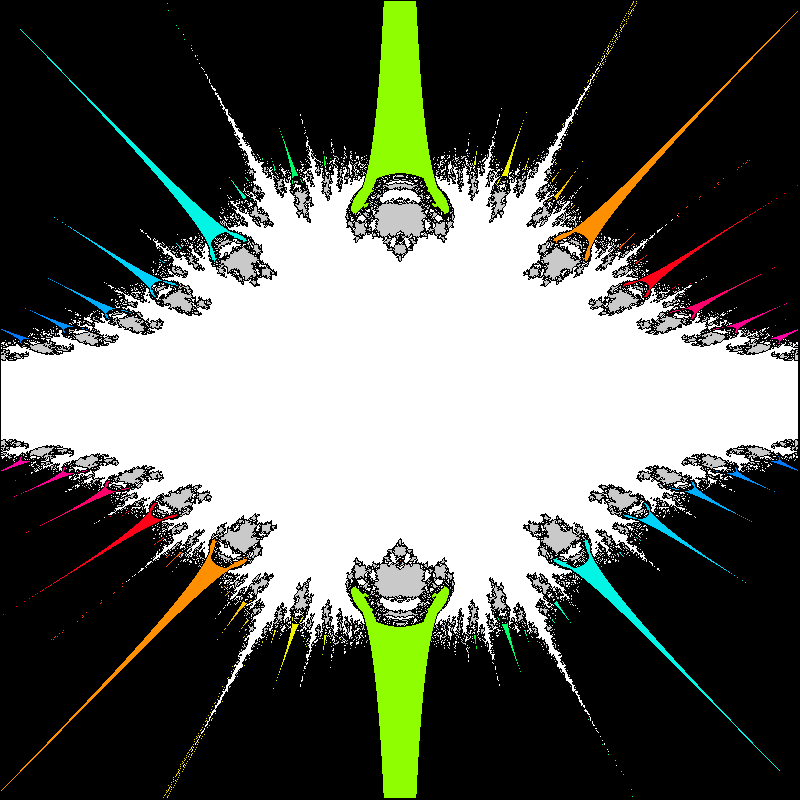}}
\caption{\it \label{f1} The $q$-parameter plane.}
\end{figure}

Figure~\ref{f1} illustrates the $q$-parameter plane for this family of maps:
(\ref{e0}).

$\bullet$ The central white region, resembling a porcupine, is the hyperbolic
 component   centered at the map
$f_0(z)=-z^3$. It consists of all $q$  for which the Julia set is a
Jordan curve separating the basins of zero and infinity. This region is
simply-connected  (see  Lemma~\ref{L-can-dif});
but its boundary is very
 far  from locally connected.  (See \cite{BBM2}, the sequel to this paper.)

\smallskip

$\bullet$ The colored regions in Figure~\ref{f1}
will be called \textbf{\textit{tongues}},
 in analogy  with Arnold tongues. Each of these is a
 hyperbolic component, stretching out to infinity. Each representative map
$f_q$ for such a tongue has a self-antipodal cycle of attracting
basins which are arranged in a loop separating zero from infinity.
 (Compare Figure~\ref{F-rat}.) Each such loop has a well defined
 combinatorial rotation number, as seen from the origin, necessarily rational
 with even denominator. These rotation numbers are indicated in Figure~\ref{f1}
by colors which range from red (for rotation number close to zero)
to blue (for rotation number close to one).
\smallskip

\begin{figure}[H]
\centerline{
\includegraphics[width=3.in]{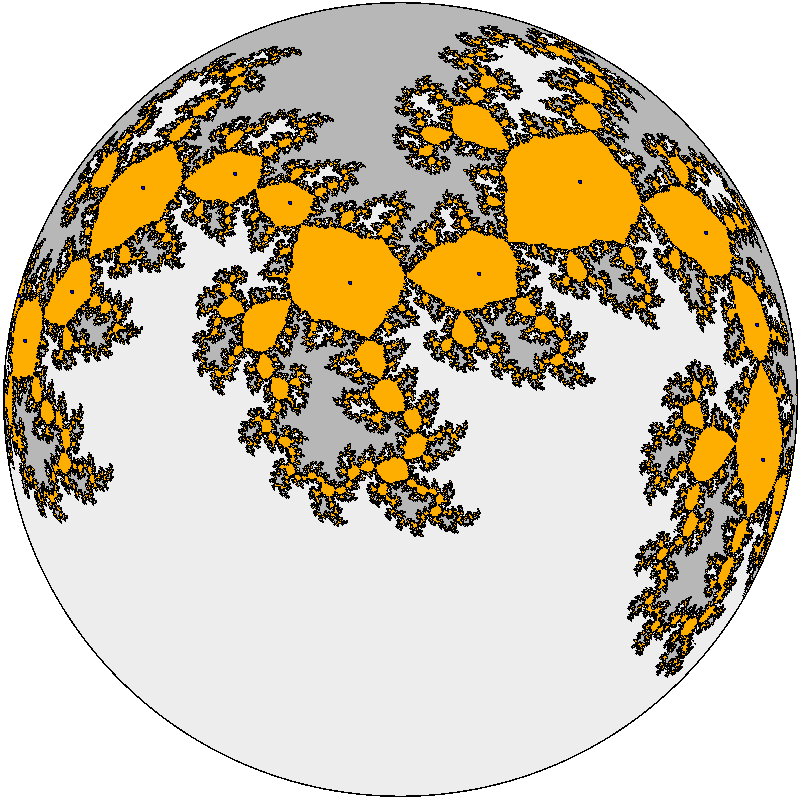}} 
\smallskip
\caption{\it
Rational rotation number with even denominator: Julia set for a point in the 
$(21/26)$-tongue. Here the Riemann sphere has been projected orthonormally onto
the plane, so that only one hemisphere is visible. 
 Zero is at the bottom and infinity at the top.
 Note the self-antipodal attracting orbit of period $26$,
which has been marked. The associated ring of attracting
 Fatou components separates the basins of zero and infinity.
$($If the sphere were transparent, then we would see
the same figure on the far side, but rotated $180^\circ$
and with  light and dark gray interchanged.$)$
\label{F-rat}
}\end{figure}

\begin{figure}\centerline{
\includegraphics[width=2.4in]{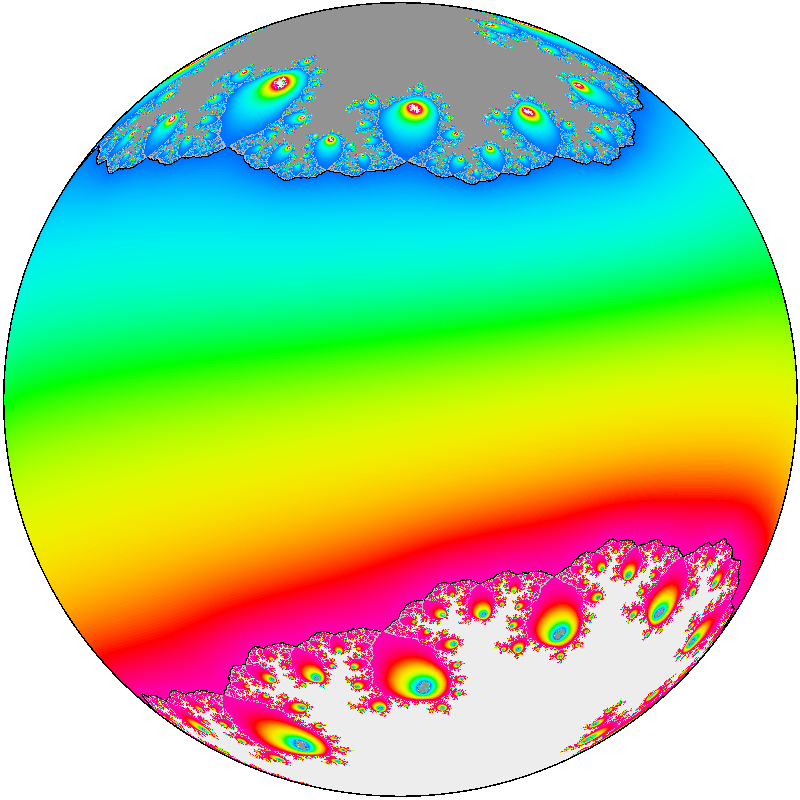}\qquad 
\includegraphics[width=2.4in]{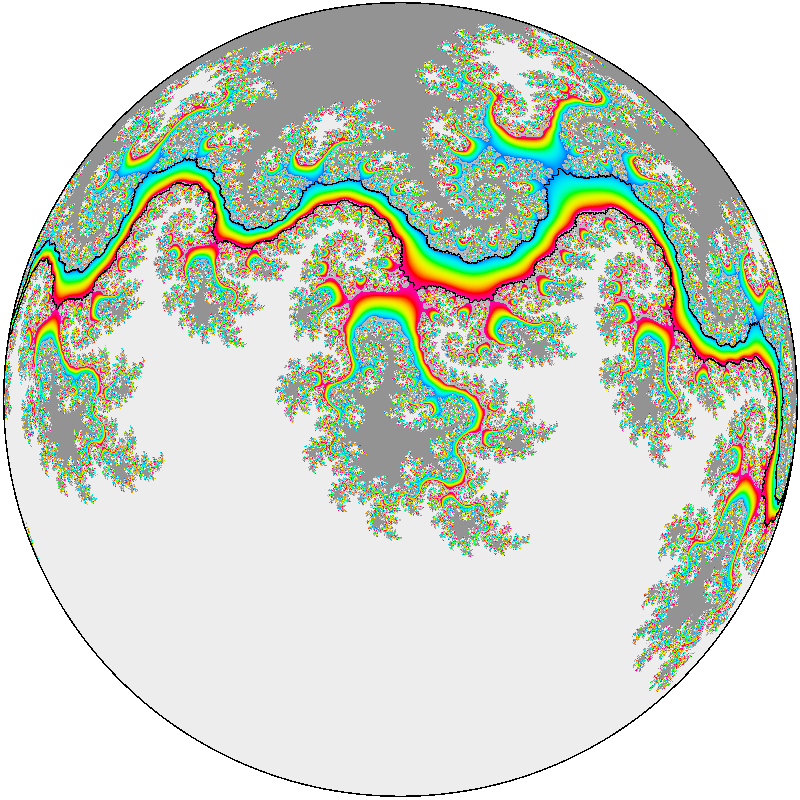}}\smallskip 
\caption{\it Irrational rotation number:
two Herman rings on the Riemann sphere, 
projected orthonormally. 
The boundaries of the rings have been outlined. Otherwise,
every point $z$ maps to a point $f_q(z)$ of the same color.
Left example: $q= 1-6i$, modulus $\approx 0.2081$. 
 Right example: 
 $q= 3.21- 1.8i$, modulus $\approx 0.0063$.
\label{F-hr}}

\end{figure}

$\bullet$ The black region   in Figure~\ref{f1} is the 
\textbf{\textit{Herman ring locus}}. For the
 overwhelming majority of parameters $q$ in this region,
 the map $f_q$ has
 a  Herman ring which separates the basins of zero and infinity.
(Compare Figure~\ref{F-hr}.) Every such
Herman ring has a  well defined rotation number (again as seen from zero),
which is always a Brjuno number. 
If we indicate these rotation numbers also by color, then there is a
smooth gradation, so that the tongues no longer
stand out. (Compare Figures~\ref{f2} and \ref{f3}.)
\smallskip

$\bullet$ For rotation numbers with odd denominator, there are
 no such tongues or rings. Instead there are channels leading out to infinity
within the central hyperbolic component. (Compare Figures~\ref{f1} and 
\ref{F-fj}.) These will  be called \textbf{\textit{fjords}}, and will be a
 central topic of this paper.
\smallskip

$\bullet$ The gray regions and the nearby small white regions in 
Figure~\ref{f1} represent
\textbf{\textit{capture components}}, such that both of the
critical points in $\C\ssm\{0\}$ have orbits which eventually land
 in the immediate basin of zero (for white) or infinity (for gray).\medskip

Note that Figure~\ref{f1} is invariant under $180^\circ$ rotation.
In fact each
$f_q$ is linearly conjugate to $f_{-q}(z) = -f_q(-z)$. In order to
 eliminate this duplication, it is often convenient to use $q^2$ as
 parameter.  The $q^2$-plane   of Figure~\ref{f2}, 
could be referred to as
the  \textbf{\textit{moduli  space}},  since each $q^2\in\C$
 corresponds to a unique
holomorphic conjugacy class of mappings of the form (\ref{e0}) with a 
marked critical fixed point at the origin.\footnote{Note that we do 
not allow
conjugacies which interchange the roles of zero and infinity, and hence
replace $q$ by $\overline q$. The punctured $q^2$-plane
(with the origin removed) could itself be considered as a parameter space,
for example by setting $z=w/q$  
to obtain the family of linearly conjugate maps
$F_{q^2}(w)= qf_q(w/q)=w^2(q^2-w)/(q^2+|q^2|w)$, which are well defined
 for $q^2\ne 0$.}
 \smallskip

\begin{figure}[ht]
\centerline{\includegraphics[width=2.8in]{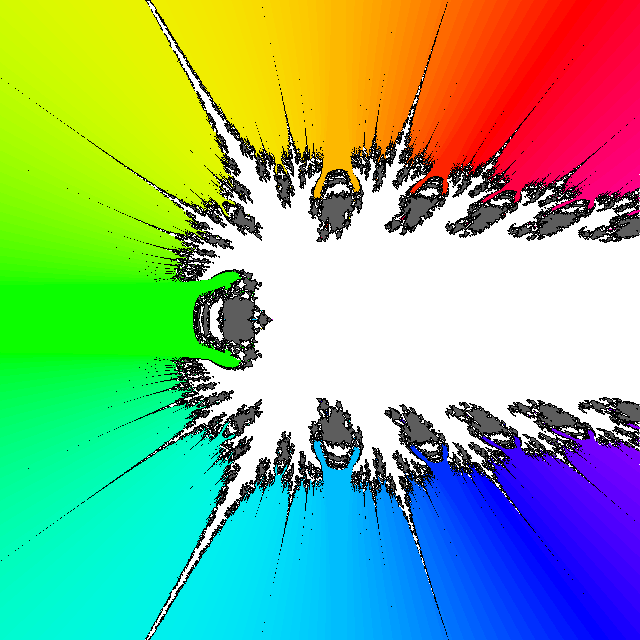}}
\caption{\it \label{f2} The $q^2$-plane.
 Here the colors code the 
rotation number, not distinguishing between tongues and Herman rings.}
\smallskip

\centerline{\includegraphics[width=2.9in]{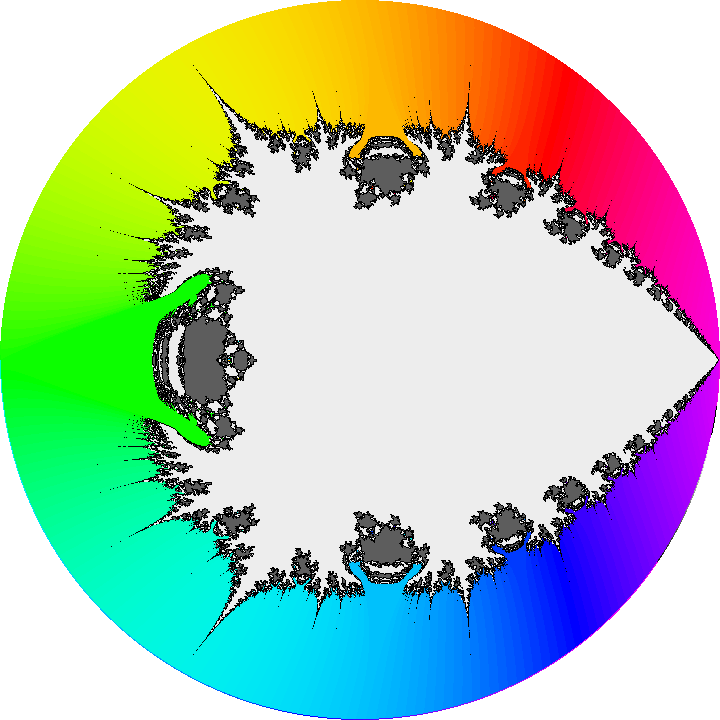}}
\caption{\it \label{f3}  Image of the circled $q^2$-plane under a 
homeomorphism $\quad \protect\CC \to\overline\D$\quad
which shrinks the circled plane to the unit disk.}
\end{figure}

\begin{definition}\label{D-cp} 
It is often useful to add a circle of points at infinity to the complex
plane. The \textbf{\textit{circled complex plane}}
$$\CC= \C\cup ({\rm circle~at~infinity}) $$
will mean the compactification of the complex plane which is obtained
 by adding one  point at infinity, $\binf_\bt\in\partial\CC$,
for each $\bt\in\RZ$; where a sequence of points $w_k\in\C$ 
 converges
to $\binf_\bt$ if and only if $|w_k|\to\infty$ and 
$w_k/|w_k|\to e^{2\pi i\bt}$.
By definition, this circled plane is homeomorphic to the closed unit
disk $\overline\D$ under a homeomorphism 
$\CC\stackrel{\cong}{\longrightarrow}\overline\D$ of the form
$$ w=re^{2\pi i \bt}\mapsto \eta(r)e^{2\pi i\bt}$$
where $
\eta:[0,\infty]\to[0,1]$ is any convenient homeomorphism. (For example
we could take
$$\eta(r)=r\Big{/}\sqrt{1+r^2} $$
but the precise choice of $\eta$ will not be important.)

In particular, it is useful to compactify the $q^2$-plane 
in this way; and
for illustrative purposes, it is often helpful to show the image of
the circled $q^2$-plane under such a homeomorphism, since we can then 
visualize the entire circled plane in one figure. (Compare Figures~\ref{f2}
 and \ref{f3}.) The circle at infinity for the $s$-plane
has an important dynamic interpretation:
\end{definition}
\smallskip

\begin{lem}\label{L1}
Suppose that $q^2$ converges to the point 
$\binf_\bt$ on the circle
at infinity. Then for any $\eps>0$
the associated maps $f_{± q}$ both converge uniformly to the rotation
$ z\mapsto e^{2\pi i \bt}z $
throughout the annulus $\eps\le|z|\le 1/\eps$.
\end{lem}

 Therefore,  most of the chaotic dynamics must be concentrated 
within the small disk $|z|<\eps$ and its antipodal image.
The proof is a straightforward exercise.\qed\bigskip

\subsection*{\bf Outline of the Main Concepts and Results.}\smallskip

In Section~\ref{s2} of the paper we  discuss the various types of Fatou
 components
for maps $f_q$ with $q \in {\mathbb C}$. In particular, we show that
 any Fatou component which
 eventually maps to a Herman ring is an annulus, but that all other Fatou
components are simply-connected. (See Theorem~\ref{t1}.)

We give a rough classification of hyperbolic components  (either in the
$q$-plane or in the  $q^2$-plane), noting that every hyperbolic  component is
simply-connected, with a unique critically finite ``center'' point.
(See Lemma~\ref{Lhyps}.) 
\smallskip

In Section~\ref{s-chc} we focus on the study of the central hyperbolic
 component $\H_0$ in the $q^2$ plane. 
We construct a canonical (but not conformal !) diffeomorphism
between $\HO$ and the open unit disk (Lemma~\ref{L-can-dif}). Using
 this, we
 can consider parameter rays ($=$ stretching rays), which correspond to 
radial lines in the unit disk, each with a specified angle 
$\bth\in\RZ$. We prove that the following landing properties of the
parameter rays $\pR_\bth$.
\medskip

{\bf (Theorem \ref{t-inray}.)} 
{\it If the angle $\bth\in\Q/\Z$ is periodic under doubling, then
the parameter ray $\pR_\bth$  either:
\begin{itemize}
\item[{\bf(a)}] lands at a parabolic point on the boundary 
$\partial\HO$,
\smallskip

\item[{\bf(b)}] accumulates  on a curve of parabolic boundary points, or 
\smallskip

\item[{\bf(c)}] lands at a point $\binf_\bt$
on the circle at infinity, where $\bt = \bt(\bth)$ can be described as the rotation
number of $\bth$ under angle doubling. {\em (Compare \S\ref{s-4}.)}
\end{itemize}
}

 The proof of this Theorem makes use of  \textbf{\textit{dynamic internal rays}}
$\dR_{q,\theta}$, which are contained within the immediate basin of zero $\Ba$ 
 for the map $f_q$. These are defined using the B\"ottcher coordinate $\botq$, 
 which is well defined throughout some neighborhood of zero in the $z$-plane whenever 
 $q\ne 0$.

\medskip


In Section~\ref{s-4} we study, for $q^2\in \HO$, the set of landing points of dynamic internal rays $\dR_{q,\theta}$.  
More precisely, if $q^2\in \pR_\bth$ with $\bth\in \RZ$, then the dynamic internal  ray $\dR_{q,\theta}$ lands if and only if 
$\theta$ is not an iterated preimage of $\bth$ under doubling. In addition, the Julia set $\J(f_q)$ is a quasicircle and 
there is a homeomorphism $\bdeta_q:\RZ\to \J(f_q)$ conjugating the tripling map to $f_q$: 
\[\bdeta_q(3x) = f_q\bigl(\bdeta_q(x)\bigr).\] 
This conjugacy is uniquely defined up to orientation, and up to the involution $x\leftrightarrow x+1/2$ of the circle. 
We specify a choice and  present an algorithm which, given $\bth\in \RZ$ and an angle $\theta\in \RZ$ which is not 
an iterated preimage of $\bth$ under doubling, returns the base three expansion of the angle $x\in \RZ$ such that 
$\bdeta_q(x)$ is the landing point of $\dR_{q,\theta}$.

In Section~\ref{sec:dynrot} we define the \textbf{\textit{dynamic rotation number}} $\bt$ of a map $f_q$ for $q^2\in \HO\ssm\{0\}$ as follows: the set of points $z\in \J(f_q)$ which are landing points of internal and external rays (or which can be approximated by landing points of internal rays and landing points of external rays) is a rotation set for the circle map $f_q:\J(f_q)\to \J(f_q)$ and $\bt$ is its rotation number. The rotation number $\bt$ only depends on the argument $\bth$ of the parameter ray $\pR_\bth\subset \HO$ containing $q^2$. The function $\bth\mapsto \bt$ is monotone and continuous of degree one, and the relation between $\bth$ and $\bt$ is the following. 
For each $\bt\in \RZ$, there is a unique \textbf{\textit{reduced rotation set}} $\bTheta\subset \RZ$ which is invariant under doubling and has rotation number $\bt$. This set carries a unique invariant probability measure $\mu_\bt$. 
\begin{itemize}
\item If $\bt=m/(2n+1)$ is rational with odd denominator or if $\bt$ is irrational, then $\bth\in \bTheta$ is the unique angle 
such that $\mu_\bt\bigl([0,\bth]\bigr) = \mu_\bt\bigl([\bth,1]\bigr)$. 
This angle is called the \textbf{\textit{balanced angle}} associated to $\bt$. 

\item If $\bt = m/2n$ is rational with even denominator, then there is a unique pair of angles $\bth^-$ and $\bth^+$ in $\bTheta$ 
such that $\mu_\bt\bigl([0,\bth^-]\bigr) = \mu_\bt\bigl([\bth^+,1]\bigr)=1/2$. 
The pair $(\bth^-,\bth^+)$ is called a \textbf{\textit{balanced pair of angles}} associated to $\bt$. Then, $\bth$ is any angle in $[\bth^-,\bth^+]$. 
\end{itemize}
We also present an algorithm which given $\bt$ returns the base two expansion of $\bth$ (or of $\bth^-$ in the case $\bt$ is rational with even denominator). 
\bigskip

In Section~\ref{s-vis} we study the set of points which are visible
 from zero and infinity in the case that $q^2\not\in\H_0\,$. 
 We first prove: \medskip

{\bf (Theorem~\ref{T-touch}.)}
{\it For any map $f_q$ in our family,
the closure $\overline\Bo$ of the basin of zero  and  the closure $\overline \Bi$ of the 
basin of infinity have vacuous intersection if and only if they are separated by a Herman ring.
Furthermore, the topological boundary of such a Herman ring is necessarily
contained in the disjoint  union $\partial\Bo\sqcup\partial  \Bi$
of basin boundaries.}
\medskip

This allows us to extend the concept of   dynamic rotation number
 to maps $f_q$ which do not represent elements of  $\HO$ provided 
 that they have locally connected Julia set. 
%
 We also deduce the following: \medskip

{\bf(Theorem \ref{T-jc}.)} 
{\it If $U$ is an attracting or parabolic basin of period two
or more for a map $f_q$, then the boundary $\partial U$ is a Jordan curve.
The same is true for any Fatou component which eventually maps to $U$.}

\bigskip

In Section~\ref{s-f}, for each rational $\bt\in\RZ$
with odd denominator we prove the existence of an associated 
fjord in the main  hyperbolic component $\H_0$.
\medskip

{\bf (Theorem \ref{T-fj}. Fjord Theorem.)}
{\it If $\bt$ is rational with odd denominator, and if
$\bth$ is the associated balanced angle,
then the parameter ray $\pR_\bth$ lands at the point
${\boldsymbol\infty}_\bt$ on the circle of points at infinity.}

\medskip

 As a consequence of this Fjord Theorem, we describe in 
Remark~\ref{R-frn} the
 concept of \textbf{\textit{formal rotation number}}
 for all maps $f_q$ with $q\neq 0$.
\bigskip

\begin{figure}[ht]
\centerline{\includegraphics[width=2.4in]{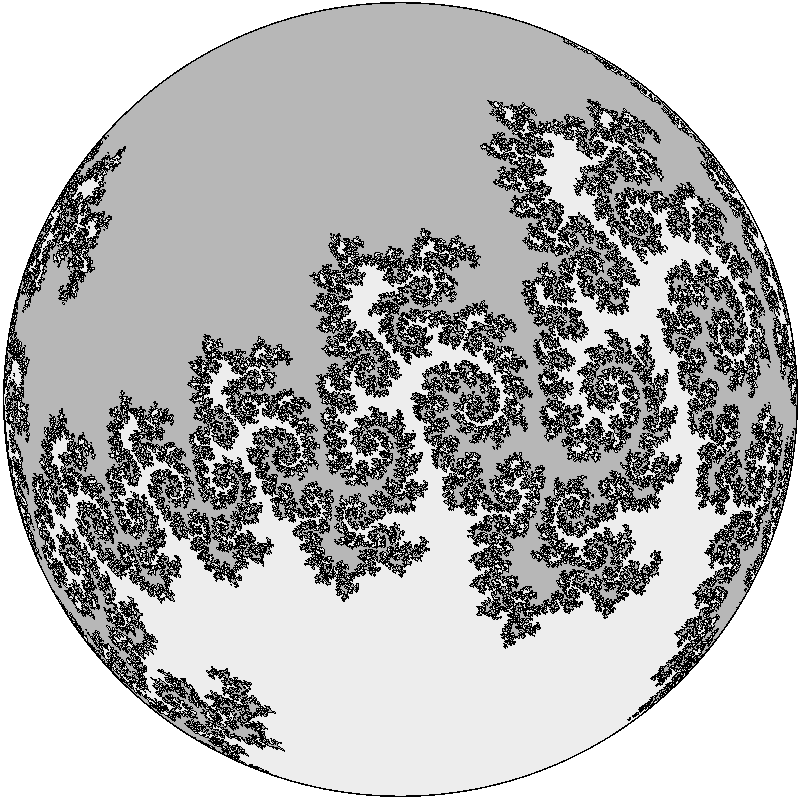}\qquad
\includegraphics[width=2.2in]{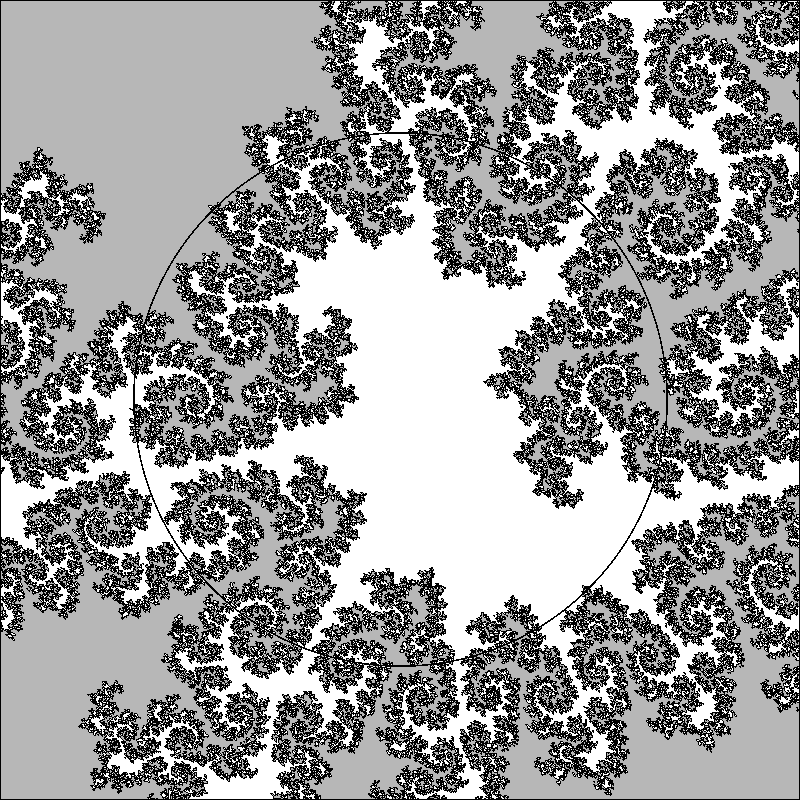}}
\smallskip
\caption{\it Julia set for a point in the
 $(5/11)$-fjord. This is a Jordan curve,
 separating the basins of zero and infinity.
The left side illustrates the Riemann sphere version,
and the right side shows the view in the $z$-plane, with the
unit circle drawn in.\label{F-fj}}
\medskip
\end{figure}

%
%
%

\section{Fatou Components and Hyperbolic
 Components}\label{s2}\bigskip

We first prove the following preliminary result.

\begin{theorem}\label{t1}  
Let $U$ be a Fatou component for some map $f=f_q$ belonging to our family
$(\ref{e0})$. If some forward image  $f^{\circ k}(U)$ is a
 Herman ring,\footnote{A priori, there could be a cycle of Herman rings
with any period. However, we will show in \cite{BBM2} that
any Herman ring in our family is fixed, with period one.}
then $U$ is an annulus; but in all other cases $U$ is simply-connected.
In particular,  the Julia set $\J(f)$ is connected if and only
 if there is no Herman ring.
\end{theorem}\medskip


\begin{proof}[\bf Proof] First consider the special case where $U$ 
is a periodic Fatou component.

\subsection*{\bf Case 1. A rotation domain.}
 If there are no critical points in this cycle of Fatou components,
then $U$ must be either
 a Herman ring (necessarily an annulus), 
 or a Siegel disk (necessarily simply connected). In either
case there is nothing to prove. Thus we are reduced to the attracting and
 parabolic cases.

\subsection*{\bf Case 2. An attracting domain of period
 $m\ge 2$.}
 Then $U$  contains an attracting periodic point $z_0$ of
 period  $m$.
 Let $\cO$ be the orbit of $z_0$, and
let $U_0$ be a small disk about $z_0$ which is chosen
 so that $f^{\circ m}(U_0)\subset U_0$,   and so that the boundary 
$\partial U_0$ does not contain any postcritical point. For each
 $j\ge 0$, let
$U_j$  be the connected component  of $f^{-j}(U_0)$ which contains
a point of $\cO$,  thus $f(U_j)=U_{j-1}$.
 Assuming inductively that $U_j$
is simply connected, we will prove that $U_{j+1}$ is also simply connected.
If $U_{j+1}$ contains no critical point, then it is an unbranched covering
of $U_j$, and hence maps diffeomorphically onto $U_j$. If there is
 only one
critical point, mapping to the critical value $v\in U_j$, then 
$U_{j+1}\ssm f^{-1}(v)$ is an unbranched covering space of $U_j\ssm\{v\}$,
 and hence is a punctured disk. It follows again that $U_{j+1}$ is simply
 connected. We must show that there cannot be two critical points in
$U_{j+1}$. Certainly it cannot contain either of
the two fixed critical points, since the period is two or more. If it
contained both free critical points, then it would intersect its antipodal
image, and hence be self-antipodal.  Similarly its image under $f_q$ would be
self-antipodal. Since two connected self-antipodal sets must intersect
each other, it would again follow that the period must be one, contradicting the
hypothesis.

Thus it follows inductively that each 
 $U_j$ is simply connected. Since the entire Fatou
component $U$ is the nested union of the $U_{jm}$,
 it is also simply connected.\end{proof}\smallskip

\subsection*{\bf Case 3. A parabolic domain of period $m$.} 
Note first that we must have $m\ge 2$. In fact it follows from
the holomorphic fixed point formula\footnote{Compare 
\cite[Corollary 12.7]{Milnor2}.}
that every non-linear rational map must have either a repelling fixed point
(and hence an antipodal pair of repelling points in our case) or a
fixed point of multiplier $+1$. But the latter case cannot occur in our
 family since the two free critical points are antipodal
and can never come together.

The argument is now almost the same as in the attracting
case, except that  in  place of a small
disk centered at the parabolic point, we take a small attracting petal
which intersects all orbits in its Fatou component.\smallskip
\smallskip

\begin{remark}
In Cases 2 and 3, we will see in Theorem~\ref{T-jc}
that the boundary $\partial U$ is always a Jordan curve.\smallskip
\end{remark}

\subsection*{\bf  Case 4. An attracting domain of period
one: the basin of zero  or infinity.} 
As noted in Case 3, 
 the two free fixed points are necessarily repelling, so the
only attracting fixed points are zero and infinity.

Let us concentrate on the basin of zero. 
Construct the sets $U_j$ as in Case 2.  If there are no other critical points
in this immediate basin, then
it follows inductively that the
$U_j$ are all simply connected. If there is another critical point $c$ in 
the  immediate basin, then there will be a smallest $j$
such that $U_{j+1}$ contains $c$. Consider
the  Riemann-Hurwitz formula
$$ \chi(U_{j+1})=d\cdot\chi(U_j) -n,$$
where $\chi(U_j)=1$,  $d$ is the degree of the map $U_{j+1}\to U_j$, 
and  $n=2$ is the number of critical points in $U_{j+1}$. If 
$d=3$, then $U_{j+1}$ is simply connected, and again 
 it follows inductively that the $U_j$ are all simply connected.
However, if $d=2$ so that $U_{j+1}$ is an annulus, then we would 
have a more complicated situation, as illustrated in Figure~\ref{r2.1}.
The two disjoint closed annuli $\overline U_{j+1}$ and 
 $\ant(\overline U_{j+1})$ would separate the Riemann sphere
it into two simply connected regions plus one annulus.
One of these simply connected regions $E$ would have to share a
boundary with $ U_{j+1}$, and the antipodal region $\ant(E)$ would share
a boundary with  $\ant( U_{j+1})$. The remaining annulus $A$
would then share a boundary with both  $ U_{j+1}$ and
 $\ant( U_{j+1})$, and hence would be self-antipodal.
We must prove that this case cannot occur.\smallskip

\begin{figure}[ht]
\centerline{\includegraphics[width=2.5in]{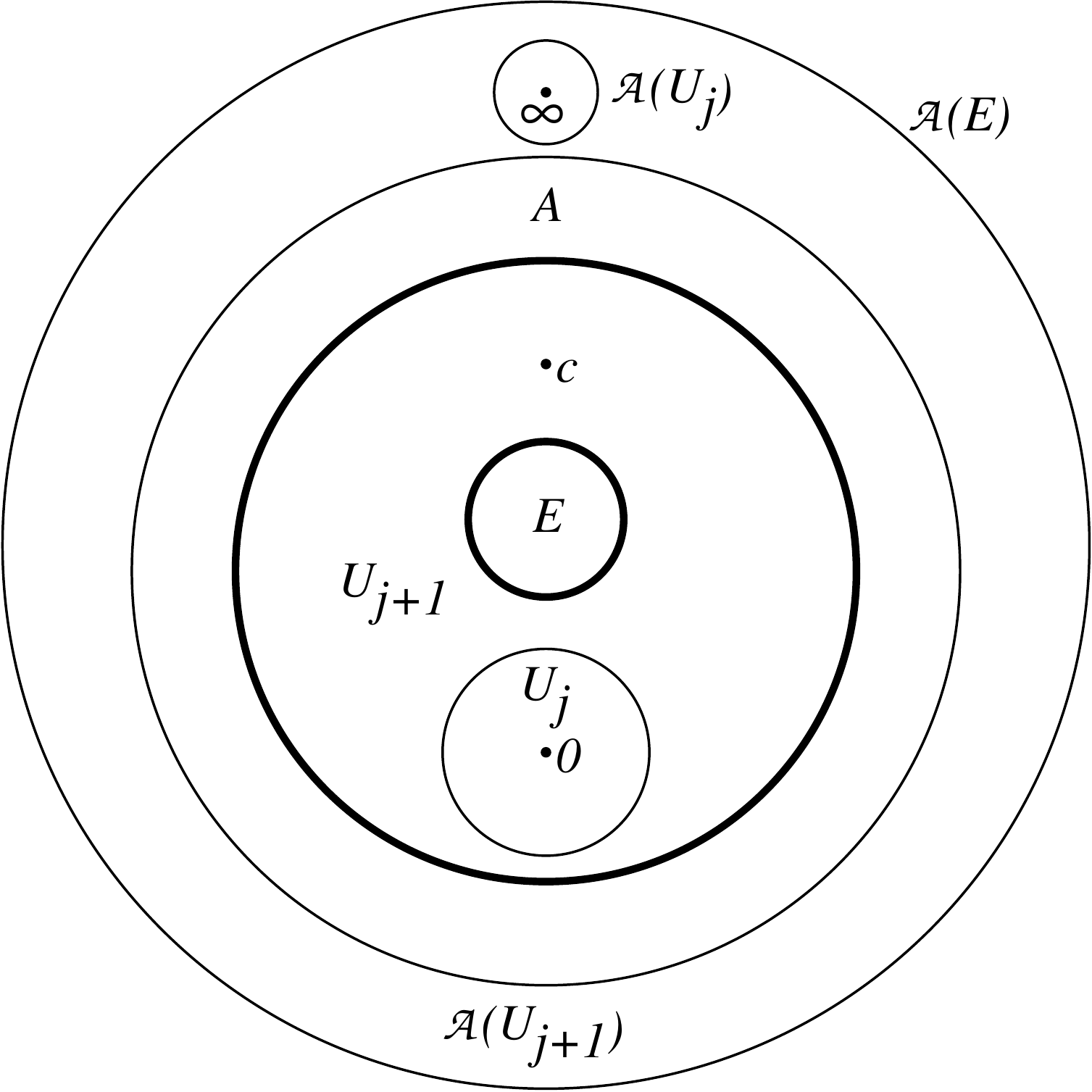}}
\caption{\label{r2.1} \it Three concentric annuli. $($See Case~$4$.$)$
 The innermost annulus
$U_{j+1}$, with emphasized boundary, contains the disk $U_j$, while the
 outermost annulus $\ant(U_{j+1})$ contains $\ant(U_j)$.
The annulus $A$ lies between these two. Finally, the innermost
region $E$ and the outermost region $\ant(E)$ are topological disks.}
\end{figure}

Each of the two boundary curves of $U_{j+1}$ (emphasized in Figure~\ref{r2.1})
must map homeomorphically
onto $\partial U_j$. Thus one of the two boundary curves of $A$ maps
homeomorphically onto $\partial U_j$, and the other must map 
homeomorphically onto $\ant(\partial U_j)$. Note that there are no 
critical points in $A$. The
 only possibility is that $A$  maps homeomorphically onto the larger
annulus $A'=\Chat\ssm\big(\overline U_j\cup\ant(\overline U_j)\big)$.
This behavior is perfectly possible for a branched covering, but is
 impossible for a holomorphic map, since the modulus of $A'$ must be
strictly larger than the modulus of $A$. This contradiction completes the proof
that each periodic Fatou component is either simply connected, or a Herman
ring.
\medskip

\subsection*{The Preperiodic Case.}
For any non-periodic
 Fatou component $U$, we know by Sullivan's nonwandering theorem
that some forward image $f^{\circ k}(U)$ is periodic. Assuming inductively
that $f(U)$ is an annulus or is simply connected, we must prove the same
for $U$. If $f(U)$ is simply connected, then it is easy to check that 
there can be at most one critical point in $U$, hence as before it follows
 that $U$ is simply connected. In the annulus case, the two free
critical points necessarily belong to the Julia set, and hence 
cannot be in $U$. Therefore $U$  is a finite unbranched
 covering space of $f(U)$, and hence is also an annulus. This completes the
proof of Theorem~\ref{t1}.
 \qed\medskip

The corresponding result in the  parameter plane is even sharper. 

\begin{lem}\label{L2} 
Every hyperbolic component, either in the $q$-plane or in the \break
\hbox{$q^2$-plane},
is simply connected, with a unique critically finite point $($called its
 \textbf{\textit{center}}$)$.
\end{lem}\medskip

\begin{proof}[\bf Proof] This is a direct application of results from 
\cite{Milnor4}. Theorem 9.3 of that paper states that every
 hyperbolic component in the moduli
space for rational maps of degree $d$ with marked fixed points is a
topological cell with a unique critically finite ``center'' point.
We can also apply this result to the sub moduli space consisting of 
 maps with one or more
marked critical fixed points (Remark 9.10 of the same paper).
The corresponding result for real forms of complex maps is also true. (See
\cite[Theorem 7.13 and Remark 9.6]{Milnor4}.)

These results apply in our case, since we are working with a real form
of the moduli space of cubic rational maps with two marked critical fixed
 points. The two remaining fixed points are then automatically marked:
 one in the upper half-plane and one in the lower half-plane. (Compare the
discussion below.) 
\end{proof}

\bigskip

\begin{figure}[ht]
\centerline{\includegraphics[width=3.5in]{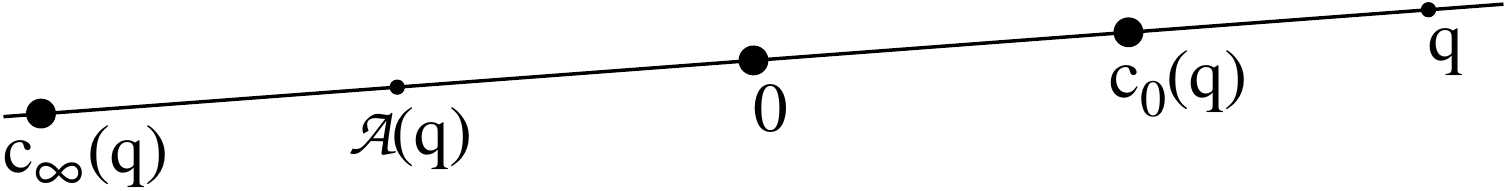}}\vspace{-.3cm}
\caption{\label{f4}\it}
\end{figure}

We will need a more precise description of the
four (not necessarily distinct) critical points of $f_q$.
 The critical fixed points zero and infinity are always present,
and there are also  two mutually antipodal
  \textbf{\textit{free critical points}}. For $q\ne 0$ we will denote the
two free critical points by
\begin{equation}\label{eq-c}
\cpo(q)=\left(\frac{a-3+\sqrt{P(a)}}{4a}\right)q\;,\quad{\rm and}\quad
\cpi(q)=\left(\frac{a-3-\sqrt{P(a)}}{4a}\right)q\;,
\end{equation}
where $a=|q^2|>0$ and $P(a)=9+10a+a^2>9$.
Thus $\cpo=\cpo(q)$  is a positive multiple of $q$, and
 $\cpi=\cpi(q)$  is a negative
multiple of $q$. A straightforward computation shows that 
$\big(a-3+\sqrt{P(a)}\big) < 4a$, so 
$\cpo(q)$  is between $0$ and $q$. As $q$ tends to zero, 
  $\cpo(q)$ tends to zero, and $\cpi(q)$ tends to infinity.
For $q\ne 0$, it follows easily
that the three finite critical points,  as well as the zero 
$q\in f_q^{-1}(0)$ and the pole $\ant(q)$, all lie on a  straight line,
 with $\cpo(q)$ between zero and $q$ and with $\ant(q)$
between $\cpi(q)$ and zero, as shown in Figure~\ref{f4}.

Setting $q=x+i y$, a similar computation shows that the two
\textbf{\textit{free fixed points}} are given by 
\begin{equation}\label{eq-fp} 
\fixq_±(q)= i\Big(y ±\sqrt{y^2+1}\Big)\qquad{\rm where}
\qquad y=\Im(q). 
\end{equation}
Thus $\fixq_+(q)$ always lies on the positive imaginary axis, while its antipode
$\fixq_-(q)$ lies on the negative imaginary axis. As noted
 in the proof of Theorem
\ref{t1}, both of these free fixed points must be strictly repelling.
One of the two is always on the boundary of
 the basin of zero, and the other is always on the boundary of the
basin of infinity.  More precisely, if $y>0$ then
$\fixq_+(q)$ lies on the boundary of the basin of infinity
and $\fixq_-(q)$ lies on the boundary of the basin of zero;
while if $y<0$ then the opposite is true. (Of course  if $q^2\in\HO$, 
then both free fixed points lie on the common self-antipodal
boundary. In particular,  this happens when $y=0$.)
\smallskip

\begin{figure}[ht]
\begin{subfigure}
\centering
\includegraphics[width=2.8in]{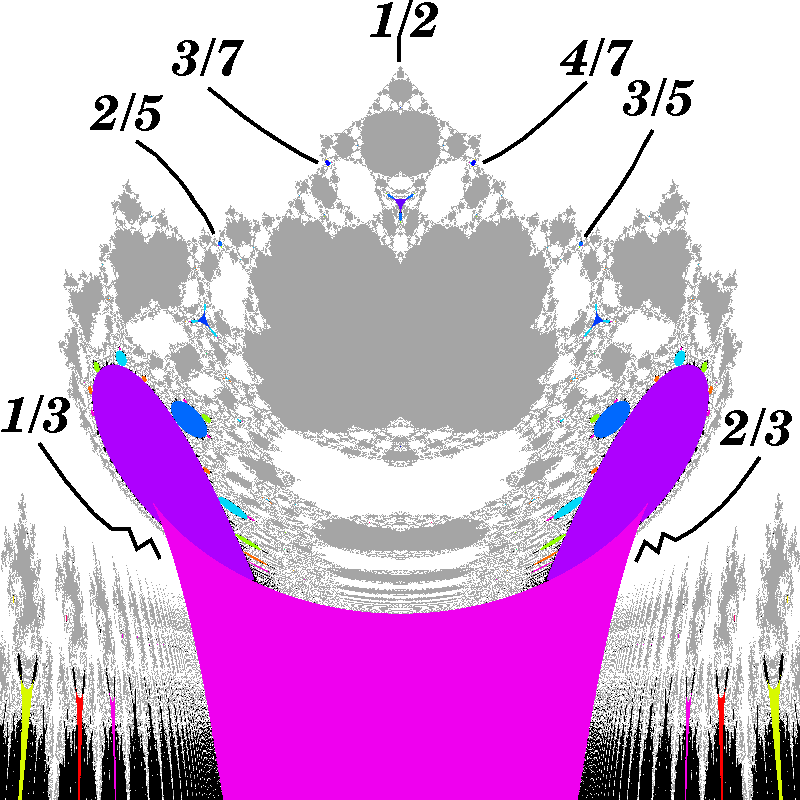}
\smallskip
\caption{\it Crown at the head of the lower $(1/2)$-tongue
in the $q$-plane, showing the angles of several parameter rays.
\label{F-halft}}
\end{subfigure}
\bigskip
\begin{subfigure}
\centering
\includegraphics[width=2.8in]{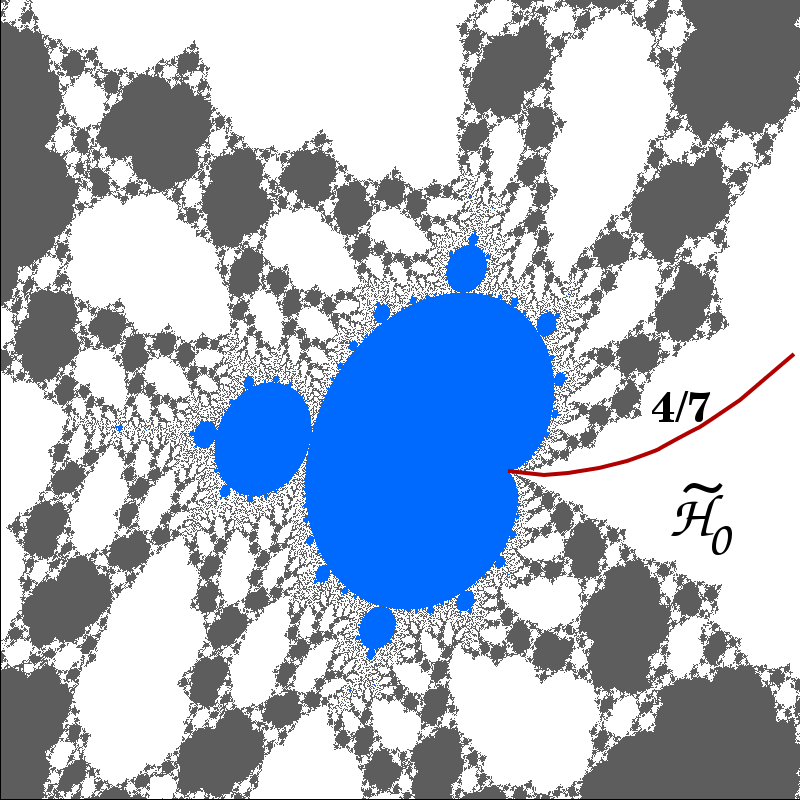}
\smallskip
\caption{ \label{F-4/7par} \it Detail
to the upper  right of Figure~$\ref{F-halft}$ 
centered at $q = .1476 - 1.927~{\rm i}$,
showing a small Mandelbrot set. {\rm (Compare Figure \ref{F-4/7jul}.)}
An internal ray of angle $4/7$ {\rm (drawn in by hand)}
lands at the
root point of this Mandelbrot set. Here $\HOtilde$
is the white region to the  right---   
The other white or grey regions are capture components.}
\end{subfigure}
\end{figure}

\begin{figure}[ht]
\centerline{\includegraphics[height=3.5in]{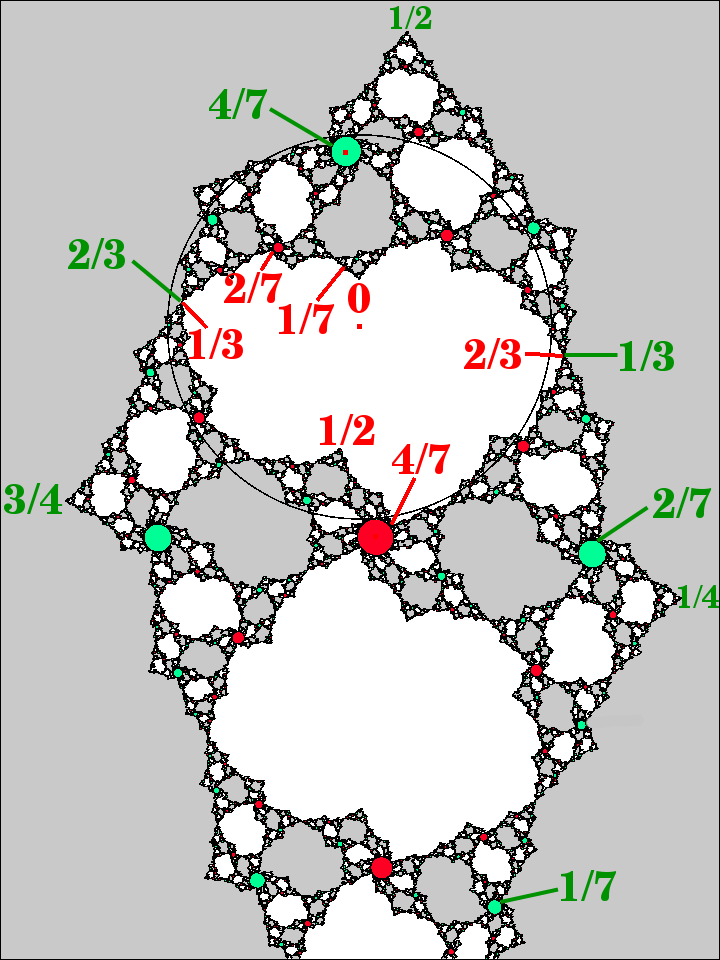}}
\caption{\label{F-4/7jul}  \it Julia set for the center of the small
Mandelbrot set 
of Figure$\ref{F-4/7par}$, with some internal and external
 angles labeled. {\rm(Compare Remark \ref{R-unbal}.)}
The unit circle has been drawn to fix the scale.}
\end{figure}

We can now give a preliminary description of the possible hyperbolic
 components.

\begin{lem}\label{Lhyps} Every hyperbolic component in the $q^2$-plane
 belongs to one of the following four types:\smallskip

$\bullet$ The \textbf{\textit{central hyperbolic component}} $\HO$ 
{\em (the white region in Figure~\ref{f2})}.
This is the unique component for which the Julia set is a Jordan curve,
which necessarily separates the 
 basins of zero and infinity.\footnote{Note that any $f_q$
having a Jordan curve Julia set is necessarily hyperbolic.
There cannot be a
parabolic point since neither complementary component can be a parabolic basin;
and no Jordan curve Julia set can contain a critical point.}
 In this case, the critical
point $\cpo(q)$ necessarily lies in the basin of zero, and $\cpi(q)$
lies in the basin of infinity.
{\em Compare Figures~\ref{F-fj}, \ref{F-rab}~({\rm left}),
{\rm and} \ref{fV}.} The corresponding white region in the $q$-plane
 {\em (Figure \ref{f1})} is a branched 2-fold  
covering space, which will be denoted by $\HOtilde$.
\smallskip

$\bullet$ \textbf{\textit{Mandelbrot Type}}. Here there are two mutually
antipodal attracting orbits of period two or more,
 with one free critical point in the immediate  basin of each one.
{\em For an example in the parameter plane, see Figure~\ref{F-4/7par}
and for the corresponding Julia set see  Figure~\ref{F-4/7jul}.)}
\smallskip

$\bullet$ \textbf{\textit{Tricorn Type}}. Here there is one self 
antipodal attracting orbit, necessarily of even period. The most conspicuous 
examples are the tongues which stretch out to infinity. 
{\em (These will be studied in \cite{BBM2}.)}
There are also small bounded tricorns. {\em (See Figures~\ref{tric} and
\ref{F-trijul3c}.)}
\smallskip

\begin{figure}[ht]

\includegraphics[width=2.7in]{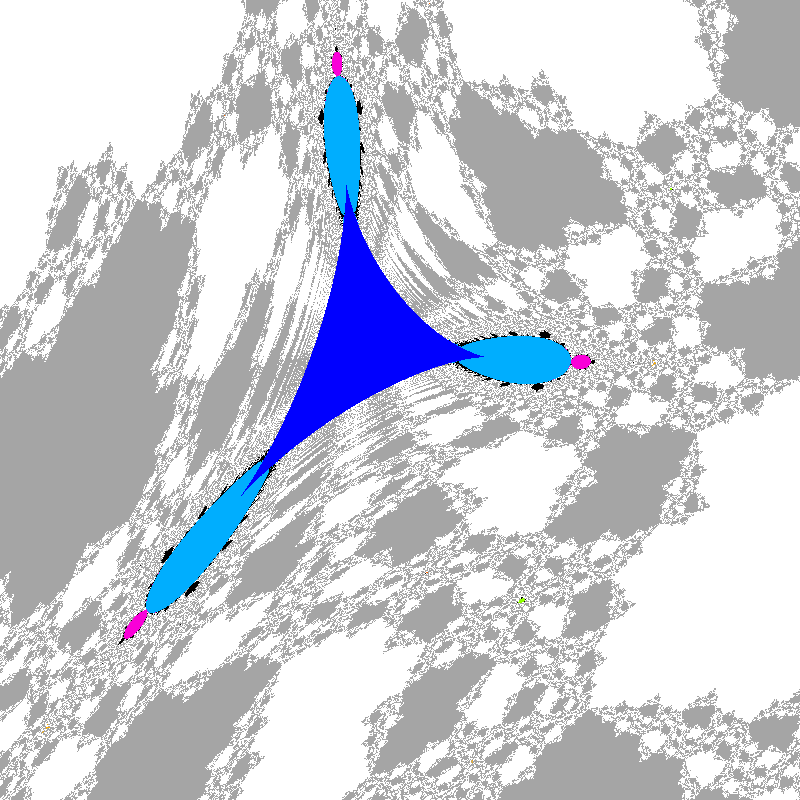}
\caption{\it\label{tric} Magnified picture of a tiny speck in
 the upper right of
Figure$\ref{F-halft}$, showing a period six tricorn centered at
$q\approx 0.394-2.24i$. The white and grey regions are capture components.
{\rm (For a corresponding Julia set, see Figure \ref{F-trijul3c}.)}
}

\bigskip

\centerline{
\includegraphics[height=3.5in]{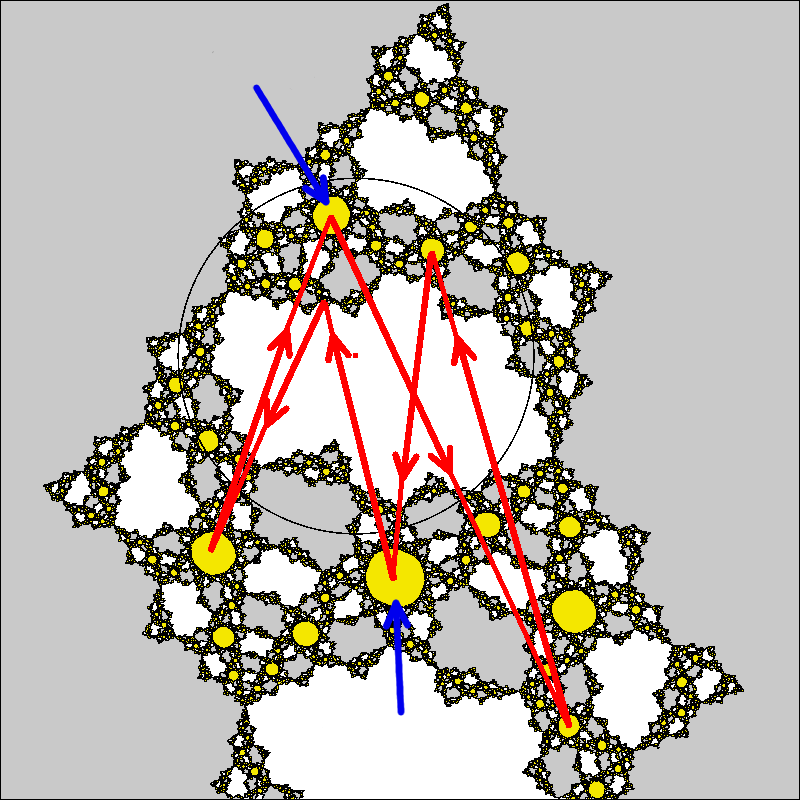}
\begin{picture}(0,0)
\put(-150,20){$z_0=z_6=c_0$}
\put(-190,230){$z_3=c_\infty$}
\end{picture}
}
\caption{\label{F-trijul3c}  \it Julia set corresponding to 
the small tricorn of Figure \ref{tric}.}
\end{figure}

$\bullet$ \textbf{\textit{Capture Type}}. Here the free critical points
 are not in the immediate basin of zero or infinity, but some forward image 
of each free critical point 
belongs  in one basin or the other. {\em (These regions are either dark grey
or white in Figures~\ref{f1}, \ref{f2}, \ref{f3}, \ref{F-halft},
\ref{F-4/7par},  \ref{tric}, according as the orbit of $\cpo(q)$ converges to $\infty$
or $0$.)}
\end{lem}
\smallskip

\begin{proof}[\bf Proof of Lemma~\ref{Lhyps}]
 This is mostly a straightforward exercise---Since there are only two
free critical points, whose orbits are necessarily antipodal to each other,
there are not many possibilities. As noted in the proof of Theorem~\ref{t1},
 the two fixed points in
$\C\ssm\{0\}$ must be strictly repelling. However, the uniqueness
 of $\HO$ requires some proof. In particular,
since $\cpo(q)$ can be arbitrarily large, one might guess that there
 could be another hyperbolic component for which 
 $\cpo(q)$ lies in the immediate basin of infinity.
We must show that this is impossible.

 According to Lemma~\ref{L2}, for the center of 
such a hyperbolic component, the point
$\cpo(q)$ would have to be precisely equal to infinity. But as $\cpo(q)$
approaches infinity, according to Equation~(\ref{eq-c}) the parameter $q$
must also tend to infinity. Since\break $f(q)=0$ and $f(\infty)=\infty$,
there cannot be any such well behaved limit.\smallskip


Examples illustrating all four cases, are easily provided. (Compare the
 figures cited above.)
 \end{proof}
 
\bigskip

\section{The Central Hyperbolic Component $\HO$.}\label{s-chc}
\bigskip

As noted in Lemma~\ref{Lhyps}, the central hyperbolic component in the
 $q^2$-plane  consists of all $q^2$ such that the Julia set of 
$f_q$ is a  Jordan curve.
 The critical point $\cpo(q)$ necessarily lies
in the basin of zero, and $\cpi(q)$ in the basin of infinity.

\begin{lem}\label{L-can-dif} 
The open set $\HO$  in the $q^2$-plane
is canonically diffeomorphic to the
open unit disk $\D$ by the map 
$$ \bot:\HO\stackrel{\cong}{\longrightarrow}\D $$ 
which satisfies 
$\bot(0)=0$ and carries each $\,q^2\ne 0\,$ to the
 B\"ottcher coordinate\footnote{A priori, the B\"ottcher coordinate 
$\botq(z)$ is defined only for $z$
in an open set containing $\cpo(q)$
 on its boundary. However, it has a well defined limiting value at
 this boundary point. (Compare the discussion below Remark~\ref{D-vis}.)}
 $\botq\big(\cpo(q)\big)$ of the marked critical point. 
\end{lem}\medskip

\begin{remark}\label{R-nc} It follows that
 there is a canonical way of putting a conformal structure on the open set
 $\HO$. However, this conformal structure  
becomes very wild near
the boundary $\partial\HO$. It should be emphasized that the entire
$q^2$-plane does {\it not} have any natural conformal structure.
 (The complex numbers $q$ and $q^2$ are not holomorphic parameters,
since the value $f_q(z)$ does not depend holomorphically on $q$.)
\end{remark}
\medskip



\begin{proof}[\bf Proof of Lemma \ref{L-can-dif}]
The proof will depend on results from 
\cite{Milnor4}. (See in particular Theorem 9.3 and Remark 9.10
in that paper.) 
In order to apply these results, we must work with polynomials
or rational maps with marked fixed points.

 First consider cubic polynomials of the form
$$p_a(z)= z^3+az^2$$
with critical fixed points at zero and infinity. Let
 $\HOtildepoly$ be the principal  hyperbolic component
consisting of values of $a$ such that both critical points of $p_a$
 lie in the immediate basin of zero. 
This set is simply-connected and contains no parabolic points, hence each
of the two free
fixed points can be defined as a holomorphic function $\fixp_±(a)$ for
$a$ in $\HOtildepoly$, with $\fixp_±(0)=± 1$. The
precise formula is 
\[\fixp_±(a)=\frac{-a±\sqrt{a^2+4}}{2}.\]

Similarly, let $\HOtilderat$ denote the principal hyperbolic
component in the space of cubic rational maps of the form
$$ f_{b,c}(z)=z^2\frac{z+b}{c z+1} $$
with critical fixed points at zero and infinity. This is again simply-connected.
Therefore the two free fixed points can be expressed as holomorphic
functions $\fixf_±(b,c)$ with $\fixf_±(0,0)=± 1$.

The basic result in \cite{Milnor4}
 implies that $\HOtilderat$ is
naturally biholomomorphic to the Cartesian product
 $\HOtildepoly\!\times\HOtildepoly$,
 where a given pair 
$(p_{a_1},p_{a_2})\in \HOtildepoly\!\times\HOtildepoly$
 corresponds to the map
 $f_{b,c}\in\HOtilderat$ if and only if:

{\bf(1)} the map $f_{b,c}$ restricted to its basin of zero
is holomorphically conjugate to the map $p_{a_1}$ restricted to its basin of zero,
with the boundary fixed point $\fixf_+(b,c)$ corresponding to
$\fixp_+(a_1)$; and

{\bf(2)}  $f_{b,c}$ restricted to its basin of infinity is
holomorphically conjugate to $p_{a_2}$ restricted to its basin of zero,
with the boundary fixed point  $\fixf_+(b,c)$ corresponding to
$\fixp_+(a_2)$.


Now let $\HOpoly$ be the moduli space obtained from
 $\HOtildepoly$ by identifying each $p_a$ with $p_{-a}$.
According to \cite[Lemma 3.6]{Milnor3}, 
the correspondence which maps $p_a$ to the
B\"ottcher coordinate of its free critical point gives rise to a conformal
isomorphism between $\HOpoly$ and the open unit disk.

Note that our space $\HOtilde$ is embedded real analytically
as the subspace of $\HOtilderat$ consisting of maps $f_q$
which commute with the antipodal map. Since $f_q$ restricted to the basin
of infinity is anti-conformally conjugate to $f_q$ restricted to the basin
of zero, it follows easily that each 
$f_q\in\HOtilde\subset\HOtilderat$ must correspond to a pair
of the form $(p_a,p_{±\bar a})\in\HOtildepoly
\times\HOtildepoly$, for some choice of sign.
In fact the correct sign is $(p_a,p_{-\bar a})$. (The other choice
of sign $(p_a,p_{\bar a})$
 would yield maps invariant under an antiholomorphic involution with an
entire circle of fixed points, conjugate to the involution
 $z\leftrightarrow\overline z$.)

It follows that the composition of the real analytic embedding
$$ \HOtilde \hookrightarrow \HOtilderat \cong\HOtildepoly\times\HOtildepoly $$
with the holomorphic projection to the first factor
 $\HOtildepoly$
is a real analytic diffeomorphism $\HOtilde\stackrel{\cong}
{\longrightarrow}\HOtildepoly$ (but is {\it not} a conformal
isomorphism). The corresponding assertion for $\HO\stackrel{\cong}
{\longrightarrow}\HOpoly\cong\D$ follows easily.
This completes the proof of Lemma~\ref{L-can-dif}.
\end{proof}

\bigskip

\begin{remark}\label{R-mate} (Compare  \cite{Milnor6}, \cite{ShTL} and
 \cite{MP}.) Let $f$ and $g$
be (not necessarily distinct) monic cubic polynomials.
Each one has a unique continuous extension to the circled plane,
such that each point on the circle at infinity is mapped to itself by the angle
tripling map $\binf_\bt\mapsto\binf_{3t}$.
Let $\CC_f\uplus\CC_g$ be the topological sphere obtained from the disjoint
union $\CC_f\sqcup\CC_g$
of two copies of the circled plane by gluing the
two boundary circles together, identifying each point $\binf_\bt\in\partial\CC_f$
with the point $\binf_{-\bt}\in\partial\CC_g$. Then the extension of
 $f$ to $\CC_f$
and the extension of $g$ to $\CC_g$ agree on the common boundary, yielding
a continuous map $f\uplus g$
from the sphere $\CC_f\uplus\CC_g$ to itself. 

Now assume that the Julia sets of $f$ and $g$ are locally connected. Let\break
 $K(f)\mat K(g)$ be the quotient space obtained from
 the disjoint union
 $K(F)\sqcup K(g)$ by identifying the landing point of the $t$-ray on
 the Julia set $\partial K(f)$ with the landing point of the $-t$-ray on
 $\partial K(g)$ for each $t$. Equivalently, $K(F)\mat K(g)$ can be
 described as the quotient space of the sphere $\CC_f\uplus\CC_g$ in which
the closure of every dynamic ray in $\CC_f\ssm K(f)$ or in $\CC_g\ssm K(g)$
is collapsed to a point.

In many cases (but not always),
 this quotient space is a topological 2-sphere. In these cases, the
\textbf{\textit{topological mating}} $f\mat g$ is defined to be
 the induced continuous
map from this quotient 2-sphere to itself. 
Furthermore in many cases (but not always)
the quotient sphere has a unique conformal structure satisfying the
 following three conditions:

\begin{quote}
{\bf(1)} The map $f\mat g$ is holomorphic (and hence rational) with respect
to this structure.

\noindent{\bf(2)} This structure  is compatible with the given conformal
structures on the interiors of $K(f)$ and $K(g)$ whenever these interiors
are non-empty.

\noindent{\bf(3)} The natural map from $\CC_f\uplus
\CC_g$ to $K(f)\mat K(g)$ has degree one.\footnote{This last condition
is needed only when $K(f)$ and $K(g)$ are dendrites.}
\end{quote}

\noindent When these conditions are satisfied,
 the resulting rational map is described as the \textbf{\textit{
conformal mating}} of $f$ and $g$.\medskip

In order to study matings which commute with the antipodal map of the Riemann
sphere, the first step is define an appropriate {\em antipodal map} $\bA$
from the intermediate sphere $\CC_f\uplus\CC_g$ to itself. In fact,
 writing the two maps as $z\mapsto f(z)$ and $w\mapsto g(w)$, define the
 involution $\bA:\CC_f\longleftrightarrow\CC_g$ by
$$ \bA: z\longleftrightarrow w= -\overline z.$$ 
Evidently $\bA$ is well defined and antiholomorphic, with no fixed
 points  outside of the common circle at infinity.
 On the circle at infinity,
it sends $z=\binf_\bt$ to $w=-\binf_{-\bt}$. Since each $z=\binf_\bt$ is
identified with  $w=\binf_{-\bt}$,
 this corresponds to a 180$^\circ$ rotation of the circle, with no fixed points.

Now suppose that the intermediate map $f\uplus g$ from $\CC_f\uplus\CC_g$
to itself commutes with $\bA$, 
 $$\bA\circ(f\uplus g)
=(f\uplus g)\circ\bA.$$ 
It is easy to check that this is equivalent to the requirement that
$$ g(-\overline z)= -\overline{ f(z)},$$
or that $g(z)= -\overline {f(-\overline z)}.$
In the special case $f(z)= z^3+az^2$, it corresponds to the condition that
$$ g(z)=z^3-\overline az^2.$$

Next suppose that the topological mating $f\mat g$ exists.
Then it is not hard to see that $\bA$ gives rise to a corresponding
map $\bA'$ of topological degree $-1$ from the sphere $K(f)\uplus K(g)$
to itself. If $\bA'$ had a fixed point $\bf x$, then there would exist
a finite chain $C$ of dynamic rays in $\CC_f\uplus\CC_g$ leading from
a representative point for $\bf x$ to its antipode.
This is impossible, since it would imply that the union $C\cup\bA(C)$
is a closed loop separating the sphere $\CC_f\uplus\CC_g$,
which would imply that the quotient $K(f)\uplus K(g)$ is not a topological
sphere, contrary to hypothesis.\medskip

Finally, if the conformal mating exists, we will show that the fixed point free
involution $\bA'$ is anti-holomorphic. To see this, note that the invariant
 conformal structure $\Sigma$ can be transported by $\bA'$ to yield
a new complex structure $ \bA'_*(\Sigma)$ which is still invariant, but
has the opposite orientation. Since we have assumed that 
$\Sigma$ is the only conformal
structure satisfying the conditions {\bf (1), (2), (3),}
it follows that  $ \bA'_*(\Sigma)$ must coincide
with the complex conjugate
 conformal structure $\overline\Sigma$. This is
just another way of saying that $\bA'$ is antiholomorphic.

In fact, all maps $f_q$ in $\HOtilde$ can be obtained as matings. Furthermore, 
Sharland  \cite{Sha} has shown that many maps outside of $\HOtilde$
can also be described as matings.\footnote{For example,
the landing points described in Theorem~\ref{t-inray}(a)
are all matings. (Compare Figure~\ref{F-4/7par}  in the $q$-plane,
as well as Figure~\ref{F-4/7jul} for an associated picture in the 
dynamic plane.)}
\end{remark}
\medskip
\medskip

\begin{definition}
Using this canonical dynamically defined diffeomorphism 
$\bot$, 
 each {\it internal angle}  $\bth\in\RZ$ determines a  
\textbf{\textit{parameter ray}}\footnote{Note that every parameter ray
$\pR_\bth$ in the $s=q^2$ plane corresponds to two different rays in the
$q$-plane: one in the upper half-plane and one in the lower half-plane,
provided that $q\ne 0$.}
$$ \pR_\bth=\bot^{-1}\bigl(\{re^{2\pi i\bth}; 0< r<1\}\bigr)\subset \HO\ssm\{0\}.$$
\end{definition}

\noindent
(These are {\it stretching rays} in the sense of Branner-Hubbard, see
\cite{B}, \cite{BH} and \cite{T}. In particular, any two maps 
along a common ray are quasiconformally conjugate to each other.)

\begin{theorem}\label{t-inray}
If the angle $\bth\in\Q/\Z$ is periodic under doubling, then the parameter 
 ray  $\pR_\bth$  either:
\medskip

\begin{itemize}
\item[{\bf(a)}] lands at a parabolic point on the boundary $\partial\HO$,
\smallskip

\item[{\bf(b)}] accumulates 
 on a curve of parabolic boundary points, or \smallskip

\item[{\bf(c)}] lands at a point $\binf_\bt$ 
on the circle at infinity, where $\bt$ can be described as the rotation
number of $\bth$ under angle doubling.\footnote{Compare the discussion
of rotation numbers in the Appendix.}
\end{itemize}
\end{theorem}\medskip

Here Cases {\bf(a)} and {\bf(c)} can certainly occur. (See
Figures~\ref{F-4/7par} and \ref{F-raydisk}.) For Case {\bf(b)},
numerical computations by Hiroyuki Inou suggest strongly that rays can 
 accumulate on a curve
of parabolic points without actually landing (indicated schematically in
Figure~\ref{F-halft}).  This will be studied further in \cite{BBM2}. For
  the analogous case of multicorns see \cite{IM}. \medskip

\begin{remark}\label{D-vis}
In case {\bf(b)}, we will show in \cite{BBM2} 
 that the orbit is bounded
away from the circle at infinity. If the angle
  $\bth$ is rational but not periodic,
 then we will show that 
$\pR_\bth$ lands on a critically finite parameter value. \medskip
\end{remark}
\medskip


The proof of Theorem~\ref{t-inray}  will make use of the following definitions.
\smallskip

Let $\Ba$ be the immediate basin of zero for 
the map $f_q$.
In the case $q^2\not\in\HO$,
note that the B\"ottcher coordinate is well defined as a conformal isomorphism
$\botq:\Ba\stackrel{\cong}{\longrightarrow}\D$
satisfying $$\botq \big(f_q(z)\big)=\bigl(\botq (z)\bigr)^2.$$
Radial lines in the open unit disk
 $\D$ then correspond to dynamic 
internal rays $\dR_{q,\theta}:[0,1)\to \Ba$, where $\theta$ is the
internal angle in $\D$. As in the polynomial case every periodic ray
lands on a periodic point.

However, for the moment we are rather concerned with the case 
$q^2\in \HO\ssm \{0\}$. In this case,
$\botq$  is defined only  on a open subset of $\Ba$.
More precisely, $\botq$ is certainly well defined as a conformal isomorphism
from a neighborhood of zero in  $\Ba$ to a neighborhood of zero 
in  $\D$.
Furthermore the correspondence $z\mapsto|\botq(z)|^2$ extends to a well
defined smooth map from $\Ba$ onto $[0,1)$. 
\textbf{\textit{The  dynamic internal rays}} can then be defined
as the orthogonal trajectories of the equipotential curves
 $|\botq(z)|^2={\rm constant}$.  The difficulty is
that this map $z\mapsto|\botq(z)|^2$  has
critical points at every critical or precritical point of $f_q$; and
that countably
 many of the internal rays will bifurcate off these critical or precritical
 points. However, every internal ray outside of this countable set is
 well defined as an injective  map  
  $\dR_{q,\theta}:[0,1)\to \Ba$,
 and has a well defined landing point in $\partial \Ba=\J(f_q)$.

The notation $\Bv$ will be used for the open subset of
 $\Ba$
consisting of all points which are \textbf{\textit{directly visible}}
 from the 
origin in the sense that there is a smooth internal ray segment from
 the origin which passes through the point.
 The topological closure $\overline \Bv$
will be called the set of \textbf{\textit{visible points}}.

For $q^2\in \HO\ssm\{0\}$, the map
 $\botq$  sends 
 the set of directly visible points  $\Bv$ 
univalently 
onto a proper open subset of $\D$  which is obtained from the open disk
by removing countably many radial slits near the boundary. (See 
Figure~\ref{f4a}.) These correspond to the countably many
 internal rays from the origin in $\Ba$ which bifurcate off 
critical or precritical points. 
 

\begin{figure}[ht]
\centerline{\psfig{figure=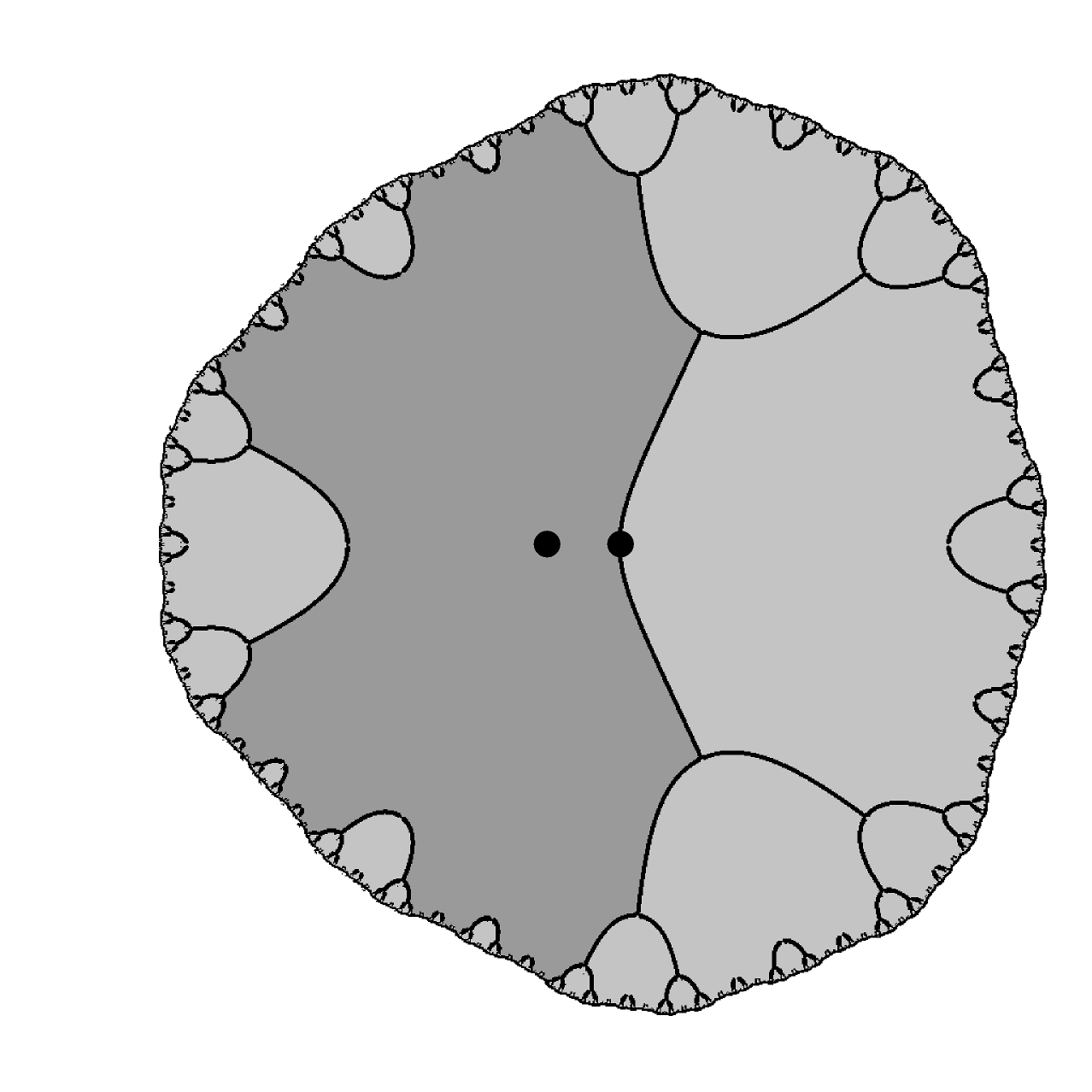,height=2.5in}\hspace{-.3cm}
\psfig{figure=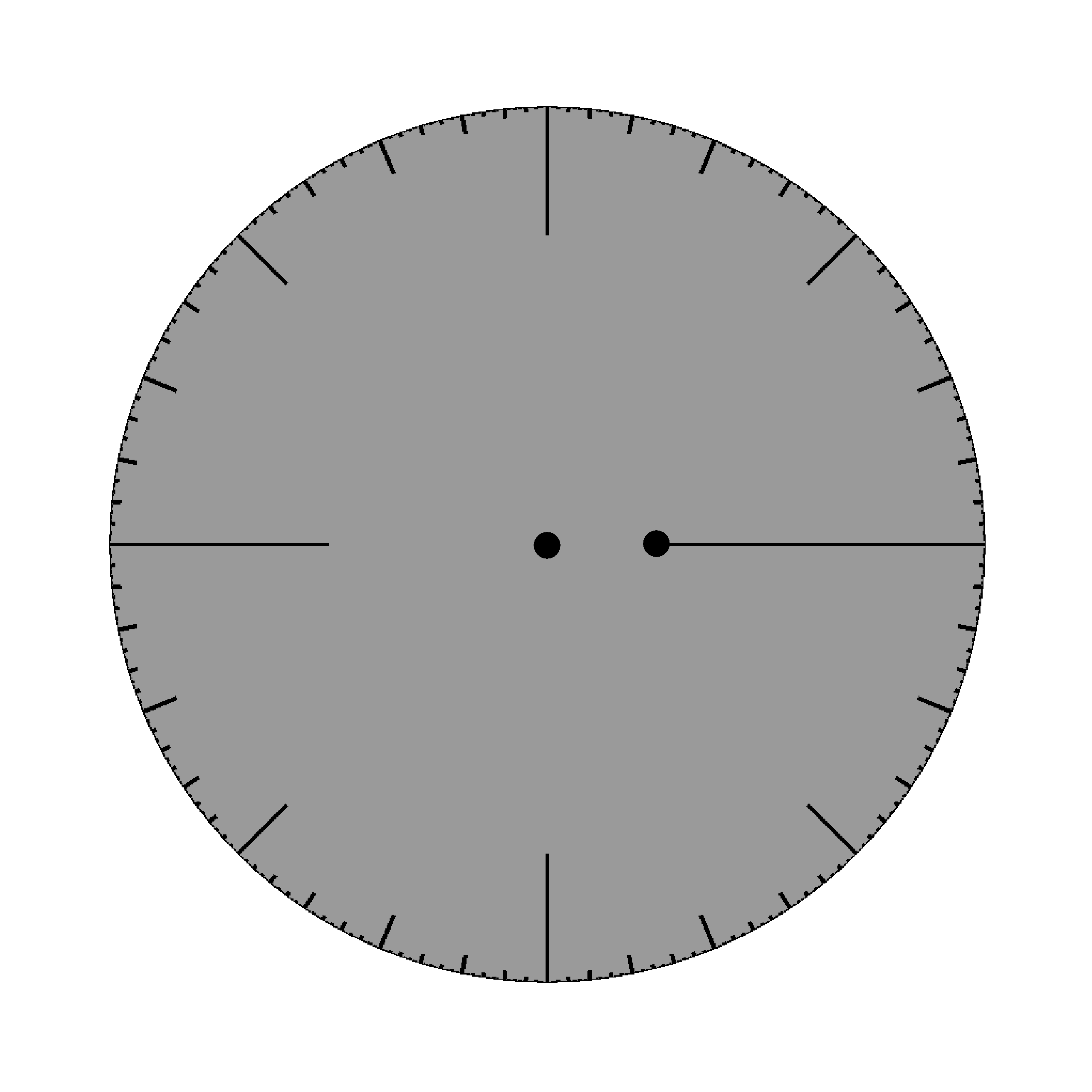,height=2.5in}}\vspace{-.3cm}
\caption{\label{f4a} \it On the left: the basin $\Ba$. In this
 example, $q>0$
so that the critical point  $\cpo(q)$  is on the positive real axis.
Points which are visible from the origin are shown in dark
gray, and the two critical points in $\Ba$ are shown as black 
dots. On the 
right: a corresponding picture in the disk of B\"ottcher coordinates. Note that
 the light grey regions on the left correspond only to slits on the right.}
\end{figure}
\medskip

\begin{proof}[\bf Proof of Theorem~\ref{t-inray}] Suppose that $\bth$ has period
$n\ge 1$
 under doubling. Let $\{q^2_k\}$ be a sequence of points on the parameter
ray $\pR_\bth\subset\HO$, converging to a limit $q^2\in\partial\HO$
in the finite $q^2$-plane. Choose representatives
 $q_k=±\sqrt{q^2_k}$, 
converging to a limit $q$. Since $\bth$ is periodic
and $q^2 \not\in \HO$,
the internal  dynamic ray  $\dR_{q,\bth}$
lands on a periodic point $z_q=f_q^{\circ n}(z_q)$. If this landing point
were repelling, then every nearby ray $\dR_{q_k, \bth}$ would
have to land on a nearby repelling point of $f_{q_k}$. But this is impossible, since each
$\dR_{q_k, \bth}$ must bifurcate. Thus $z_q$ must be a parabolic
point.

Since the ray $\dR_{q, \bth}$ landing on $z_q$ has period $n$, we
say that this  point $z_q$ has \textbf{\textit{ray period}} $n$.
Equivalently, each connected component
 of the immediate parabolic basin of $z_q$ has exact period $n$.
If the period of $z_q$ under $f_q$ is $p$, 
note that the multiplier of $f_q^{\circ p}$
at $z_q$ must be a primitive $(n/p)$-th root of unity.
(Compare \cite[Lemma 2.4]{GM}.)
The set of all parabolic points of ray period $n$ forms a real algebraic
subset of the $s$-plane. Hence each connected component is either a
one-dimensional real algebraic
curve or an isolated point. Since the set of all accumulation
points of the parameter ray $\pR_\bth$ is necessarily connected, the only
possibilities are that it lands at a parabolic point or accumulates on 
a possibly unbounded
connected subset of a curve of parabolic points. (In fact,
we will show in \cite{BBM2}
 that this accumulation set is always bounded.)
\smallskip

Now suppose that the sequence $\{q^2_k\}$ converges to a point 
on the circle at infinity. It will be convenient to introduce a new
coordinate $u=z/q$ in place of $z$ for the dynamic plane. The map
$f_q$ will then be replaced by the map
$$ u\mapsto \frac{f_q(uq)}{q}=\frac{q^2u^2(1-u)}{1+|q^2|u}.$$
Dividing numerator and denominator by $|q^2|u$, and
setting $\lambda=q^2/|q^2|$ and $b=1/|q^2|$, this takes the form
$$ u\mapsto~g_{\lambda,b}(u)=\frac{\lambda u(1-u)}{1+b/u},$$
with $|\lambda|=1$ and $b>0$.
The virtue of this new form is that the sequence of maps $g_{\lambda_k,b_k}$
associated with $\{q^2_k\}$
 converges to a well defined limit
$P_\lambda(u) = \lambda u(1-u)$ as $b\to 0$, or in other words as
$|q^2|\to\infty$. (Compare Figure~\ref{F-rab}.) 
Furthermore, the convergence is locally uniform for
$u\in\C\ssm\{0\}$. Note that $u=0$ is a fixed point of multiplier
 $\lambda$ for $P_\lambda$.  A quadratic polynomial can have at most
one non-repelling periodic orbit. (This is an easy case of the
Fatou-Shishikura inequality.) Hence
it follows that all other periodic orbits of $P_\lambda$ 
 are strictly repelling.

\begin{figure}[ht]
\centerline{\psfig{figure=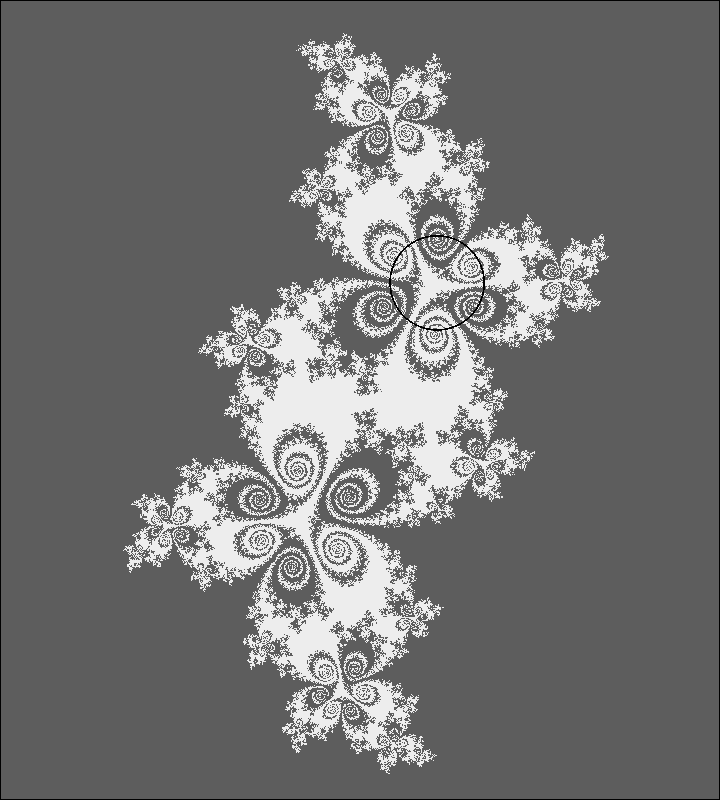,height=2.5in}\quad
\psfig{figure=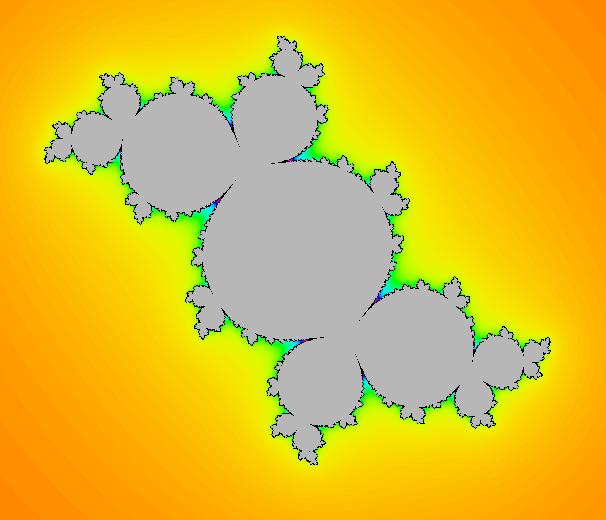,height=2.5in}}
\caption{\label{F-rab} \it On the left:  Julia set for a map $f_q$ on the
$2/7$ ray. The limit of this Julia set
 as we tend to infinity along the $2/7$ ray, 
 suitably rescaled and rotated, would be the Julia set for the 
 quadratic map $w\mapsto e^{2\pi it}w(1-w)$, as shown on the
right. Here $t=1/3$ is the rotation 
number of $2/7$ under angle doubling. In this infinite limit,
the unit circle $($drawn in on the left$)$ shrinks to the
 parabolic fixed point.}
\end{figure}

Since $\bth$ is a periodic angle,  the external ray of angle
 $\bth$ for the polynomial $P_\lambda$ must land at some periodic point.
In fact, we claim that it must land at the origin. For otherwise
it would land at a repelling point. We can then obtain a contradiction,
just as in the preceding case. Since the maps $g_{\lambda_k,b_k}$ converge
to $P_\lambda$ locally uniformly on $\Chat\ssm\{0\}$ the 
corresponding B\"ottcher coordinates in the basin of infinity 
 also converge locally uniformly,  and hence the associated external
rays converge. Hence it would follow that the rays of angle $\bth$
from infinity for the maps $g_{\lambda_k, b_k}$ 
would also land on a repelling point for large $k$. But this is impossible
since, as in the previous case, these rays must bifurcate.

This proves that the external ray of angle $\bth$ for the quadratic map 
$P_\lambda$ lands at zero. The Snail Lemma\footnote{See for example
\cite{Milnor2}.} 
 then implies that the map
$P_\lambda$ is parabolic, or in other words that $\lambda$ has the form
$e^{2\pi i\bt}$ for some rational $\bt$. It follows that the angle $\bth$
has a well defined rotation number, equal to $\bt$, under angle doubling.
Now recalling that $\lambda_k\to\lambda$ where 
 $\lambda_k=s_k/|s_k|$, it follows that $s_k$ must converge
to the point $\binf_\bt$ as $k\to\infty$. 
If the $\bth$-ray has no accumulation points in the finite plane,
 then it follows that this ray actually lands at the point $\binf_\bt$.
This completes the proof of Theorem~\ref{t-inray}.
\end{proof}\medskip

The following closely related lemma will be useful in \S\ref{s-f}.

\begin{lem}\label{L-paraland}
 If $q^2\in\partial\HO$ is a finite accumulation point
 for the periodic parameter ray $\pR_\bth$ with $\bth$ of period $n$ under doubling, then
 the dynamic internal
ray of angle $\bth$ in $\Ba$ lands at a parabolic point of ray period
 $n$.
\end{lem}

\begin{proof}[\bf Proof] This follows from the discussion above.\end{proof}
\bigskip

\section{Visible Points for Maps in $\HO$.}
\label{s-4}

If $q^2\in \HO$, then, $f_q$ is in the same 
 hyperbolic component as
 $f_0(z)=-z^3$. It follows that $f_q~$ on its Julia set is topologically
 conjugate to $f_0$ on its Julia set (the unit circle).
As a consequence, there is a homeomorphism 
$\bdeta_q:\RZ\to \J(f_q)$ conjugating the tripling map $\m_3:x\mapsto 3x$ to $f_q$: 
\[\bdeta_q\circ \m_3 = f_q\circ \bdeta_q \qquad \text{on}\quad \RZ.\] 
This conjugacy is uniquely defined up to orientation, and up to the involution $x\leftrightarrow x+1/2$ of the circle. 
Note that the two fixed points $x=0$
and $x=1/2$ of the tripling map correspond to the two fixed points 
$\bdeta_q(0)$
and $\bdeta_q(1/2)$ on the Julia set. We will always choose the orientation so
that $\bdeta_q(x)$ travels around the Julia set in the positive direction as 
seen from the origin as $x$ increases from zero to one. 

Assume $q^2\in \pR_\bth$ and set 
\[\bLambda = \bigl\{\theta\in \RZ~ ;~ \forall m\geq 0,~ 2^m\theta\neq \bth\bigr\}.\]
If $\theta\in \bLambda$, the dynamic ray $\dR_{q,\theta}$ is smooth and lands at some point $\bzeta_q(\theta)\in \J(f_q)$. 
If $\theta\not\in \bLambda$, the dynamic ray $\dR_{q,\theta}$ bifurcates and we define 
\[\bzeta_q^±(\theta)= \lim \bzeta_q(\theta±\eps)\quad \text{as}\quad  \eps\searrow 0\quad\text{with}\quad \theta±\eps\in\bLambda.\]
The limit exists since $\J(f_q)$ is a Jordan curve and $\bzeta_q$ preserves the cyclic ordering. 
 
We will choose $\bdeta_q:\RZ\to \J(f_q)$ so that 
\[\bdeta_q(0)=\bzeta_q(0)\text{ when }\bth\neq 0\quad\text{and}\quad \bdeta_q(0)=\bzeta_q^-(0)\text{ when }\bth=0.\]

\begin{remark} For each fixed $\theta\,$, 
as $q$ varies along some stretching ray $\pR_\bth$, the angle $\bdeta_q^{-1}\bigl(\bzeta_q(\theta)\bigr)$ does not change (and thus, only depends on $\bth$ and $\theta$). Indeed, $\bdeta_q$ depends continuously on $q$. In addition, $\theta$ is periodic of period $p$ for the doubling map $\m_2:\theta\mapsto 2\theta$ if and only if $\bzeta_q(\theta)$ is periodic of period $p$ for $f_q$, so if and only if $\bdeta_q^{-1}\bigl(\bzeta_q(\theta)\bigr)$ is periodic of period $p$ for $\m_3$. The result follows since the set of periodic points of  $\m_3$ and the set of non periodic points are both  totally disconnected. 
\end{remark}

Let $\Vq \subset \J(f_q)$ be the closure of the set of landing points of smooth internal rays. This is also the set of visible points in $\J(f_q)$. 
Our goal is to describe the set 
\[\V=\bdeta_q^{-1}(\Vq)\subset \RZ\]
in terms of $\bth$ and compute, for $q^2\in \pR_\bth$, 
\[\begin{cases}
\bphi_\bth(\theta)=\bdeta_q^{-1}\circ \bzeta_q(\theta)& \text{if}\quad \theta\in \bLambda\\
\bphi_\bth^±(\theta)=\bdeta_q^{-1}\circ \bzeta_q^±(\theta)& \text{if}\quad \theta\not\in  \bLambda.
\end{cases}\] 
We will prove the following: 

\begin{theorem}\label{theo:describeXtheta}
Given $\bth\in [0,1)$, set $\bthmin=\min(\bth,1/2)$ and $\bthmax=\max(\bth,1/2)$,
so that $0\leq\bthmin\le\bthmax<1$.
\begin{itemize}
\item  If $\theta\in \bLambda$, then
\[\bphi_\bth(\theta) = \frac{x_1}{3}+ \frac{x_2}{9}+\frac{x_3}{27}+\cdots,\] 
where $x_m$ takes the value $0$, $1$, or $2$ according as the
angle $2^m\theta$ belongs to the interval $[0,\bthmin)$, $[\bthmin,\bthmax)$ or $[\bthmax, 1)$
modulo $\Z$. 

\item If $\theta\not\in \bLambda$, the same algorithm yields $\bphi_\bth^+(\theta)$, whereas
\[\bphi_\bth^-(\theta) = \frac{x'_1}{3}+ \frac{x'_2}{9}+\frac{x'_3}{27}+\cdots,\] 
where $x'_m$ takes the value $0$, $1$, or $2$ according as the
angle $2^m\theta$ belongs to the interval $(0,\bthmin]$, $(\bthmin,\bthmax]$ or $(\bthmax, 1]$
modulo $\Z$. In addition
\[ \bphi_\bth^{+}(\theta)- \bphi_\bth^{-}(\theta) = \sum_{\{m~;~2^m\theta=\bth\}} \frac{1}{3^{m+1}}.\]
\end{itemize}
The set $\V$ is the Cantor set of angles whose orbit under tripling avoids the critical gap $\cgap = \bigl(\bphi_\bth^-(\bth),\bphi_\bth^{+}(\bth)\bigr)$. 
The critical gap $\cgap$ has length 
\[\ell(\cgap) = \begin{cases}
1/3&\text{ when }\bth \text{ is not periodic for }\m_2\quad and \\
3^{p-1}/(3^p-1)\in (1/3,1/2]  & \text{ when }\bth \text{ is periodic of period }p\geq 1.
\end{cases}\]
\end{theorem}

\begin{remark}
It might seem that the expression for $\bphi_\bth$ should have a discontinuity
whenever $2^m\theta=1/2$, and hence $2^n\theta=0$ for $n>m$, but a brief
computation shows that the discontinuity in $x_m/3^m$ is precisely
 canceled out by the discontinuity in $x_{m+1}/3^{m+1}+\cdots$.
 \end{remark}
\smallskip

\begin{coro}\label{C-mono} The left hand endpoint 
$a(\bth) = \bphi_\bth^-(\bth)$ of $\cgap$ is continuous from the left as a function of  $\bth$.
Similarly, the right hand endpoint 
$b(\bth)=\bphi_\bth^+(\bth)$ is continuous from the right. Both $a(\bth)$
 and $b(\bth)$ are monotone as functions of $\bth$.
\end{coro}
\smallskip

\begin{proof}
This follows easily from the formulas in Theorem \ref{theo:describeXtheta}.
\end{proof}

\begin{proof}[\bf Proof of Theorem \ref{theo:describeXtheta}]
In the whole proof, we assume $q^2\in \pR_\bth$. 
If an internal ray $\dR_{q,\theta}$ lands at  $\bzeta_q(\theta)$, then the image ray $\dR_{q,2\theta}$ lands at $\bzeta_q(2\theta)=f_q\circ \bzeta_q(\theta)$. It follows that the set of landing points of smooth rays is invariant by $f_q$, and so the set $\V$ is invariant by $\m_3$. 

A connected component of $(\RZ)\ssm \V$ is called a \textbf{\textit{gap}}. 

\begin{lem}
The gaps are precisely the open interval 
\[\gap_\theta = \big(\bphi_\bth^-(\theta),\bphi_\bth^+(\theta)\big),\quad \theta\not \in \bLambda.\]
The image by $\m_3$ of a gap of length less $\ell<1/3$ is a gap of length $3\ell$. 
The gap $\cgap$ is the unique gap of length $\ell\geq 1/3$; in addition $\ell<2/3$. 
\end{lem}

\begin{proof}
Assume $\gap=(a,b)$ is a gap. If the length of $\gap$ is less than $1/3$, then $\m_3$ is injective on $\overline \gap$. Since $a$ and $b$ are in $\V$ and since $\V$ is invariant by $\m_3$, we see that the two ends of $\m_3(\gap)$ are in $\V$. To prove that $\m_3(\gap)$ is a gap, it is therefore enough to prove that $\m_3(\gap)\cap \V=\emptyset$. 

First, $\bdeta_q(a)$ and $\bdeta_q(b)$ can be approximated by landing points of smooth rays. So, there is a (possibly constant) non-decreasing sequence $(\theta_n^-)$ and a non-increasing sequence $(\theta_n^+)$ of angles in $\bLambda$ such that $\bzeta_q(\theta_n^-)$ converges to $\bdeta_q(a)$ and $\bzeta_q(\theta_n^+)$ converges to $\bdeta_q(b)$. 
Then, the sequence of angles $a_n = \bphi_\bth(\theta_n^-)\in \V$ is non-decreasing and converges to $a$, and 
the sequence of angles $b_n = \bphi_\bth(\theta_n^+)\in \V$ is non-increasing and converges to $b$. 

Now, choose $n$ large enough to that the length of the interval $I_n=(a_n,b_n)$
is less than $1/3$ and let $\Delta$ be the triangular region bounded by the rays $\dR_{q,\theta_n^±}$, together
with $\bdeta_q(I_n)$ (see Figure~\ref{F-vis}). Then $f_q$ maps $\Delta$
 homeomorphically onto the triangular region $\Delta'$ which is bounded by
 the internal rays of angle $\dR_{q,2\theta_n^±}$ together with $\bdeta_q\bigl(\m_3(I_n)\bigr) = f_q\bigl(\bdeta_q(I_n)\bigr)$.
In fact $f_q$ clearly maps the boundary $\partial\Delta$ homeomorphically
onto $\partial\Delta'$. Since $f_q$ is an open mapping, it must map the
interior into the interior, necessarily with degree one. 

We can now prove that $\m_3(\gap)\cap \V=\emptyset$. Indeed, if this were not the case, then 
$\Delta'$ would contain a smooth ray landing on $\bdeta_q\bigl(\m_3(\gap)\bigr)=f_q\bigl(\bdeta_q(\gap)\bigr)$. 
Since $f_q:\Delta\to \Delta'$ is a homeomorphism, $\Delta$ would contain a smooth ray landing on $\bdeta_q(\gap)$, contradicting that $\gap\cap \V=\emptyset$. This completes the proof that the image under $\m_3$ of a gap of length $\ell<1/3$ is a gap of length $3\ell$. 

We can iterate the argument until the length of the gap $\m_3^m(\gap)$ is $3^m\ell\geq 1/3$. So, there is at least one gap $\gap_\star$ of length $\ell_\star\geq 3$. If there were two such gaps, or if $\ell_\star\geq 2/3$, then every point $x$ in $\RZ$ would have two preimages in $\gap_\star$, contradicting the fact that every point in $\J(f_q)$ which is the landing point of a smooth internal ray $\dR_{q,\theta}$ (except the ray $\dR_{q,2\bth}$ which passes through the critical value) has two preimages which are landing points of the smooth internal rays $\dR_{q,\theta/2}$ and $\dR_{q,\theta/2+1/2}$. So, there is exactly one gap $\gap_\star$ of length $\ell_\star\geq 1/3$ ; we have $\ell_\star<2/3$ and the gaps are precisely the iterated preimages of $\gap_\star$ whose orbit remains in $(\RZ)\ssm \gap_\star$ until it lands on $\gap_\star$. 

To complete the proof of the lemma, it is enough to show that $\gap_\star$ is the critical gap:
\[\gap_\star = \bigl(\bphi_\bth^-(\bth),\bphi_\bth^+(\bth)\bigr).\]
To see this, we proceed by contradiction. If $\bphi_\bth^-(\bth)= \bphi_\bth^+(\bth)$ or if the length of the interval 
$\bigl(\bphi_\bth^-(\bth), \bphi_\bth^+(\bth)\bigr)$ were less than $1/3$, then as above, we could build a triangular region 
$\Delta$ containing the bifurcating ray $\dR_{q,\bth}$ and mapped homeomorphically by $f_q$ to a triangular region $\Delta'$. This would contradict the fact that the critical point $\cpo(q)$ is on the bifurcating ray $\dR_{q,\bth}$, therefore in $\Delta$. 
\end{proof}

\begin{figure}[ht]
\includegraphics[width=2.5in]{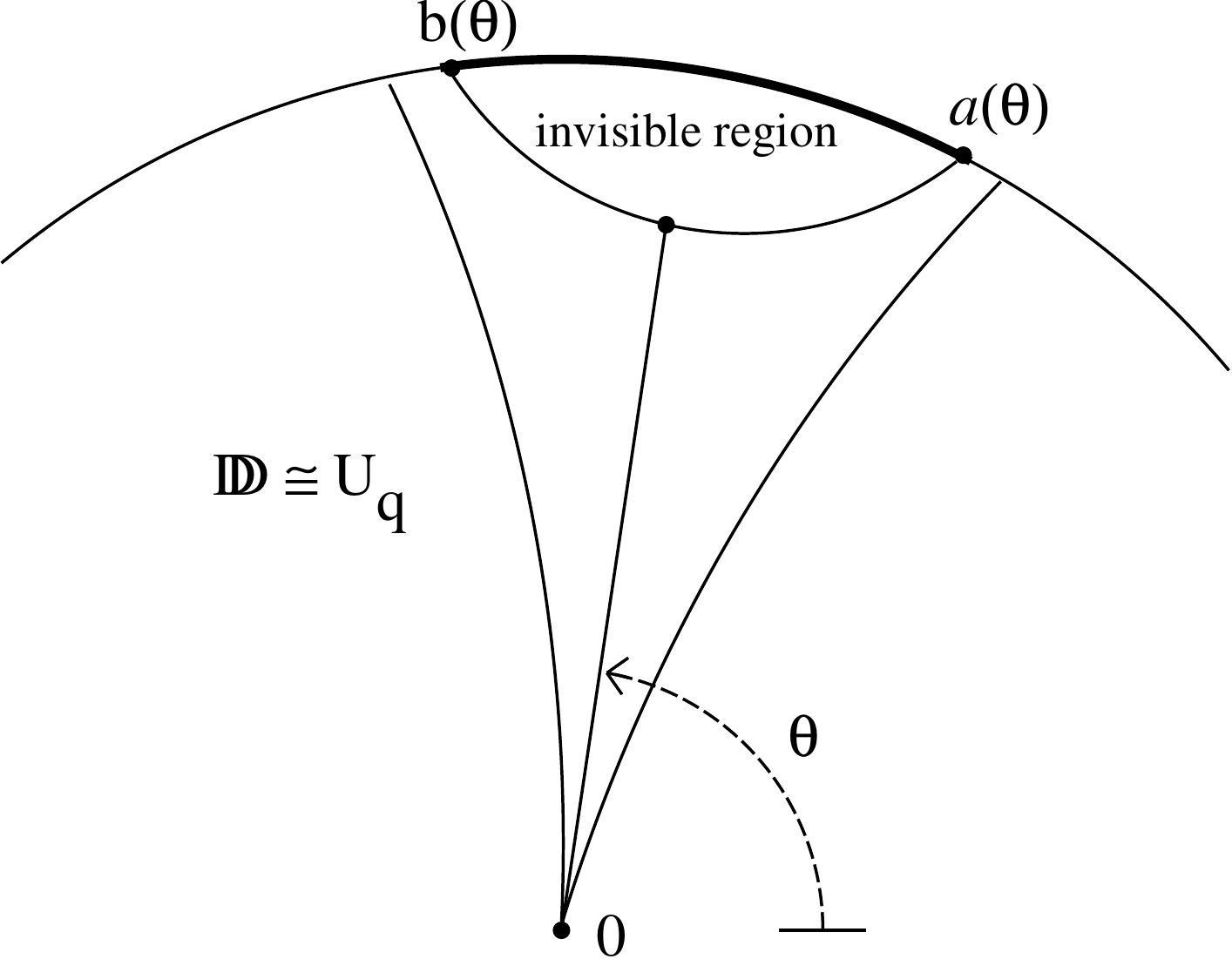}
\caption{\label{F-vis}\it Schematic picture after mapping
the basin $\Ba$ homeomorphically
onto the unit disk, showing three internal rays.
 Here the  middle internal ray of angle
$\theta$ bifurcates from a critical or precritical point. The gap
$\gap_\theta = \big(\bphi_\bth^-(\theta),\bphi_\bth^+(\theta)\big)$ has been emphasized.}
\end{figure}

We have just seen that the gaps are precisely the iterated preimages of $\cgap$ whose orbit remains in $(\RZ)\ssm \cgap$ until it lands on $\cgap$. 
In other words, $\V$ is the Cantor set of angles whose orbit under $\m_3$ avoids $\cgap$. 
We will now determine the length of $\cgap$.

\begin{lem}\label{L-per}
If $\bth$ is not periodic for $\m_2$, then the length of 
$\cgap$ is $1/3$. If $\bth$ is periodic of period $p$ 
for $\m_2$, then the length of $\cgap$ is $\displaystyle 
\frac{3^{p-1}}{3^p-1}$. 
\end{lem} \smallskip

\begin{proof}[\bf Proof] Let $\ell$ be the length of the critical gap $\cgap$. According to the previous lemma,  $1/3\le\ell<2/3$. 

If $\bth$ is not periodic for $\m_2$, then the internal ray $\dR_{q,2\bth}$ lands at a point $z$ which has three preimages $z_1$, $z_2$ and $z_3$ in $\J(f_q)$. The ray $\dR_{q,\bth+1/2}$ lands at one of those preimages,  say $z_0$, and the ends of the gap $\cgap$ are $\bdeta_q^{-1}(z_1)$ and $\bdeta_q^{-1}(z_2)$. In particular, $\m_3$ maps the two ends of $\cgap$ to a common point, so that $\ell=1/3$. 

If $\bth$ is periodic of period $p$ for $\m_2$, then the internal ray $\dR_{q,2\bth}$ bifurcates and there is a gap $\gap_{2\bth}$. 
The image of the gap $\cgap$ by $\m_3$ covers the gap $\gap_{2\bth}$ twice 
and the rest of the circle once. In particular, the length of $\gap_{2\bth}$ is $3\ell-1$. 
In addition, $\m_3^{\circ p-1}(\gap_{2\bth})= \cgap$, which implies that 
\[3^{p-1}\cdot (3\ell-1)=\ell,\qquad\text{hence}\qquad \ell=\frac{3^{p-1}}{3^p-1}.\qedhere\]
\end{proof} \medskip

\begin{remark}\label{r-sum=1}\rm
The length of an arbitrary gap $\gap_\theta$ is equal to the length of the critical gap 
 $\cgap$ divided by $3^m$, where $m$ is the smallest 
integer with $2^m\theta=\bth\in\RZ$. Whether or not $\bth$ is
 periodic, we can also write
\[\ell(\gap_\theta)=\bphi_\bth^+(\theta)-\bphi_\bth^-(\theta) = \sum_{\{m~;~2^m\theta=\bth\}}1/3^{m+1},\]
to be summed over {\em all} such $m$. In all cases, the sum of the lengths of
 all of the gaps is equal to $+1$. In other words, $\V$ is
 always a set of measure zero.
\end{remark}\smallskip

\begin{remark}\label{rem:defpsi}
The map $\bpsi_\bth:\RZ\to \RZ$ defined by 
\[\begin{cases}
\bpsi_\bth(x) = \theta&\text{if }x = \bphi_\bth(\theta)\quad\text{and}\\
\bpsi_\bth(x) = \theta&\text{if }x  \in \overline \gap_\theta.
\end{cases}
\]
 is continuous and monotone; it is constant on the closure of each gap; it satisfies
\[\bpsi_\bth\circ \m_3 = \m_2\circ \bpsi_\bth\quad\text{on}\quad \V.\]
\end{remark}

We finally describe the algorithm that computes the value of $\bphi_\bth(\theta)$ for $\theta\in \bLambda$. 
To fix our ideas, let us assume that $\bth\ne 0$.
Then exactly one of the two fixed points of $\m_3$ must be visible.\footnote
{In the period one case the critical gap
is $(0,1/2)$, and both fixed points $\bdeta_q(0)$ and
$\bdeta_q(1/2)$ are visible boundary points.  (This case is illustrated in
 Figure~\ref{f4a}. However note that these Julia set fixed points 
 are near the top
and bottom of this figure---not at the right and left.) In this case,
 the set $\V$ is a classical
middle third Cantor set.}
One of them is the landing point of the ray $\dR_{q,0}$ and we normalized $\bdeta_q$ so that this 
fixed point is $\bdeta_q(0)$. 
The other fixed point, $\bdeta_q(1/2)$ is not visible since the map $\m_2$ has only one fixed point.
So, $1/2$ must belong to some gap. If this gap had length $\ell<1/3$, it would be mapped by $\m_3$ 
to a gap of length $3\ell$ still containing $1/2$ (because $\m_3(1/2)=1/2$). This is not possible. 
Thus,  $1/2$ must belong to the critical gap $\cgap$.

Since the three points $0,1/3,2/3\in\RZ$
all map to $0$ under $\m_3$ and since $0\in \V$,  it follows that 
\begin{itemize}
\item either both $1/3$ and $2/3$ are in the boundary of the critical gap $\cgap$ 
and are mapped to $\bth=1/2$, 
\item or one of the two points $1/3$ and  $2/3$ belongs to  $\V$ and is mapped by $\bpsi_\bth$ to $1/2$ ; the other is in the critical gap $\cgap$ and 
is mapped to $\bth$. 
\end{itemize}
To fix our ideas, suppose that $1/3$ 
belongs to $\cgap$ so that  $2/3\in \V$. It then follows easily that
\[\bthmin=\bth = \bpsi_\bth(1/3)<\bpsi_\bth(2/3)=1/2 = \bthmax.\]
Fix $\theta\in \bLambda$ and set $x=\bphi_\bth(\theta)=x_1/3+x_2/9+x_3/27+\cdots$. 
A brief computation shows that
\[\frac{x_1}{3}\le x< \frac{x_1+1}{3}.\]
Hence $\bphi_\bth(x)=\theta$ belongs to the interval
\[ [0,\bth)\quad{\rm or}\quad [\bth,1/2)
\quad{\rm or}\quad [1/2,1)\qquad {\rm mod}~\Z\]
according as $x_1$ takes the value $0$, $1$ or $2$. Similarly,
$\bphi_\bth(3^m x)=2^m \theta$ belongs to one of these three intervals according to
the value of $x_{m+1}$. The corresponding property for $\bphi_q$ follows
immediately.

Further details will be left to the reader.
\end{proof}

\newpage

As in the appendix,  if $\I\subset \RZ$ is an open set, we denote by $X_d(\I)$ the set of points
in $\RZ$ whose orbit under multiplication by $d$ avoids $\I$. Lemma~\ref{L-XJ} asserts that if
$\I\subset\RZ$ is  the union of disjoint open subintervals 
$\I_1,\I_2,\cdots,\I_{d-1}$, each of length precisely $1/d$, then $X_d(\I)$ is a  rotation
 set for multiplication by $d$.
 
If the angle $\bth$ is not periodic under doubling, so that the critical gap $\cgap$ has length
$1/3$, then the set $\V$ is equal to $X_3(\I)$, taking $\I$ to be the critical gap $\cgap$. 
In the periodic case,  where $\cgap$
has length $3^{p-1}/(3^p-1)>1/3$, we can  take $\I$ to be {\em any} open
interval of length $1/3$ which is compactly contained in the critical gap.

\begin{lem}\label{L-X=X}
Let $\I$ be an  open interval of length $1/3$ which is equal to $\cgap$ or 
compactly contained in $\cgap$. Then the  set $\V$ is
equal to the set $X_3(\I)$.
\end{lem}
\smallskip

\begin{proof}[\bf Proof]
It is enough to observe that the orbit of any point that enters
 $\cgap$ eventually enters $\I$. If $\I=\cgap$, there is noting to prove. So, assume $\I$ is compactly contained in $\cgap$. 
 Let $H$  be either one of the two components of
 $\cgap\ssm \I$.  Then, 
\[{\rm Length}(H)< {\rm Length}(\cgap)-{\rm Length}(\I)
= \frac{3^{p-1}}{3^p-1}-\frac{1}{3} = \frac{1}{3(3^p-1)}\]
and therefore
\[(3^p-1) {\rm Length}(H) < \frac{1}{3} = {\rm Length}(\I).\]
It follows that
\[{\rm Length}\bigl(\m_3^{\circ p}(H)\bigr) = 3^p {\rm Length}(H) <
{\rm Length}(H) + {\rm Length}(\I).\]
Since the boundary points of $\cgap$ are periodic of period 
$p$,  this implies that
\[\m_3^{\circ p}(H) \subset H \cup \I.\]
Since $\m_3^{\circ p}$ multiplies distance by $3^p$, the orbit of any
 point in $H$ under $\m_3^{\circ p}$ eventually enters $\I$,  as
 required.
\end{proof}

\section{Dynamic Rotation Number for Maps in $\HO$\label{sec:dynrot}}

We now explain how to assign a \textbf{\textit{dynamic rotation number}} 
to  every parameter $q^2$ on the parameter ray  $\pR_\bth$. 

As previously, we denote by $\Vq$ the set of points $z\in \J(f_q)$ which are visible from the origin. In the previous section, we saw that  $\Vq=\bdeta_q(\V)$  where $\V$ is the set of angles whose orbit under $\m_3$ never enters the critical gap $\cgap$. 
The set of points $z\in \J(f_q)$ which are visible from infinity is 
\[\ant(\Vq) = \bdeta_q(1/2+\V).\]
So, the set of \textbf{\textit{doubly visible points}}, that is those which are simultaneously visible from zero and infinity, is simply 
\[\bdeta_q(\X)\quad\text{with}\quad \X = \V\cap (1/2+\V).\]
Note that $\X$ is the set of angles whose orbit under $\m_3$ avoids $\cgap$ and $1/2+\cgap$.

Since the length of $\cgap$ is at most $3^0/(3^1-1) = 1/2$, the intervals $\cgap$ and $1/2+\cgap$ are disjoint. 
Let $\I_1\subset \cgap$ be an open interval of length $1/3$, either equal to $\cgap$, or compactly contained in $\cgap$. 
Set $\I_2=1/2+\I_1$ and $\I=\I_1\cup \I_2$. 
It follows from Lemma \ref{L-X=X} that $\X$ is equal to the set $X_3(\I)$ of  points whose orbit never enters $\I$. 
According to Lemma~\ref{L-XJ}, $\X$ is a rotation set for $\m_3$. 

\begin{remark}
If we only assume that $\I_1\subset \cgap$ without assuming that it is compactly contained, and set $\I=\I_1\cup(1/2+\I_1)$, then
the set $X_3(\I)$ is still a rotation set containing $\X$ but it is not necessarily reduced. In any case, the rotation number of $X_3(\I)$ is equal to the rotation number of $\X$. 
\end{remark}

%
%

\begin{figure}[ht]
\centerline{\includegraphics[width=2.5in]{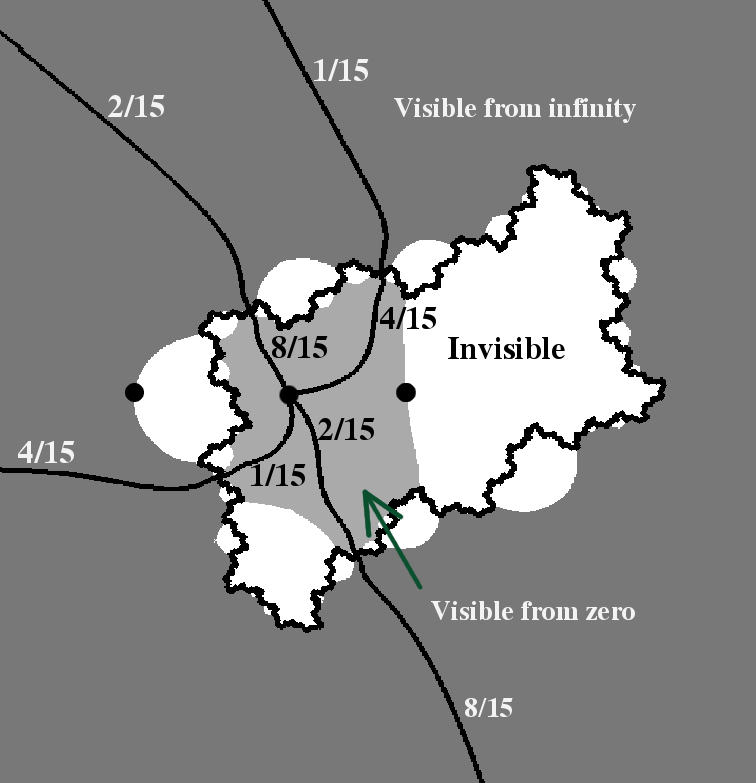}}
\caption{\it The Julia set of a map $f_q$ with dynamic rotation
number $1/4$. There are  four 
points visible from both zero and infinity. The white points are not
visible from either zero or infinity, The angles of the rays from infinity,
counterclockwise starting from the top, are $1/15, 2/15, 4/15$ and 
$8/15$, while the
corresponding rays from zero have angles $4/15, 8/15, 1/15$ and 
$2/15$.  The heavy
dots represent the free critical points.
\label{fV}}
\end{figure}
\bigskip

\begin{definition}{\bf Dynamic Rotation Number.}
\label{D-drn}
For any parameter $q^2\in \pR_\bth$, the
 rotation number $\bf t = \brho(\bth)$ associated with the set
 \[\X=\V \cap (1/2+\V)\subset\RZ,\] 
 representing points in the Julia set which are doubly visible, will be called the
\textbf{\textit{dynamic rotation number}}\footnote{This definition
will be extended to many points outside of $\HO$ 
in Remark \ref{R-drn}.
 (See also Remark~\ref{R-frn}.)}
of the map $f_q$ or of the point $q^2$ in moduli space.
\end{definition}\smallskip

Consider the monotone degree one map $\bpsi_\bth:\RZ\to \RZ$ defined Remark~\ref{rem:defpsi}. It is a semiconjugacy between 
\[\m_3:\V\to \V\quad\text{and}\quad \m_2:\RZ\to \RZ.\] 
Restricting $\bpsi_\bth$ to the $\m_3$-rotation set $\X\subset \V$, the image $\bpsi_\bth(\X)$
is an $\m_2$-rotation set, with rotation number equal to $\bt=\brho(\bth)$. In fact, this rotation set is reduced. 
According to Goldberg (see Theorem~\ref{T-gold}), 
there is a unique such rotation set $\bTheta$. 
If $q^2\in \pR_\bth$, then  $\bTheta$ is the set of internal angles of doubly visible points. 
\smallskip

The entire configuration consisting of the rays from zero and infinity
with angles in $\bTheta$,
 together with their common landing points in the
Julia set, has a well defined rotation number $\bt$ under the map $f_q$.
For example, in Figure~\ref{fV} the light grey region is the union of interior
rays, while the dark grey region is the union of rays from infinity. These two
regions have four common boundary points, labeled by the points of $\X$.
The entire configuration of four rays from zero and four rays from infinity
maps onto itself under $f_q$ with combinatorial rotation number $\bt$
 equal to $1/4$.
\medskip

\begin{lem}\label{lem-psi}
The dynamic rotation number $\bt$ is continuous and monotone
$($but not strictly monotone$)$ as a function $\brho(\bth)$ of 
 the critical angle $\bth$. 
 Furthermore, as $\bth$ increases from zero to one, the dynamic rotation
 number also increases from zero to one.
\end{lem}

\begin{proof}[\bf Proof] Setting $\cgap=\big(a(\bth), b(\bth)\big)$, and setting
$\I_\bth=\big(a(\bth),a(\bth)+1/3)\subset \cgap$, we have $\rho(\bth)=\rot\bigl(X_2(\I_\bth)\bigr)$ 
where $X_2(\I_\bth)$ is the set of points whose orbit never enters $\I_\bth$ under iteration of $\m_2$. 
Hence continuity from the left follows from Corollary~\ref{C-mono} together with  
Remark~\ref{R-mono}. Continuity from the right and
monotonicity follow by a similar argument.
\end{proof}
\bigskip

We will  describe the function $\bth\mapsto \bt=\brho(\bth)$
conceptually in Theorem~\ref{T-balcomp}, and give an explicit computational
description in Theorem~\ref{T-inv}.    (For a graph of this function see 
Figure~\ref{F-t2th}.) We first need a preliminary discussion. 
\smallskip

\begin{figure}[ht]
\centerline{\includegraphics[width=2.5in]{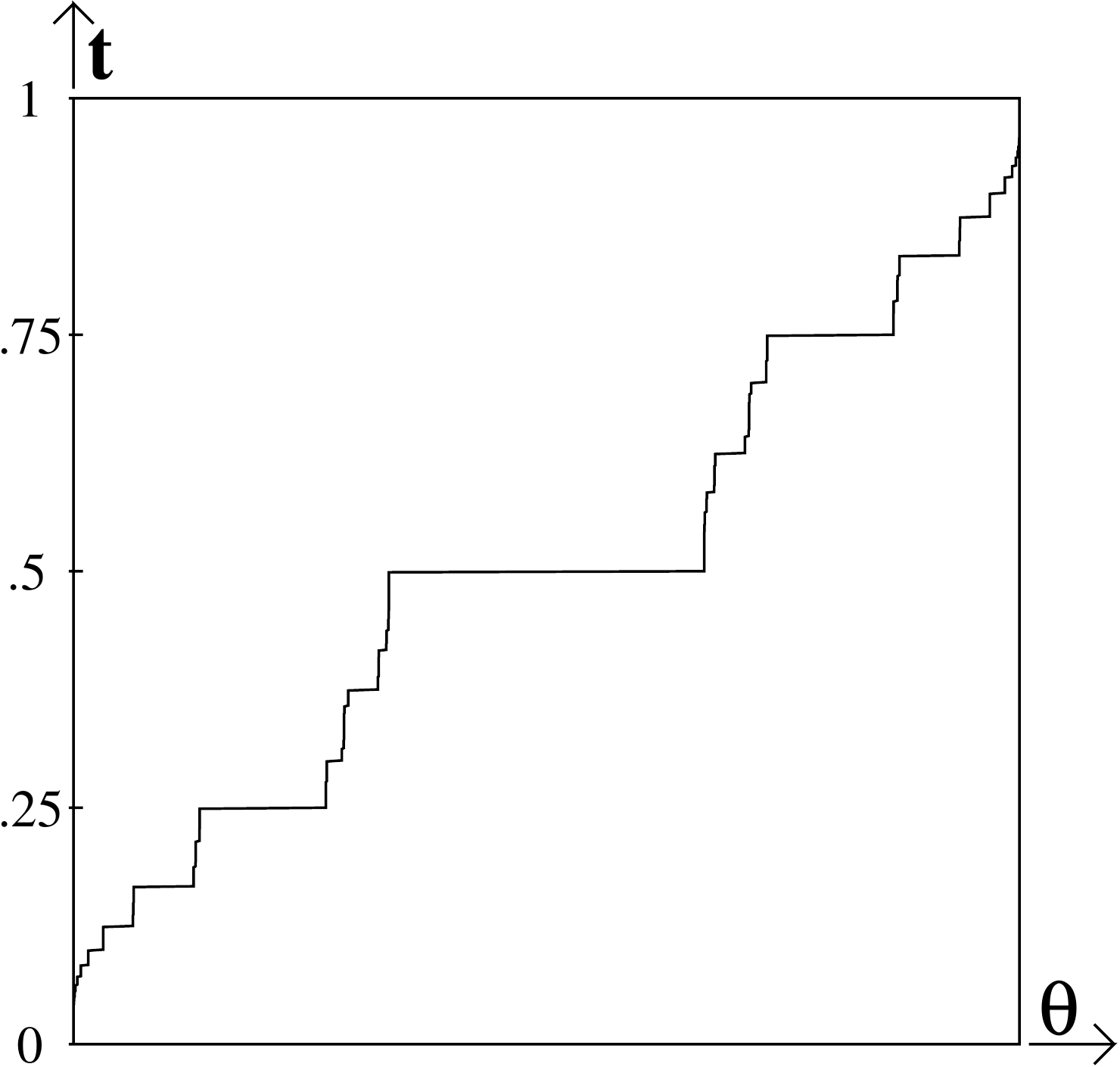}}
\caption{\label{F-t2th}\it Graph
of the dynamic rotation number $\bf t$ as a function of the 
critical angle $\bth$.}
\end{figure}
\smallskip

\begin{prop}\label{p7}
If  $\bt$ is rational and  $\bth$ is not periodic with rotation number
 $\bt$ under $\m_2$, then $\X$ consists of a single cycle for 
$\m_3$. However, if $\bt$ 
 is rational and $\bth$ is periodic with
 rotation number $\bt$ under $\m_2\,$, then $\X$ may consist of
either one or two cycles for $\m_3$.
\end{prop}
\smallskip

As examples, in Figure \ref{f-clov} , the angle $\bth=2/7$
has rotation number $1/3$ under $\m_2$; and in fact $\X$ is the union of two cycles, both
with rotation number $1/3$.
On the other hand, in Figures \ref{f-circ} and \ref{fV}, $\bth$
can be any angle between $2/15$ and $4/15$, but can never be periodic with
rotation number $1/4$. In this case the set $\X$
is a single cycle of rotation number $1/4$. 
\smallskip

\begin{proof}[\bf Proof of Proposition \ref{p7}]
Assume $\bt$ is rational. Then $\bTheta$ is a finite set consisting of a single cycle for $\m_2$. 

On the one hand, if $\bth$ does not belong to $\bTheta$, then $\bpsi_\bth:\X\to \bTheta$ is a  homeomorphism.
 In that case, $\X$ consists of a single cycle for $\m_3$. 

On the other hand, if $\bth$ belongs to $\bTheta$, then $\bpsi_\bth^{-1}(\bTheta)\cap \V$ is the union of two cycles, namely the
orbits of the two boundary points of the critical gap $\cgap$. Either
\begin{itemize}
\item 
$\bpsi_\bth^{-1}(\bTheta)= 1/2+\bpsi_\bth^{-1}(\bTheta)$ and $\X$ contains two periodic cycles, or
\item 
$\bpsi_\bth^{-1}(\bTheta)\neq 1/2+\bpsi_\bth^{-1}(\bTheta)$ and $\X$ contains a single periodic cycle.\qedhere
\end{itemize}
\end{proof}
\medskip

\begin{figure}[ht]
\centerline{\includegraphics[width=2.5in]{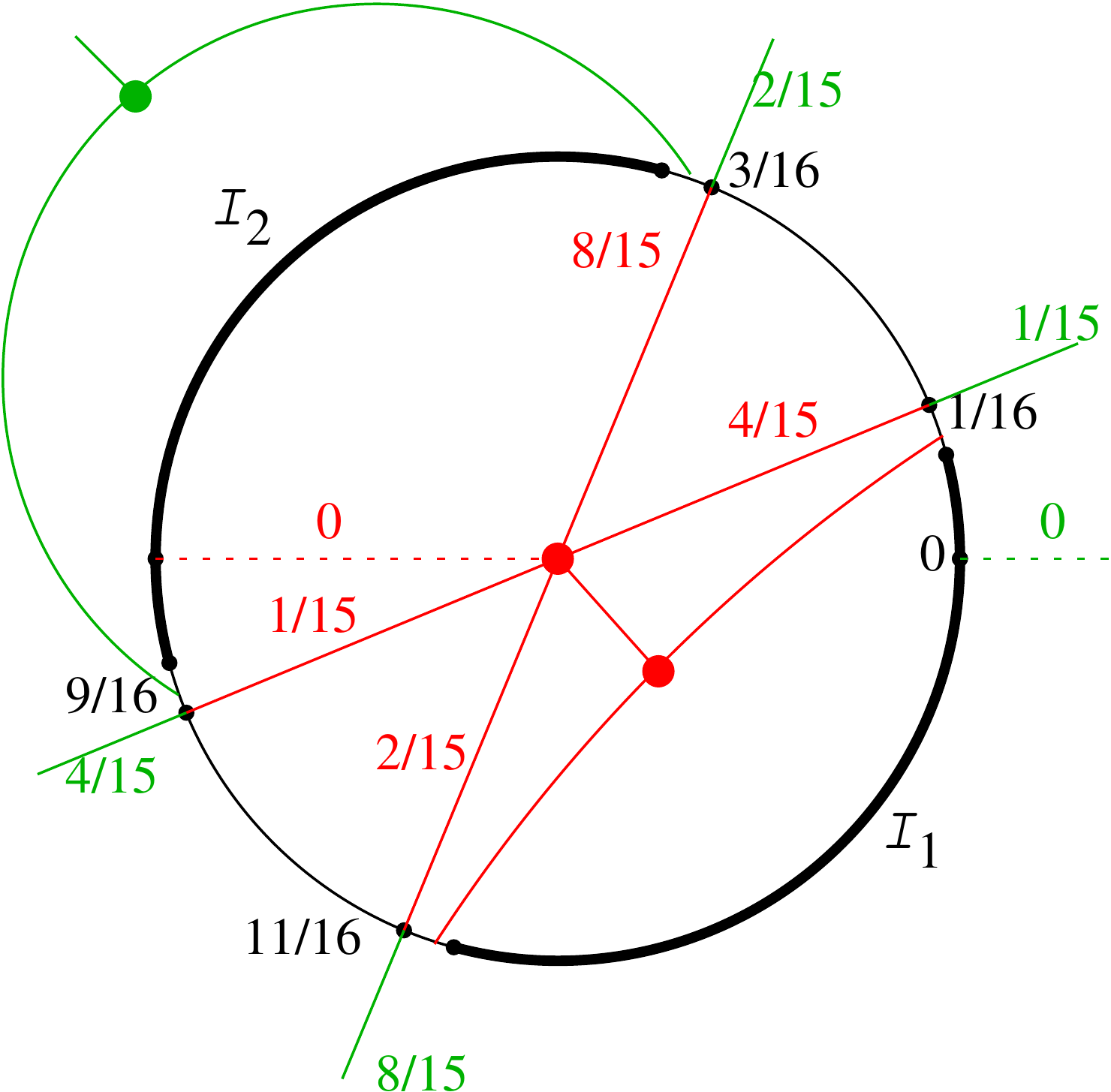}}
\caption{\label{f-circ}\it Schematic picture representing the Julia set
by the unit circle. 
Here $\I_2$ can be any open interval
 $(x,x+1/3)$ with $3/16\le x<x+1/3\le 9/16$, and $\I_1=\I_2+1/2$. 
The associated rotation number is $1/4$. The
interior angles $(${\rm in red}$)$ and the exterior angles $(${\rm in green}$)$ 
double under the map, while the angles around the Julia set 
$(${\rm in black}$)$ multiply by $~3\,$.
Heavy dots indicate critical points. 
}
\end{figure}

\begin{definition}{\bf Balance.}\label{D-bal} For each 
$\bt \in\RZ$, let $\bTheta$ be the unique reduced rotation set
with rotation number $\bt$ under the doubling map $\m_2$ (see the appendix).
If $\bt$ is rational with denominator $p>1$, then the points of $\bTheta$
can be listed in numerical order within the open interval $(0,1)$
as $\theta_1<\theta_2<\cdots<\theta_p$. If $p$ is odd, then there is a unique middle
element $\theta_{(p+1)/2}$ in this list. By definition, this middle element
will be called the \textbf{\textit{balanced}} angle in this
rotation set. (For the special case $p=1$, the unique element $0\in \Theta_0$
will also be called balanced.)

On the other hand, if $p$ is even, then there is no middle element.
However, there is a unique pair $\{\theta_{p/2},\theta_{1+p/2}\}$ in the middle.
By definition, this will be called the \textbf{\textit{balanced-pair}}
for this rotation set, and the two elements of the pair will be called
\textbf{\textit{almost balanced}}. 

Finally, if $t$ is irrational, then $\bTheta$ is topologically a Cantor set.
However, every orbit in $\bTheta$ is uniformly distributed with respect to
a uniquely defined invariant probability measure. By definition, the
\textbf{\textit{balanced point}} in this case is the unique $\theta$
such that both $\bTheta\cap(0,\theta)$ and $\bTheta\cap(\theta,1)$ have measure $1/2$.
(The proof that $\theta$ is unique depends on the easily verified statement
 that $\theta$ cannot fall in a gap of $\bTheta$.)
\end{definition}
\medskip

Here is a  conceptual  description of the function $\bt\mapsto\bth$.
\medskip

\begin{theorem}\label{T-balcomp}
 If $\bt$ is either irrational, or rational with odd denominator,
then $\bth$ is equal to 
 the unique balanced angle in the rotation set $\bTheta$.
However, if $\bt$ is rational with even denominator, then there is an entire
closed interval of corresponding $\bth$-values. This interval is bounded by
the unique balanced-pair in $\bTheta$.
\end{theorem}
\smallskip

We first prove one special case of this theorem. 
\smallskip

\begin{lem}\label{L-bal} Suppose that the dynamic rotation number is
  $\bt=m/n=m/(2k+1)$,
rational with odd denominator. Then the corresponding critical angle
 $\bth$
is characterized by the following two properties:
\smallskip

\begin{itemize}
\item $\bth$ is periodic under doubling with rotation number 
$\bt$. \smallskip

\item $\bth$ is balanced. If $k>0$ this means that
 exactly  $k$ of the points of 
the orbit  $\{2^\ell\bth\}$ belong to the open interval 
\hbox{$(0,\bth) \mod 1$},
and $k$ belong to the interval \hbox{$(\bth,1) \mod 1$}.
\end{itemize}
\end{lem}
\medskip


\begin{figure}[ht]
\centerline{\psfig{figure=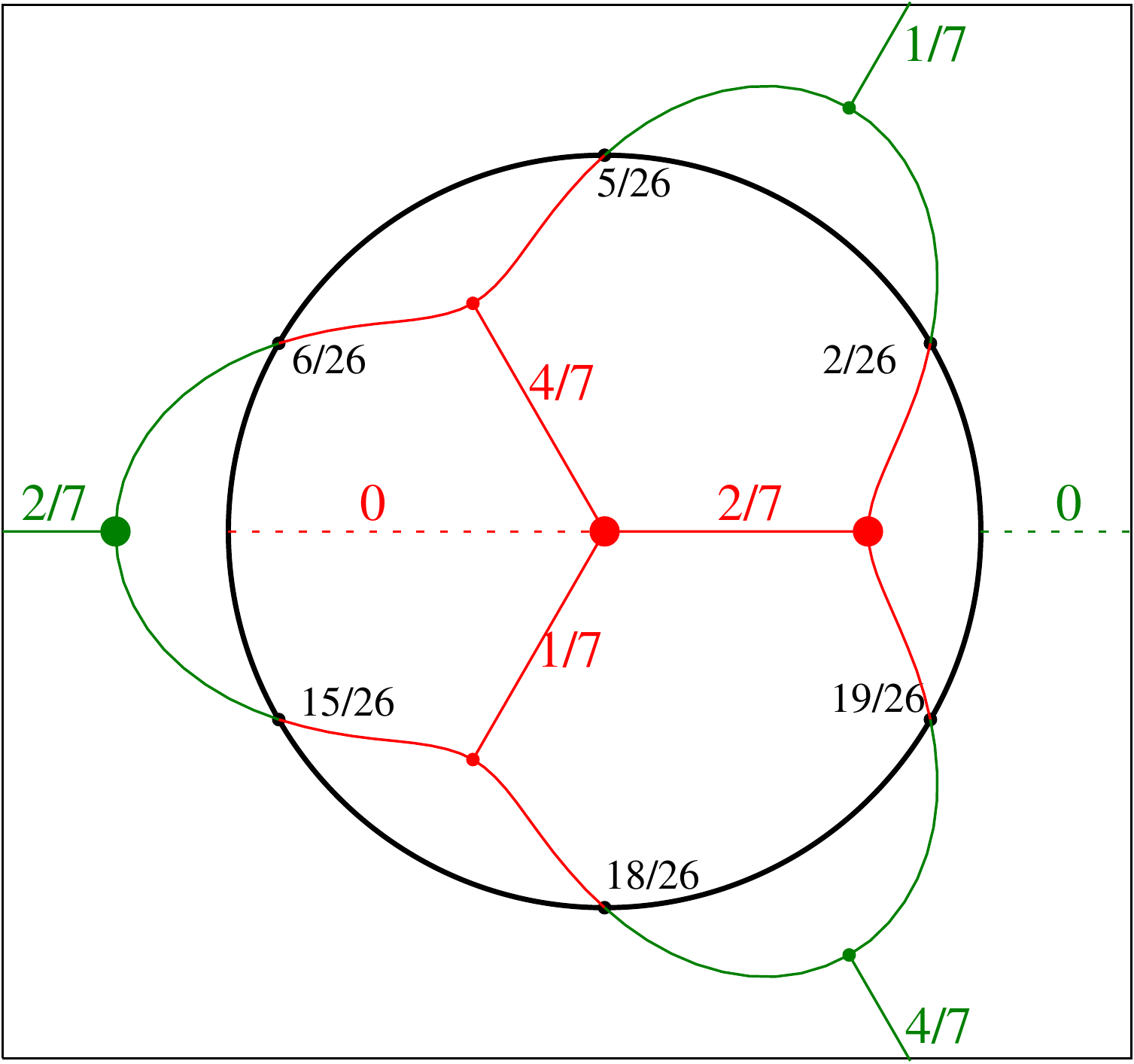,height=2.3in}}
\caption{\label{f-clov} \it The circle in this cartoon represents
 the Julia set $($which is actually a quasicircle$)$
 for a map belonging to the $2/7$
 ray in $\HO$. 
The heavy dots represent the three finite critical points. As in
The interior angles $($red$)$ and exterior angles
$($green$)$ are periodic under doubling,
 while the angles around the Julia set are periodic under tripling.
$($The figure is topologically correct, but all angles are distorted.$)$}
\end{figure}
\medskip

(Compare Definition~\ref{D-bal}.)
Figure~\ref{f-clov} provides an example to illustrate this lemma. Here
$\bth=2/7$, and the rotation number is $\bt=1/3$.
 Exactly one point of the orbit of $\bth$ under
doubling belongs to the open interval $(0,\bth)$, and exactly one
 point belongs to $(\bth,1)$. The critical gap $\cgap$ for points
 visible
from the origin is the interval $(6/26,15/26)$  of length $9/26$. 
Similarly the critical gap $1/2+\cgap$ for points visible from
 infinity is the arc from $19/26\equiv -7/26$ to $2/26$.
\bigskip

\begin{proof}[\bf Proof of Lemma \ref{L-bal}] According to \cite{G}, there is only one
 $\m_2$-periodic orbit with rotation number
 $\bf t$, and clearly such a periodic orbit is balanced with respect to
 exactly one of its points. (Compare Remark~\ref{R4} and Definition 
\ref{D-bal}.)

To show that $\bth$ has these two properties,
consider the associated $\m_3$-rotation set $\X$. Since $\X$ is
self-antipodal with odd period,
 it must consist of two mutually antipodal periodic orbits.
As illustrated in Figure~\ref{f-clov},
we can map $\X$ to the unit circle in such a way that $\m_3$ corresponds to
 the rotation $z\mapsto e^{2\pi i\bt}z$. If the rotation number is non-zero,
so that  $k\ge 1$, then 
 each of the two periodic orbits must
map to the  vertices of a regular  $(2k+1)$-gon. As in 
Figure~\ref{f-clov},
one of the two fixed points $0$ and $1/2$ of $\m_3$ must lie in the interval $\cgap$,
and the other must lie in $1/2+\cgap$. It follows that  $2k+1$ of the
 points of $\X$ lie in the interval $(0,1/2)$, and the other 
 $2k+1$ must lie in
$(1/2,1)$. It follows easily that  $k$ of the points of the
 orbit of  $\bth$
must lie in the open interval $(0,\bth)$ 
and $k$ must lie in  $(\bth,1)$, as required.
\smallskip

The case of rotation number zero,
 as illustrated in Figure~\ref{f4a}, is somewhat different. 
In this case, the rotation set $\X$ for $\m_3$
consists of the two fixed points $0$ and
$1/2$. The corresponding rotation set $\bTheta$ for $\m_2$
consists of the single fixed point zero,
which is balanced by definition.
\end{proof}
\medskip

\begin{figure}[ht]
\centerline{\includegraphics[width=2.3in]{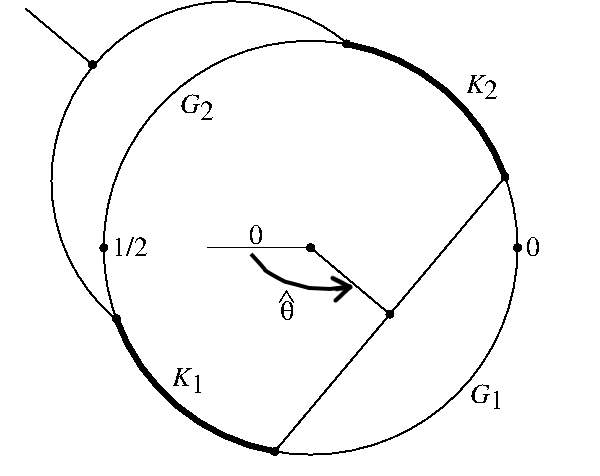}}
\caption{\label{F-bal} \it  Schematic diagram for any example with 
rotation number $\bt\ne 0$ $($with the Julia set represented
as a perfect circle$)$.}
\end{figure}
\smallskip

\begin{proof}[\bf Proof of Theorem~\ref{T-balcomp}] The general
 idea of the argument can be described as follows.
Since the case $\bt=0$ is easily dealt with, we will assume that
 $\bt\ne 0$. Let $\cgap$ be the critical gap for rays from zero,
and let $\cgapp=1/2+\cgap$ 
 be the corresponding gap for rays from
infinity. Then one component $K$ of $\RZ\ssm(\cgap\cup \cgapp)$is contained
 in the open interval $(0,1/2)$ and the other, $K'=1/2+K$, is
contained in $(1/2,1)$. (Compare Figure~\ref{F-bal}.)
Thus any self-antipodal set which is disjoint
from $\cgap\cup \cgapp$ must have half of its points in $K$ and the
other half in $K'$. However, all points of $K$
correspond to internal angles in the interval $(0,\bth)$,
while all points of $K'$ correspond to internal angles in the interval
$(\bth,1)$. (Here we are assuming the convention that
$0\in \cgap$ and hence $1/2\in \cgapp$.)
In each case, this observation will lead to the
 appropriate concept of balance.
First assume that $\bt$ is rational.\smallskip

{\bf Case 1. The Generic Case.} If $\bth$ does not belong to $\bTheta$,
then $\X$ consists of a single cycle of $\m_3$ by Proposition~ \ref{p7}.
(This is the case illustrated in Figure~\ref{fV}.)
This cycle is invariant by the antipodal map, therefore has even period, hence
 $\bt$ is rational with even denominator. In addition, by symmetry, the 
number of points of $\X$ in $K$ is equal to the number of points of $\X$
 in $K'$. Therefore the number of points in $\bTheta$ in $(0,\bth)$ is
 equal to the 
number of points of $\bTheta$ in $(\bth,1)$. (In other words, the set $\X$ is 
``{\em balanced}'' with respect to $\bth$ in the sense that half of its
elements belong to the interval $(0,\bth)$ and half belong to
$(\bth,1)$.) If  $\{\bth^-,\bth^+\}\subset \bTheta$
is the unique \textbf{\textit{balanced pair}}
 for this rotation number, in the sense of
Definition \ref{D-bal}, this means that
$\bth^-\!<\bth<\bth^+$.\break
As an example, for rotation number $1/4$ with $\bTheta=\{1/15,
2/15,4/15,8/15\}$,  the balanced pair consists of $2/15$ and
 $4/15$. Thus in this case the orbit is ``balanced'' with respect
to $\bth$ if and only if 
$2/15<\bth<4/15$. 

\smallskip

{\bf Case 2.}
Now suppose that $\bth$ belongs to $\bTheta$, and that the rotation set $\X$
consists of a single periodic orbit under $\m_3$. Since this orbit is
invariant under the antipodal map $x\leftrightarrow x+1/2$, it must
have even period $p$. Furthermore $p/2$ of these points must lie in 
$(0,1/2)$
and $p/2$ in $(1/2,1)$. The image of this orbit in $\bTheta$ is
 not quite balanced with respect to $\bth$ since one of the
two open intervals $(0,\bth)$ and $(\bth,1)$ must contain only
$p/2-1$ points. However, it does follow
  that $\bth$ is one of the balanced pair
for $\bTheta$.
(Examples of such balanced pairs are $\{1/3,2/3\}$ and
$\{2/15,4/15\}$.) \smallskip

More explicitly, since $\bth$  is periodic of period $p$, it follows
that both endpoints of the critical gap $\cgap$ are also periodic
of period $p$. However, only one of these two endpoints is also visible
from infinity, and hence belongs to $\X$.\smallskip

{\bf Case 3.} Suppose that $\bth\in \bTheta$ and that $\X$ contains more than
one periodic orbit. According to \cite[Theorem 7]{G}, any $\m_3$-rotation set
with more than one periodic orbit must consist of exactly two orbits of
odd period. Thus we are in the case covered by Lemma~\ref{L-bal}. 
(This case is unique in that 
both of the  endpoints of $\cgap$
or of $\cgapp$ are visible from both zero and infinity.)\smallskip

{\bf Case 4.} Finally suppose that
 $\bt$ is irrational. The set $\X$ carries a unique probability measure 
$\mu$  invariant by $\m_3$ and the push-forward
of $\mu$ under $\bpsi_\bth$  is the unique
probability measure $\nu$ carried on $\bTheta$ and invariant by $\m_2$. By 
symmetry, $\mu(K) = \mu(K') =1/2$. Therefore, 
\[\nu [0,\bth] = \nu [\bth,1]=1/2.\]
By definition, this means that the angle $\bth$ is balanced.
\end{proof}
\bigskip

For actual computation, it is more convenient
 to work with the inverse function. 
\medskip

\begin{theorem}\label{T-inv} The inverse function
 $\bt\mapsto\bth=\brho^{-1}(\bt)$
is strictly monotone, and is discontinuous at $\bt$ if and only if 
$\bt$
is rational with even denominator. The base two expansion
 of its right hand limit $\brho^{-1}(\bt^+)=
\lim_{\eps\searrow 0}\brho^{-1}(\bt+\eps)$ can be written as
\begin{equation}\label{E-th0}
 \theta^+_\bt=\brho^{-1}(\bt^+)=\sum_{\ell=0}^\infty \frac{b_\ell}{2^{\ell+1}}
\end{equation}
where
\begin{equation}\label{E-th1}
b_\ell:=\begin{cases} 0 & {\rm if} \qquad \fract(1/2+\ell\bt)
\in[0,1-\bt),\\
1 & {\rm if} \qquad \fract\big(1/2+\ell\bt\big)\in[1-\bt,1).
\end{cases}\end{equation}
Here $\fract(x)$ denotes the fractional part of $x$, with
 $0\le\fract(x)<1$ and with $\fract(x)\equiv x \mod \Z$.
\end{theorem}
\medskip

\begin{remark}\label{R-disc}
It follows that the ``discontinuity'' of this function at $\bt$,
that is the difference $\Delta\bth$ 
between the right and left limits, is equal to
the sum of all powers  $2^{-\ell-2}$ such that 
 $1/2+\ell\bt\equiv 1-\bt({\rm mod}\Z)$. If $\bt$ is a 
fraction with even
 denominator  $n=2k$, then this discontinuity takes the value
 $$ \Delta\bth=2^{k-1}/(2^{2k}-1), $$
while in all other cases it is zero.
\end{remark}
\medskip

The proof of Theorem~\ref{T-inv} will depend on Lemma~\ref{L-bal}.
\medskip


\begin{proof}[\bf Proof of Theorem \ref{T-inv}.]
 It is not hard to check that the function defined by Equations~(\ref{E-th0})
 and (\ref{E-th1}) is strictly monotone, with discontinuities only at rationals
with even denominator, and that it increases from zero to one as $\bt$
increases from zero to one. Since rational numbers with odd denominator
are everywhere dense, it suffices to check that this expression coincides
with $\brho^{-1}(\bt)$  in the special
 case  where  \hbox{$\bt=m/n=m/(2k+1)$}
 is rational with odd denominator. In that case,
 each  $1/2+\ell\bt$ is rational with even denominator,
and hence cannot coincide with either $0$ or $1-\bt$. The cyclic 
order of the $n=2k+1$ points $t_\ell:= 1/2 + \ell\bt\in\RZ$ must be the
same as the cyclic order of the $n=2k+1$ points 
$\theta_\ell:=2^\ell\bth\in\RZ$.
This proves that $\bth$ is periodic under $\m_2$, with the required
rotation number.
Since the arc $(0,1/2)$ and $(1/2,1)$ each contain  $k$ of
 the points   $t_\ell$,
it follows that the corresponding arcs $(0,\bth)$ and $(\bth,1)$
each contain $k$ of the points  $\theta_\ell$.\end{proof}\bigskip

\section{Visible Points  for Maps outside $\HO$.}
\label{s-vis}

Next let us extend the study\footnote{This subject will be studied more
 carefully in \cite{BBM2}, the sequel to this paper.} of ``doubly visible points visible''
to maps $f_q$ which do not represent elements of $\HO$.
Fixing some  $f_q$, let $\Bo$ and
 $\Bi$ be the immediate basins of zero and infinity.

\begin{theorem}\label{T-touch}  For any map $f=f_q$ in our family, either
\begin{enumerate}
\item the two basin closures  $\overline\Bo$ and $\overline\Bi$
have a non-vacuous intersection, or else \smallskip

\item these basin closures are separated by a Herman ring $\HR$.
Furthermore, this ring has topological boundary  $\partial \HR=\BRo\cup \BRi$
with $\BRo\subset\partial\Bo$ and $\BRi\subset\partial\Bi$.
\end{enumerate}
\end{theorem}
\smallskip

\begin{proof}
Assume $\overline \Bo\cap \overline \Bi=\emptyset$. Since 
$\Chat$ is connected, there must be at least one component $A$ of
$\Chat\ssm(\overline \Bo\cup \overline \Bi)$ whose closure
intersects both
 $\overline \Bo$ and $\overline \Bi$. We will show that this set
 $A$ is necessarily an annulus.  If $A$ were simply
connected, then the boundary $\partial A$ would be a connected subset
of $\overline \Bo\cup \overline \Bi$ which intersects both of these 
sets, which is impossible. The complement of $A$ cannot have more
 than two components, since each complementary component must contain
either $\Bo$ or $\Bi$. This proves that $A$ is an annulus,
 with boundary components $\Co\subset \overline \Bo$ and $\Ci\subset 
\overline \Bi$. Evidently $A$ is unique, and hence is self-antipodal.
\smallskip

Note that $f(A)\supset A$: This follows easily from the fact
that $f$ is an open map with $f(\Co)\subset\partial \Bo$ and
$f(\Ci)\subset\partial \Bi$.
\medskip

In fact, we will prove that $A$ is a fixed Herman ring for $f$.
 The map $f$ that carries $f^{-1}(A)$ onto $A$ is a proper map of 
degree $3$.
The argument will be divided into three cases according as
 $f^{-1}(A)$ has $1$, $2$ or $3$ connected components. \smallskip

{\bf Three Components (the good case).}
In this case, one of the three components, say $A'$,
must be preserved by the antipodal map
 (because $A$ is preserved by the antipodal map).
Then $f:A'\to A$ has degree $1$ and is therefore an isomorphism.
In particular, the modulus of $A'$ is equal to the modulus of $A$.
 Since the annuli $A$ and $A'$ are both self-antipodal, they must 
intersect. (This follows from the fact that the antipodal map necessarily 
interchanges the two complementary components of $A$.) In fact $A'$
 must actually be contained
in $A$. Otherwise $A'$ would have to intersect the boundary
 $\partial A$, which is impossible since it would imply that $f(A')=A$
 intersects $f(\partial A)\subset(\overline\Bo\cup\overline\Bi)$.
Therefore, using the modulus equality,
 it follows that $A'=A$. Thus the annulus $A$ is
 invariant by $f$, which shows that it is a fixed Herman ring for $f$.
\smallskip

{\bf Two Components.} In this case, one of the two components would
map to $A$ with degree $1$ and the other would map with degree $2$.
Each must be preserved by the
 antipodal map, which is impossible since any two self-antipodal annuli
must intersect.
\smallskip

{\bf One Component.} It remains to prove that $f^{-1}(A)$ cannot be connected.
 We will use a  length-area argument that was suggested to us by Misha Lyubich.

In order to apply a length-area argument, we will need to consider conformal
metrics and the associated area elements. Setting $z=x+iy$, it will be
convenient to use the notation $\mu=g(z)|dz|^2$  for a conformal metric,
with $g(z)>0$, and the notation
$|\mu|=g(z)dx\wedge dy$ for the associated area element.
 If $z\mapsto w=f(z)$ is a holomorphic map, then the push forward of
the measure $|\mu|=g(z)dx\!\wedge\! dy$ is defined to be the measure
$$ f_*|\mu|= \sum_{\{z;f(z)=w\}} \frac{g(z)}{|f'(z)|^2}\; du\wedge dv,$$
where $w=u+iv$. Thus, for compatibility, we must define the
\textbf{\textit{
push-forward}} of the metric $\mu=g(z)|dz|^2$
 to be the possibly singular\footnote{This pushed-forward metric has poles at
critical values, and discontinuities along the image of the boundary. However,
for the length-area argument, it only needs to be measurable with finite
area. (See \cite{Ahlfors}.)}
 metric
$$ f_*\mu= \sum_{\{z;f(z)=w\}} \frac{g(z)}{|f'(z)|^2}\; |dw|^2,$$

We will need the following Lemma.
\smallskip

\begin{lem}
Let $A$ and $A'$ be two annuli in $\Chat$, 
$U \subset A'$ an open subset, 
and $f:U \rightarrow A$ a proper holomorphic map. Suppose that
a generic
curve joining the two boundary components of $A$ has a preimage by $f$ which
joins the boundary components of $A'$. Then
$${\rm mod}\;A' \leq {\rm mod}\; A$$
with equality if and only if $U=A'$ and $f:A' \rightarrow A$ is an 
isomorphism. 
\end{lem}
\smallskip

Here a curve will be described as ``generic'' if it does not pass through
any critical value. We must also take care with the word ``joining'', since
the boundary of $A$ is not necessarily locally connected. The requirement
is simply that the curve is properly embedded in $A$,
and that its closure joins the boundary components of $A$.\smallskip

\begin{proof}
Let  $\mu=g(z) |dz|^2$ be an extremal (flat) metric on $A'$.
Then
$${\rm Area}_{f_*\mu}(A) = {\rm Area}_\mu U \leq  {\rm Area}_\mu A'.$$
Let $\gamma$ be a generic curve joining the two boundary components of $A$,
and let $\gamma'$ be a preimage of $\gamma$ by $f$ joining the two
boundary components of $A'$.
\begin{eqnarray}\label{eq-aap}
{\rm Length}_{f_*\mu} \gamma &=&
 \int_{w \in \gamma} \sqrt{\sum_{f(z) = w} \frac{g(z)}{|f'(z)|^2}} \;|dw|
 \geq \int_{z\in \gamma_0'} \sqrt{|g(z)|}\; |dz| \\[.2ex]
& =& {\rm Length}_\mu \gamma_0' \geq \sqrt{{\rm Area}_\mu A' \cdot {\rm mod} A'.} \nonumber
\end{eqnarray}
Therefore
\begin{equation}\label{eq-la2}
\frac{({\rm Length}_{f_*\mu} \gamma)^2}{{\rm Area}_{f_*\mu}A}
\geq\frac{{\rm Area}_\mu A'\cdot {\rm mod} A'}{{\rm Area}_\mu A'}
={\rm mod} A'\end{equation}
Now let $\nu$ vary over all conformal metrics on $A$.
By the definion of extremal length,
\begin{eqnarray*}
{\rm mod}A &\equiv& {\rm sup}_{\nu}\;{\rm inf}_\gamma \left\{
\frac{\big({\rm Length}_\nu(\gamma)\big)^2}{{\rm Area}_\nu A}\right\} \\
&\geq& {\rm inf}_\gamma \left\{
\frac{({\rm Length}_{f_*\mu} A)^2}{{\rm Area}_{f_*\mu} A}\right\} \\
&\geq& {\rm mod}(A').
\end{eqnarray*}
If ${\rm mod}(A)={\rm mod}(A')$ then $f_*\mu$ is an extremal metric
on $A$.  For a generic curve $\gamma$ of minimal length, 
Inequality(\ref{eq-aap}) becomes an equality, so that $\gamma_0'$ is the only component of $f^{-1}(\gamma)$. As a consequence, $f:U\to A'$ has degree 1 and $f:U\to A$ is an isomorphism. According to the hypothesis of the Lemma, any curve $\gamma'\subset U$ joining the two boundary components of $U$ also joins the two boundary components of $A'$, since $\gamma'=f^{-1}\bigl(f(\gamma)\bigr)$ and $f(\gamma)$ joins the two boundary components of $A$. So, $U=A'$ and $f:U'\to U$ is an isomorphism.
Thus $f:A' \rightarrow A$ is an isomorphism.\end{proof}
\smallskip

In order to apply this lemma to the annulus $A$ of Theorem \ref{T-touch}, taking $A=A'$, we need
 the following.
\smallskip

\begin{lem}
A generic curve joining the two boundary components of $A$
 has a preimage by $f$ which joins both boundary components of $A$.
\end{lem}

(Again a ``generic'' curve means one that does not pass through a critical value,
so that each of its three preimages is a continuous lifting of the
 curve.)  \smallskip

\begin{proof}
We may assume that $\Bo$ contains only one critical point, since the case
where it contains two critical points
is well understood. (Compare \S3.) 
Thus a generic point of $\Bo$ has two preimages in $\Bo$ and a third preimage in a disjoint copy 
$\Bo'$. Similarly a generic point point in the closure  $\overline \Bo$ 
has two preimages in $\overline\Bo$ and one preimage in $\overline\Bo'$.
There is a similar discussion for the antipodal set $\Bi$.

A generic path $\gamma$ through $A$ from some point of $C_0$ to 
some point of $C_\infty$ has three preimages.
Two of the three start at points of $\overline \Bo$ and the third starts
a point of $\overline\Bo'$. Similarly, two of the three preimages end at
points of $\overline  \Bi$, and one ends at a point of
 $\overline\Bi'$.
It follows easily that at least one of the three preimages must cross from
$\overline \Bo$ to $\overline \Bi$, necessarily
 passing through $A$.
\end{proof}
\medskip

This completes the proof of Theorem \ref{T-touch}. 
\end{proof}
\medskip

\begin{coro}\label{C-vis} If the Julia set of $f_q$ is 
connected and locally connected, 
then there are points which are visible from both zero and infinity.
\end{coro}
\smallskip

\begin{proof}[\bf Proof] If the free critical point  $\cpo=\cpo(q)$ belongs to 
the basin of zero,
then this statement follows from the discussion at the beginning of
\S\ref{s-vis}, hence we may assume that $\cpo\not\in\Bo$.
Hence the B\"ottcher coordinate defines a diffeomorphism from $\Bo$ onto
$\D$. Following Carath\'eodory, local connectivity implies that the
inverse diffeomorphism extends to
 a continuous map from $\overline\D$ onto $\overline\Bo$. In particular,
it follows that every point of $\partial\Bo$ is the landing point of a
ray from zero. Since  $\J(f_q)$ is connected, there can be no Herman
 rings, and the conclusion follows easily from the theorem.\end{proof}
\medskip

\begin{definition} \label{D-mer}
By a \textbf{\textit{meridian}} in the Riemann sphere
will be meant a path from zero to infinity which is the union of an 
internal ray in the basin of zero and an external ray in the basin of infinity,
together with a common landing point in the Julia set. It follows from
Corollary \ref{C-vis} that such meridians exist whenever
the Julia set is connected and locally connected. Note that the image of a meridian
under the antipodal map is again a meridian. The union of two such
antipodal meridians is a Jordan curve $\cC$
 which separates the sphere into
two mutually antipodal simply-connected regions.
\end{definition}

\begin{remark}{\bf The Dynamic Rotation Number for $q^2\not\in\HO$.}\label{R-drn} 
For any $q\ne 0$ such that there are Julia points which are landing points of rays from both zero
and infinity, the internal rays landing on these points 
 have a well defined rotation number. The discussion in
Definition~\ref{D-drn} shows that this is true when $q^2$ belongs to the
principal hyperbolic component $\HO$. For $q^2\not\in\HO$ the proof
can be sketched as follows.
Suppose that there are $n$
meridians. Then these meridians divide the Riemann sphere into $n$
``{\bf sectors}''. It is easy to see that the two free fixed point (the landing
points of the rays of angle zero from zero and infinity) belong to
opposite sectors. Let $\theta_1$ and $\theta_2$ be the internal angles
of the meridians bounding the zero internal ray, with
 $\theta_1<0<\theta_2$; and let $\theta_1'<\theta_2'$ be the associated
angles at infinity. We claim that
\begin{equation}\label{E-wide}
\theta_2-\theta_1\ge1/2,
\end{equation}
and hence by Theorem~\ref{T-rs} that the rotation number of the internal
angles is well defined.
For otherwise we would have $2\theta_1<\theta_1<\theta_2<2\theta_2$
with total length $2\theta_2-2\theta_1<1$.
 It would follow that the corresponding
interval of external rays $[\theta_1',\theta_2']$ would map to the larger
interval $[2\theta_1',2\theta_2']\supset[\theta_1',\theta_2']$.
 Hence there would be a fixed angle in
$[\theta_1',\theta_2']$, contradicting the statement that the two free
fixed points must belong to different sectors. 

For further discussion of this dynamic rotation number, see Remark \ref{R-frn}.
\end{remark}
\medskip

Here is an important application of Definition \ref{D-mer}. 
\smallskip

\begin{theorem}\label{T-jc}
If $U$ is an attracting or parabolic basin of period two
or more for a map $f_q$, then the boundary $\partial U$ is a Jordan curve.
The same is true for any Fatou component which eventually maps to $U$.
\end{theorem}

\begin{proof}[\bf Proof] The first step is to note the the Julia set $\J(f_q)$ is
connected and locally connected. In fact, since $f_q$ has an attracting
or parabolic point of period two or more, it follows that every critical
orbit converges to an attracting or parabolic point. In particular, there
can be no Herman rings, so the Julia set is connected by Theorem~\ref{t1}.
It then follows by Tan Lei and Yin Yangcheng \cite{TY} that the Julia set
is also locally connected.

Evidently $U$ is a bounded subset of the finite plane $\C$. Hence
the complement $\C\ssm\overline U$ of the closure has a unique unbounded
component $V$. Let $U'$ be the complement $\C\ssm\overline V$.
We will first show that $U'$ has Jordan curve boundary. This set
 $U'$ must be simply-connected. 
Choosing a base point in $U'$, the hyperbolic geodesics from the base point
will sweep out $U'$, yielding a Riemann homeomorphism from the standard
open disk $\D$ onto $U'$.  Since the Julia set
is locally connected, this extends to a continuous mapping $\overline\D\to
\overline U'$ by Carath\'eodory. If two such geodesics landed at
a single point of $\overline U'$, then their union would map
 onto a Jordan curve in $\overline U'$ which would
form the boundary of an open set $U''$, necessarily contained in $U'$.
But then any intermediate geodesic would be trapped within $U''$, and hence
could only land at the same boundary point. But this is impossible by the Riesz
brothers' theorem\footnote{For the classical results used here,
see for example \cite[Theorems 17.14 and A.3]{Milnor2}.}. Therefore, the pair
 $(\overline U',\;\partial U')$ must be
homeomorphic to $(\overline\D,\;\partial\D)$.

\begin{figure}
\centerline{\includegraphics[width=1.5in]{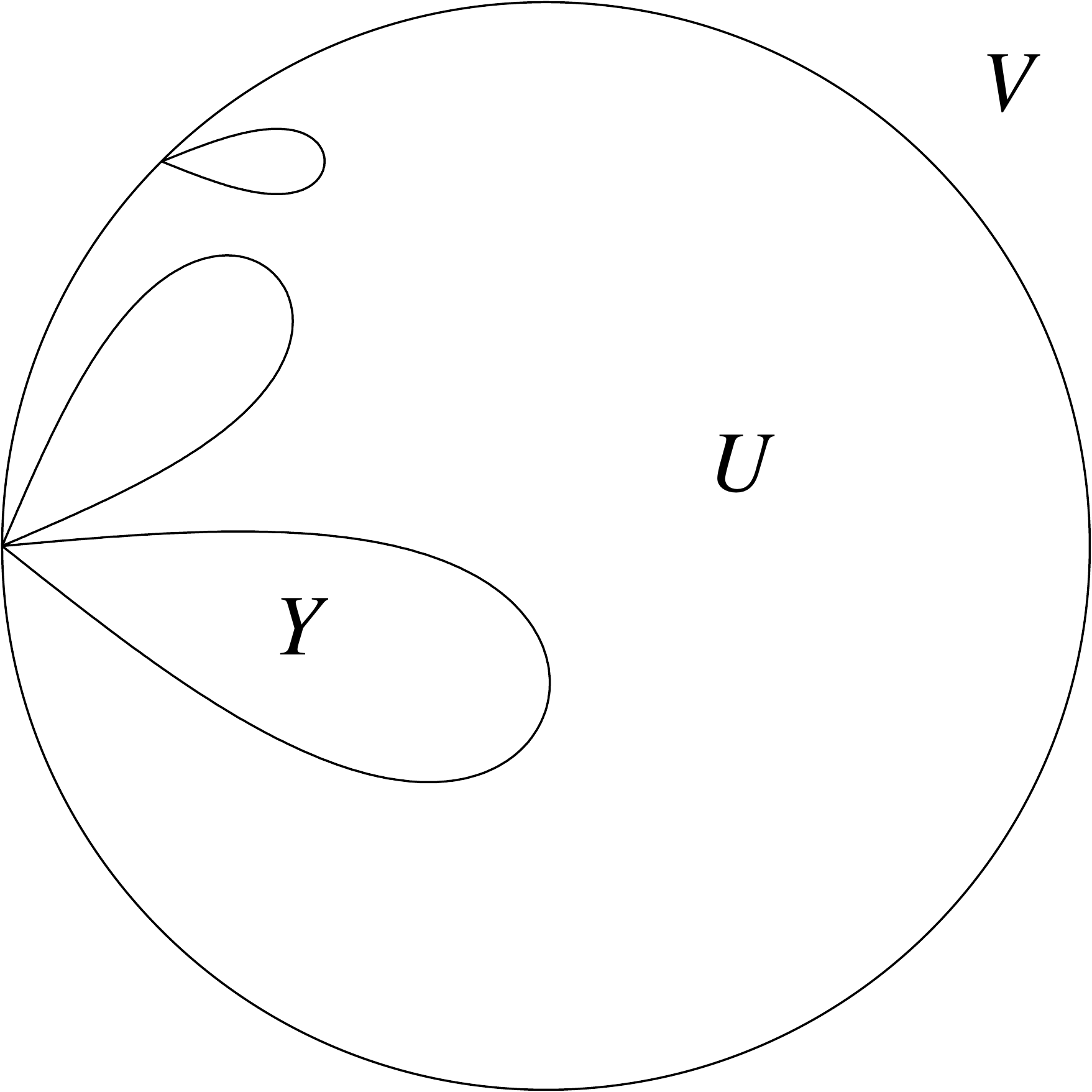}}
\caption{\label{F-ypic} \it Schematic picture, with $U'$ represented by the
open unit disk.}
\end{figure}

Thus, to complete the proof, we need only show that $U=U'$. But otherwise
we could choose some connected component $Y$ 
of $U'\ssm U$. (Such a set
$Y$ would be closed as a subset of $U'$, but its closure $\overline Y$
would necessarily 
 contain one or more points of $\partial U'$.) Note then that
 no iterated forward image of the boundary $\partial Y$ can cross the
self-antipodal Jordan curve $\cC$ of Definition \ref{D-mer}. 
 In fact any
 neighborhood of a point of $\partial Y$ must contain points of $U$.
Thus if $f_q^{\circ j}(\partial Y)$ contained points on both sides of
$\cC$, then the connected open set $f_q^{\circ j}(U)$ would also. Hence
this forward image of $U$ would have to contain points of $\cC$ belonging to
the basin of zero or infinity, which is clearly impossible. 

Note that the union  $Y\cup U$ is an open subset of $U'$. In fact
any point of $Y$ has a neighborhood which is contained in $Y\cup U$.
For otherwise every neighborhood would have to intersect
infinitely many components of $U'\ssm U$, contradicting local connectivity
of the Julia set. 

We will show that some forward image of the interior 
 $\stackrel{\circ}{Y}$ must contain the point  at
 infinity, and hence must contain some neighborhood of infinity. 
For otherwise, since $\partial Y\subset \partial U$ certainly has bounded 
orbit, it would follow from the maximum principle that $Y$ also has bounded
 orbit. This would imply that the open set $Y\cup U$ is contained
 in the Fatou set. On the other hand, $Y$ must contain points of
$\partial U$, which are necessarily in the Julia set, yielding a
 contradiction.

Similarly the forward orbit of $\stackrel{\circ}{Y}$ must contain zero. There are now
 two possibilities. First suppose that some forward image $f_q^{\circ j}(Y)$
contains just  one of the two points  zero and infinity. Then it follows
that the boundary $\partial f_q^{\circ j}(Y)\subset f_q^{\circ j}(\partial Y)$
must cross $\cC$, yielding a contradiction.

Otherwise, since an orbit can hit zero or infinity only by first hitting the
preimage $q$ or $\ant(q)$, it follows that there must be some $j\ge 0$
such that $f_q^{\circ j}(Y)$ contains both of $q$ and $\ant(q)$, but
neither of zero and infinity. In this case, it is again clear that 
$f_q^{\circ j}(\partial Y)$ must cross $\cC$. This contradiction completes
the proof of Theorem~\ref{T-jc}.\end{proof}

\bigskip

\section{Fjords}\label{s-f}
Let $\bt\in\RZ$ be rational with odd denominator.
 It will be convenient to use
the notation  $\bth=\bth_\bt\in\RZ$ for the unique angle which 
satisfies the
equation  $\brho(\bth)=\bt$. 
In other words,  $\bth$ is the unique
 angle which is balanced, with rotation number $\bt$ under doubling.
(Compare  Theorem~\ref{T-inv} and Lemma~\ref{L-bal}.)
This section will sharpen Theorem~\ref{t-inray} by proving the following result.
\smallskip

\begin{theorem}\label{T-fj}
With $\bt$ and  $\bth=\bth_\bt$ as above,
 the internal ray in $\HO$ of angle $\bth$ lands at the point
${\boldsymbol\infty}_\bt$ on the circle of points at infinity.
\end{theorem}
\smallskip

(We will see in Corollary~\ref{c-unique}
 that every other internal ray is bounded away from this point.)
 Intuitively, we should think of the $\bth_\bt$ ray as passing
to infinity \textbf{\textit{through the $\bt$-fjord}}. (Compare
 Figure~\ref{F-raydisk}.)
\medskip

In order to prove this result, it is enough
 by Theorem~\ref{t-inray} to show that the parameter ray $\pR_\bth$ has 
 no finite accumulation point on $\partial\HO$ (that is, no accumulation
point outside of the circle at infinity). In fact, we will  show that
the existence of a finite accumulation point would lead to a contradiction.
The proof will be based on the following.\smallskip

Let $\bth$ be any periodic angle such that
the $\bth$ parameter ray has a finite accumulation point
 $\qsqinf\in\partial\HO$, 
 and let $\{q^2_k\}$ be a sequence of points
 of this ray converging to $\qsqinf$.
 In order to work with specific maps, rather than
points of moduli space, we must choose a sequence of square roots
 $q_k$ converging to a square root $\q$ of
 $\qsqinf$. For each $q_k$, the internal dynamic ray
$\dR_{q_k,\bth}$ bifurcates at the critical point $\cpo(q_k)$.
However, there are well defined left and right limit rays which extend all
the way to the Julia set $\J(f_{q_k})$. Intuitively, these can be described
as the Hausdorff limits
\[\dR_{q_k,\bth^±}=\lim_{\eps\searrow 0} \dR_{q_k,\bth±\eps}\]
as $\bth±\eps$ tends to $\bth$ from the left or right.
More directly, the limit ray $\dR_{q_k,\bth^+}$ can be
constructed by starting out along  $\dR_{q_k,\bth}$ and then
 taking the left hand branch at every bifurcation point, and
$\dR_{q_k,\bth^-}$ can be obtained similarly by talking every
right hand branch.  (Compare Figure~\ref{f4a}.) It is not hard to see that the
 landing points $\bzeta_{q_k}^±(\bth)$ of these left and right limit rays are well 
defined repelling periodic points, with   period equal to the period of
 $\bth$. 
In the terminology of  Remark~\ref{R-mono} these periodic landing
 points form the end
points of the ``critical gap'' in the corresponding rotation set. Their
 coordinates around the Julia set $\J(f_{q_k})\cong\RZ$ can be computed from
Theorem~\ref{theo:describeXtheta}.
\medskip

\begin{figure}[ht]
\centerline{\includegraphics[width=2.3in]{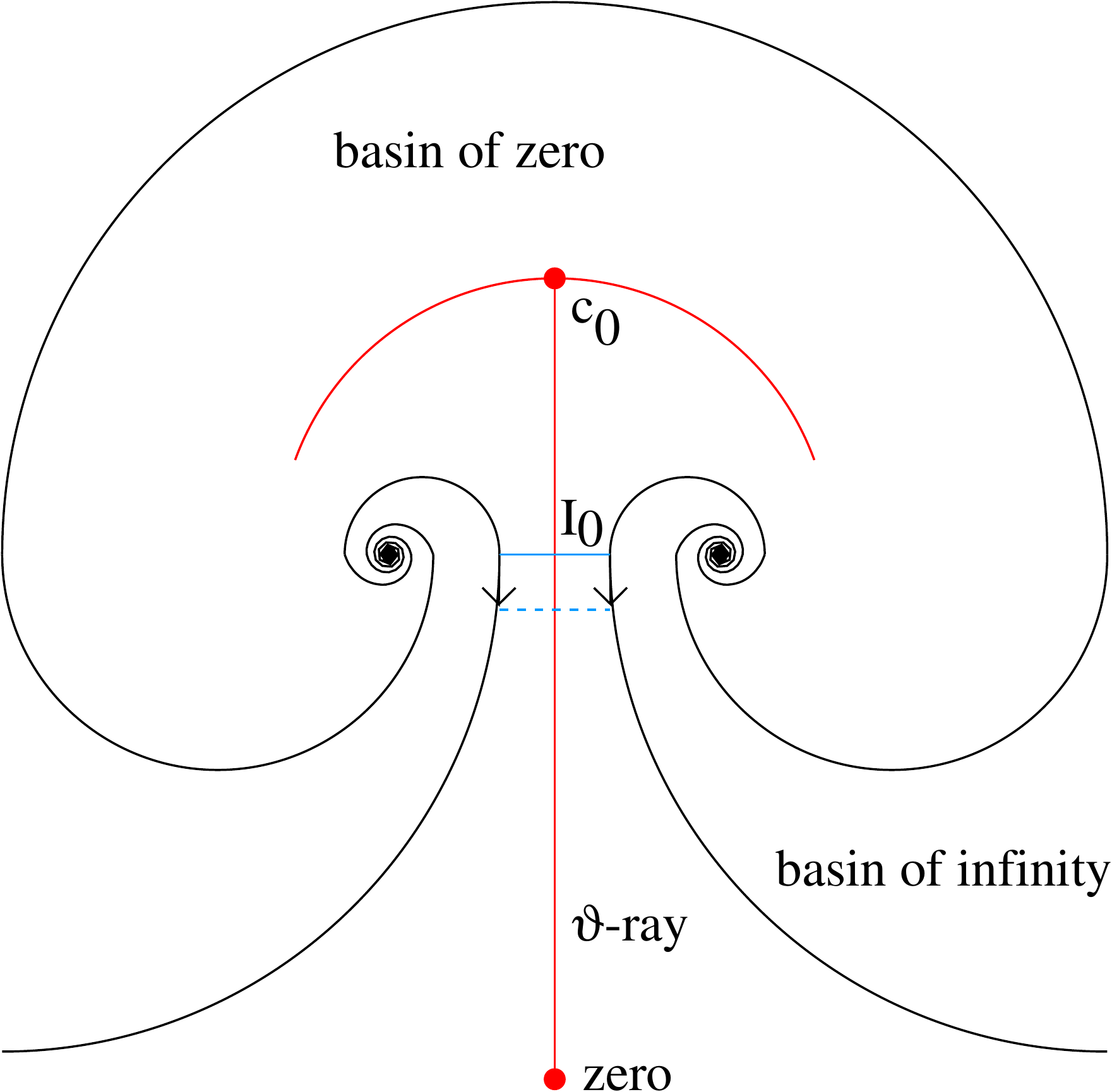}}
\caption{\label{F-par} \it Schematic picture of the Julia set for 
Lemma~$\ref{L-para}$. As the gap $I_0$ shrinks to a point, we will show that
 the two repelling
points, which are the landing points of the left and right limit rays,
must converge to this same point.}
\end{figure}
\medskip

\begin{lem}\label{L-para} Let $\bth$ be periodic under doubling,
 with odd period. Consider a sequence of points $q_k$ converging to a
 finite limit $\q$, where each $q_k^2$ belongs to the
 $\bth$ parameter ray, and where $\qsqinf$ belongs to the boundary
$\partial\HO$.
Then the corresponding
left and right limit landing points $\bzeta_{q_k}^±(\bth)\in \J(f_{q_k})$ must both
converge towards the landing point for the dynamic ray 
$\dR_{\q,\bth}$, which is a
parabolic periodic point in $\J(f_\q)$ by Theorem~$\ref{t-inray}$.
\end{lem}
\smallskip

 For an illustration of the Lemma see Figure~\ref{F-par}.
(This statement is also true for most $\bth$ with even period; but it is
 definitely false when $\bth$ is almost balanced.)
Assuming this lemma for the moment,  we can illustrate the proof of
 Theorem~\ref{T-fj} by the following explicit example.
\medskip

\begin{example}\label{E-2/7}
As in Figure~\ref{f-clov},
let $f_{q_k}$ be maps such that $s_k=q_k^2\in \HO$ belong to the $\bth=2/7$
ray, which has a  rotation number equal to $1/3$ under doubling.
 This  angle is balanced,
that is, its orbit $2/7\mapsto 4/7\mapsto 1/7$  contains one
 element in the interval $(0,2/7)$ and one element in $(2/7, 1)$, hence
 the dynamic rotation number $\bt=\bt(\bth)$ is also $1/3$.  
The dynamic ray of angle $2/7$ bounces off the critical point $\cp$,
 but it has well defined left and right limits.
These have Julia set
coordinates
$x= 6/26$ and $15/26$, as can be computed by Theorem~\ref{theo:describeXtheta}. 
Evidently these $x$-coordinates are periodic under tripling. Similarly, 
the internal ray of angle $2^n\cdot 2/7$ bifurcates off a precritical
point, and the associated left and right limit rays land at points on the Julia
set with coordinates $3^n\cdot 6/26$ and $3^n\cdot 15/26$ in $\RZ$.

Now suppose that points $s_k$  along this 
parameter ray accumulate at some finite point of $\partial\HO$,
then according to Lemma~\ref{L-para} the two landing points with coordinates
$x=6/26$ and
$15/26$ must come together towards a parabolic limit point, thus pinching
 off a parabolic basin containing $\cp$. Similarly, 
the points with $x=18/26$ and $19/26$ would merge, and also the points with
$x=2/26$ and $5/26$, thus yielding a parabolic orbit of period $3$.

Now look at the same picture from the viewpoint of the point at infinity
(but keeping the same parametrization of $\J(f_{q_k})$).
(Recall that the action of the antipodal map on the Julia set is given
by $x\leftrightarrow x+1/2\mod 1$.)  Then a similar argument shows for
example, that the points with coordinate $x=19/26$ and $2/26$ should merge.
{\it That is, all six of the indicated landing points on the Julia
set should pinch together, yielding just one fixed point, which must be
invariant under the antipodal map.} Since this is clearly impossible, it
 follows that the parameter ray of angle $2/7$ cannot accumulate on any
finite point of $\partial \HO$. Hence it must diverge to 
 the circle at infinity.
\end{example}
\medskip

\begin{proof}[\bf Proof of Theorem~\ref{T-fj} (assuming Lemma \ref{L-para})]
 The argument in the general case
 is completely analogous to the argument in Example~\ref{E-2/7}.
Let $f_q$ be a map representing a point of the $\bth$ parameter ray, where 
$\bth$ is
periodic under doubling and balanced. It will be convenient to write the
rotation number as $\bt=m/(2n+1)$, with period $2n+1$.
 Since the case $\bt=n=0$ is straightforward, we may assume
 that  $2n+1\ge 3$. The associated
left and right limit rays land at  period orbits $\cOo^±$ of
rotation  number $m/(2n+1)$ in the Julia set.
 According to Goldberg (see Theorem~\ref{T-gold}), 
any periodic orbit $\pazocal O$ which has a rotation number
is uniquely determined by this rotation number, together with the number
 $N(\cO)$  of orbit points for which the coordinate $x$
lies in the semicircle
$0<x<1/2$. (Thus the number in the complementary semicircle $1/2<x<1$ is
equal to $2n+1-N(\cO)$.
 Here $0$ and $1/2$ represent the coordinates of the two
fixed points.) Since $\bth$ is balanced  with rotation number
$\bt$ with denominator $2n+1\ge 3$,
it is not hard to check that $N(\cO)$ must be either $n$ or 
$n+1$.
Now consider the antipodal picture, as viewed from infinity. Since the antipodal
map corresponds to $x\leftrightarrow x+1/2$, we see that
$$ N(\cOi^±)=2n+1-N(\cOo^±)
=N(\cOo^\mp),$$
 Therefore, it follows from Goldberg's result that the orbit
$\cOi^±$ is identical with $\cOo^{\mp}$.
\smallskip

Now consider a sequence of points $q^2_k$ belonging
 to  the $\bth$ parameter ray, and converging to
a finite point $\qsqinf\in \partial\HO$, and choose
 square roots $q_k$ converging to $\q$.
As we approach this limit point, the right hand
limit orbit $\cOo^+$
must converge to the left hand limit orbit
$\cOo^-$ which is equal to $\cOi^+$. In other
 words,
the limit orbit must be self-antipodal. But this is impossible, since the
period of this limit orbit must divide $2n+1$, and hence be odd. Thus
the $\bth$ parameter ray has no finite accumulation points, and hence
 must diverge towards the circle at infinity.\end{proof}
\medskip
        
\begin{remark}\label{R-unbal}                                 
The requirement that $\bth$ must be balanced is essential here.
As an example of the situation when the  
angle $\bth$ has no rotation
 number or is unbalanced, see Figure \ref{F-4/7jul}. 
In this example, the internal dynamic rays with angle 1/7, 2/7, 4/7
land at the roots points of a cycle of attracting basins, and the corresponding
external rays  land at an antipodal cycle of attracting basins. These
are unrelated
to the two points visible from zero and infinity along the 1/3 and 2/3 rays.
\end{remark}

\medskip

\begin{proof}[\bf \bf Proof of Lemma \ref{L-para}] 
Let $F_q=f_q^{\circ n}$ be the $n$-fold
iterate, where $n$ is the period of $\bth$. As $q_k$ converges to 
$\q$ along the $\bth$ parameter ray, we must prove that the landing point
 $\bzeta_{q_k}^+(\bth)$ of the right hand dynamic limit ray $\dR_{q_k,\bth^+}$ must 
converge to the  landing point $z_0$ of   $\dR_{\q,\bth^+}$.
(The corresponding statement for left limit rays will follow similarly.)
\medskip

{\bf First Step (Only one fixed point).} Let
$B$ be the immediate parabolic basin containing the marked critical
point $\cpo$ for the limit
map $F_\q$. First note that there is only one fixed
point in the boundary $\partial B$. Indeed, this boundary is a Jordan
curve by Theorem~\ref{T-jc}.
 Since $n$ is odd, the cycle of
basins containing $B$ cannot be self-antipodal, hence $\cpo$ must be
 the only critical point of $F_\q$ in $B$. It follows easily that the
map $f_\q$ restricted to $\partial B$ is topologically conjugate
to angle doubling, with only one fixed point.
\medskip

{\bf Second Step (Slanted rays).}
For $q$ on the $\bth$ parameter ray, it will be convenient to
 work with \textbf{\textit{slanted rays}} in $\Bv$, that is,
smooth curves which make a constant angle with the equipotentials, or with the
dynamic rays. (Here we are following the approach of Petersen and Ryd
 \cite{PR}.) Let $\bH$ be the upper half-plane.
Using the universal covering map
$$ \E:\bH{\longrightarrow}\D\ssm\{0\},\qquad
\E(w)=e^{2\pi i w},$$
and embedding $\Bv$ into $\D$ by the B\"ottcher coordinate
 $\botq$, as in
Figure~\ref{f4a}, the dynamic rays in $\Bv$ correspond
 to vertical lines in $\bH$, and the slanted rays correspond to lines
of constant slope $dv/du$, where $w=u+iv$. (See Figure~\ref{F-sl}.)
\smallskip

\begin{figure}[ht]
\centerline{\includegraphics[width=2.5in]{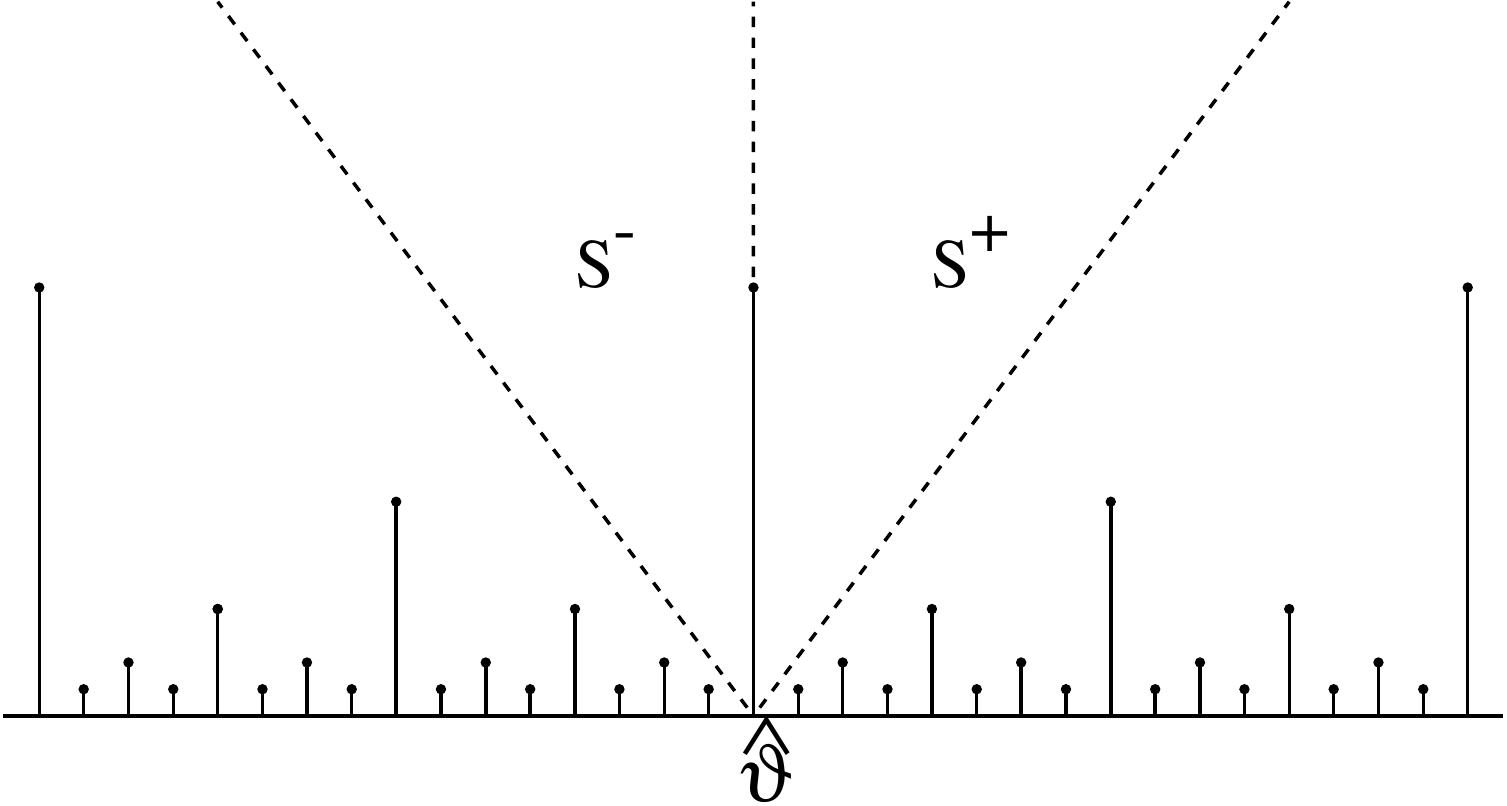}}
\caption{\label{F-sl} \it Two slanted rays in the upper half-plane. Here
the vertical dotted line represents the $\bth$ dynamic ray, and the solid
interval below it represents the corresponding invisible part of $\Ba$. }
\end{figure}
\smallskip

Let $\widehat\bth$ be a representative in $\R$ for $\bth\in\RZ$.
Fixing some arbitrary slope $\slope_0>0$ and some sign $±$, let  $\Sector^±$
be the sector consisting of all $u+iv$ in $\bH$ with
$$  ±(u-\widehat\bth)>0\qquad{\rm and}\qquad 
\frac{v}{|u-\widehat\bth|}>\slope_0.$$
If $q$ is close enough to $\q$ along the $\bth$ parameter ray,
then we will prove that the composition
 $w\mapsto \botq^{-1}\big(\E(w)\big)\in \Bv$
 is well defined throughout $\Sector^±$, yielding a holomorphic map
\begin{equation}\label{E-slits}
 \botq^{-1}\circ \E:\Sector^±\to \Bv.
\end{equation}
To see this, we must show that the sectors $\Sector^±\subset\bH$ are
disjoint from all of the vertical line segment in $\bH$ which
correspond to critical or precritical points in $\Ba$. If $\botq(\cpo)
=\E(\widehat\bth+ih)$, note that the critical slit has height $h$,
which tends to zero as $q$ tend to the limit $\q$. For a
precritical point $z$ with $f_q^{\circ k}(z)=\cpo$, the corresponding
slit height is $h/2^j$.

First consider the $n$ vertical slits corresponding to
points on the critical orbit, together with all of their integer translates.
If the $\Sector^±$ are disjoint from these, then we claim that they are disjoint
from all of the precritical slits.
For any precritical point $f_q^{\circ k}(z)=\cpo$,
we can set $k=i+jn$ with $0\le i<n$ and $j\ge 0$.
Then the image $f_q^{\circ jn}(z)$ will represent a point on the vertical
 segment corresponding to the postcritical angle $2^{i}\bth$, or to some
integer translate of this.
Using the fact
that $f_q^{\circ n}$ corresponds  to multiplication by $2^n$ modulo one,
with $\widehat\bth$ as a fixed point, we see that the endpoint of the
$z$ slit must belong to the straight line segment joining the point
$\widehat\bth\in\partial\bH$ to the endpoint of the $2^{i}\bth$ slit,
or to some integer translate of it. Thus we need impose $n$
inequalities on the slope $\slope_0$. Furthermore, for fixed $\slope_0>0$ all of
these inequalities will be satisfied for $q$ sufficiently close to
 $\q$.\medskip

{\bf Third Step (Landing).} To study the landing of slanted rays, the following
estimate will be used. Given 
a bounded simply-connected open set $U\subset\C$, 
let $d_U(z_0,z_1)$ be the hyperbolic
distance between any two points of $U$, and let $r(z_0)$ be the
Euclidean distance of $z_0$ from $\partial U$. Then
\begin{equation}\label{E-d-est}
|z_1-z_0|\le r(z_0)\Big(e^{2d_U(z_0,z_1)}-1\Big).
\end{equation}
In particular, for any fixed $h_0$, as $z$ converges to a boundary point
of $U$, the entire neighborhood $\{z'; d_U(z,z')<h_0\}$ will converge
to this boundary point. To prove  the inequality~(\ref{E-d-est}),
 note that the hyperbolic metric has the
form $\rho(z)|dz|$ where $2\rho(z)r(z) >1$. (See for example
\cite[Corollary A.8]{Milnor2}.) If $a$ is the hyperbolic arc length parameter
along a hyperbolic geodesic, so that $da=\rho(z)|dz|$, note that
$$\log\frac{ r(z_1)}{r(z_0)}=\int_{z_0}^{z_1}\frac{dr}{r}\le
\int_{z_0}^{z_1}\frac{|dz|}{r}<2\int_{z_0}^{z_1}\rho(z)|dz| =2d_U(z_0,z_1),$$
and hence 
$ r(z_1)\le~r(z_0)e^{2d_U(z_0,z_1)}$, 
or more generally
$$ r(z) \le r(z_0)e^{2\big(a(z)-a(z_0)\big)}.$$
It follows that
$$|z_1-z_0|\le\int_{z_0}^{z_1}|dz|=\int_{z_0}^{z_1}\frac{da}{\rho(z)}
<2\int_{z_0}^{z_1}r(z_0)e^{2\big(a(z)-a(z_0)\big)}da. $$
Since the indefinite integral of $2e^{2a}da$ is $e^{2a}$,
the required formula~(\ref{E-d-est}) follows easily.

We can now prove that a left or right slanted ray of any specified slope
for $f_q$ always lands, provided that $q$ is close enough to $\q$
along the $\bth$ parameter ray. Note that the map $F_q=f_q^{\circ n}$
corresponds to the hyperbolic isometry $w\mapsto 2^n w$ from $\Sector^±$
onto itself. Furthermore, the hyperbolic distance between $w$ and $2^nw$
is a constant, independent of $w$. Since the map from $\Sector^±$ to $\Ba$
reduces hyperbolic distance, it follows that $d_{\Ba}\big(z, F_q(z)\big)$
is uniformly bounded as $z$ varies. Hence as $z$ tends to the boundary
of $\Ba$, the Euclidean distance $|F_q(z)-z|$ tends to zero. It follows
that any accumulation point of a slanted ray on the boundary $\partial U$
must be a fixed point of $F_q$. Since there are only finitely many fixed
points, it follows that the slanted ray must actually land.\smallskip

Note that the landing point of the left or right slanted ray is identical to
 the landing point of the left or right limit ray.
Consider for example the right hand slanted ray in Figure~\ref{F-sl}. Since
 it crosses every vertical ray to the right of $\bth$, its landing point
cannot be  to the right of the landing point of the limit vertical ray.
On the other hand, since the slanted ray is to the right of the limit ray,
its landing point can't be to the left of the limit landing point.
\medskip

{\bf Fourth Step. (Limit rays and the Julia set).}
 Now consider a sequence of points $q_k$ on the $\bth$
parameter ray converging to the landing point $\q$. After
passing to a subsequence, we can assume that the landing points of the
associated left or right slanted rays converge to some fixed point of
$F_\q$. Furthermore, since the non-empty compact subsets of
$\Chat$ form a compact metric space in the Hausdorff topology,
we can assume that this sequence of slanted rays $\sR_{q_k}^±$
 has a Hausdorff limit, which will be denoted by
 $\sRlim^±$.  
We will prove the following statement.

\begin{quote}
{\it Any intersection point of $\sRlim^±$ with the Julia set
 $\J(F_\q)$ is a fixed point of $F_\q$.}
\end{quote}

\noindent To see this, choose any $\eps>0$. Note first that the 
intersection point
 $z_0$ of $\sRlim^±$ and $\J(F_\q)$ 
can be
$\eps$-approximated by a periodic point $z_1$ in $\J(f_\q)$.
We can then choose $q_k$ close enough to $\q$ so that the
corresponding periodic point $z_2\in \J(f_{q_k})$ is $\eps$-close
 to $z_1$. On the other hand, we can also choose $q_k$ close enough 
so that
the ray $\sR^±_{q_k}$ is $\eps$-close to
 $\sRlim^±$. In particular, this implies that there
 exists a point $z_3\in \sR_{q_k}^±$ which is $\eps$-close
 to $z_0$.

Then $z_3$ is $3\eps$-close to a point $z_2\in \J(f_{q_k})$, so
it follows from Inequality~(\ref{E-d-est}) that the distance
 $|F_{q_k}(z_3)-z_3|$ tends to zero as $\eps\to 0$.  Since $z_3$
is arbitrarily close to $z_0$ and $F_{q_k}$ is uniformly arbitrarily
close to $F_\q$, it follows by continuity that $F_\q(z_0)
=z_0$, as asserted.
\medskip

{\bf Fifth Step (No escape).}
 Conjecturally there is only one parabolic basin $B$
with root point at the the landing point of dynamic ray $\dR_{\q,\bth}$. Assuming
this for the moment, we will prove that the Hausdorff limit ray
$\sRlim\pm$ 
 cannot escape from $\overline B$.
Since this limit ray intersects $\partial B$ only at the root point by the
Step 1 of this proof, it
could only escape through this root point. In particular, it would have to
escape through the repelling petal $P$ at the root point.

Let $\Phi: P\to\R$ be the real part of the repelling
Fatou coordinate, normalized so  that
 $\Phi\big(F_\q(z)\big)=\Phi(z)+1$.
Thus $\Phi(z)\to -\infty$ as $z$ tends to the root point through $P$.

Let $\sigma_k:\R\to\C$ be a parametrization
 of the slanting ray $\sR^±_{q_k}$, normalized so that
$F_{q_k}\big(\sigma_k(\tau)\big)=\sigma_k(\tau+1)$. Thus 
for $q_k$ close to $\q$,  the composition
$\tau\mapsto\Phi\big(\sigma_k(\tau)\big)$ should satisfy
$$ \Phi\big(\sigma_k(\tau+1)\big)=\Phi\Big(F_{q_k}\big(\sigma_k(\tau)\big)\Big)
\approx\Phi\Big(F_\q\big(\sigma_k(\tau)\big)\Big)=
\Phi\big(\sigma_k(\tau)\big)+1 $$
for suitable $\tau$ with $\sigma_k(\tau)\in P$. In particular, the estimate
$$ \Phi\big(\sigma_k(\tau+1)\big)>\Phi\big(\sigma_k(\tau)\big) $$
is valid, provided that $\tau$ varies over a compact interval $I$,
and that $|q_k-\q|$ is less than some $\eps$ depending on $I$.

Now let us follow the slanted ray from the origin, with $\tau$ decreasing from
$+\infty$. If the ray enters the parabolic basin $B$, and then escapes
into the petal $P$, then the function $\Phi\big(\sigma_k(\tau)\big)$
must increase to a given value $\Phi_0$ for the first time as $\tau$
decreases. But it has already reached a value greater than $\Phi_0$
for the earlier value $\tau+1$. It follows easily from this contradiction
that the limit ray $\sRlim$ can never escape from
 $\overline B$

This completes the proof in the case
that there is only one basin $B$ incident to the root point. If there are
several incident
 basins $B_1,\ldots,B_k$, then a similar argument would show
that the limit ray cannot escape from $\overline B_1\cup\cdots\cup \overline
B_k$, and the conclusion would again follow. This completes the proof of
 Lemma~\ref{L-para} and Theorem~\ref{T-fj}.
 \end{proof}
\bigskip

\begin{remark}{\bf Formal rotation number:
the parameter space definition.} \label{R-frn} 
As a corollary of this theorem, we can describe a concept of rotation number
which makes sense for {\it all} maps $f_q$ with $q\ne 0$. (Compare the
dynamic rotation number which  is defined for points in $\HO$ 
in Definition~\ref{D-drn}, and for many points outside of $\HO$ in 
Remark~\ref{R-drn}.)

Let  $I\subset \RZ$ be a closed interval such that the endpoints
 $\bt_1$ and $\bt_2$ are rational numbers with odd denominator,
and let $\tri(I)\subset{\CC}$ 
 be the compact triangular region bounded
 by the unique rays which join $\binf_{\bt_1}$
 and $\binf_{\bt_2}$ to the origin, together with the arc
 at infinity consisting of all points $\binf_\bt$ with
 $\bt\in I$. (Compare Figure~\ref{F-raydisk}.) 
\medskip

\begin{figure}[htpb!]
\centerline{\includegraphics[width=3in]{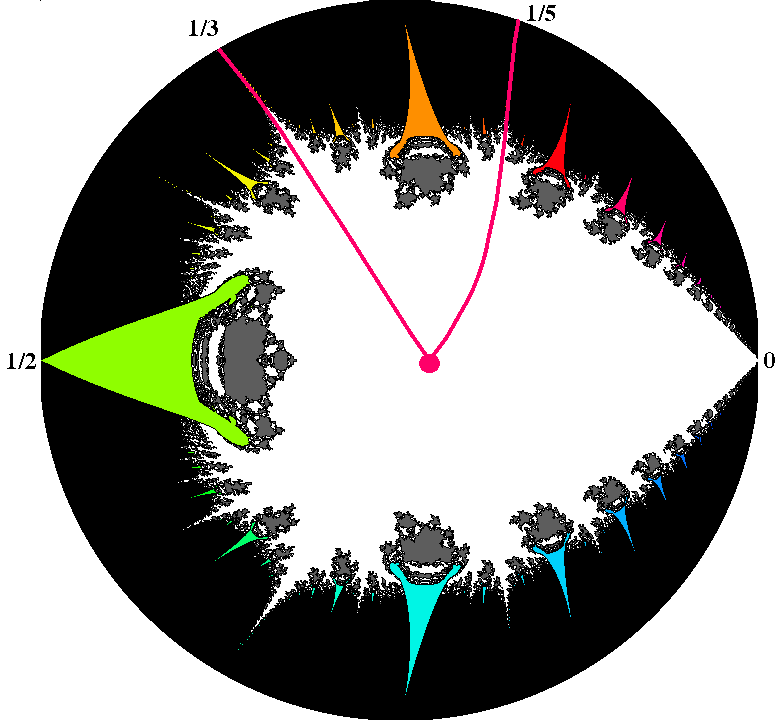}}
\caption{\label{F-raydisk} \it  Unit disk model for the circled $q^2$-plane, 
with  two rays from the origin to the circle at infinity sketched in.}
\end{figure}

For any angle $\bt\in\RZ$, let $\tri_\bt$ be the 
 intersection of the compact sets ${\tri}(I)$ as $I$ varies over all closed
intervals in $\RZ$ which contain $\bt$ and have odd-denominator
rational endpoints. 

\medskip
\begin{lem}\label{L-Del}
Each $\tri_\bt$ is a compact connected set which intersects the circle
at infinity only at the point $\binf_\bt$. These sets
$\tri_\bt$ intersect at the origin, but otherwise are pairwise disjoint.
In particular, the intersections $\tri_\bt\cap(\CC\ssm \HO)$
are pairwise disjoint compact connected sets. $($Here it is essential
that we include points at infinity.$)$
\end{lem}
\smallskip

\begin{proof}[\bf Proof] The argument is straightforward, making use of the fact that
$\HO$ is the union of an increasing sequence of open sets bounded
by equipotentials.\end{proof}\medskip

\begin{definition}\label{d-frn2}
Points in $\tri_\bt\ssm\{0\}$ have \textbf{\textit
{formal rotation number}} equal to $\bt$. It is sometimes convenient
to refer to these as points in the $\bt$-\textbf{\textit{wake}}. 
\smallskip

It is not difficult to check that
this is compatible with  Definition~\ref{D-drn} in the special case of 
points which belong to $\HO$. Furthermore, 
by Lemma~\ref{lem-psi} this formal rotation number is
continuous as a map from $\CC\ssm\{0\}$ to $\RZ$ (or equivalently as
a continuous retraction of $\CC\ssm\{0\}$ onto the circle at infinity).
\end{definition}
\end{remark}
\medskip

One immediate consequence of Lemma~\ref{L-Del} is the following.
\smallskip

\begin{coro}\label{c-unique}
For $\bt$ rational with odd denominator, the ray  $\pR_{\bth_\bt}$ of 
Theorem~\ref{T-fj} is the only internal ray in $\HO$ which accumulates on the
point $\binf_\bt$.
\end{coro}
\smallskip

In fact, for any $\bt'\ne \bt$ the ray $\pR_{\bth_{\bt'}}$ is 
 contained in a compact set $\tri_{\bt'}$
which does not contain  $\binf_\bt$.\qed
\medskip  

\begin{remark}{\bf Accessibility and Fjords.}\label{r-admfjord}
A topological boundary point 
$$q^2_0\;\in\;\partial\HO\;\subset\;\CC$$
is said to be \textbf{\textit{accessible}} from $\HO$ if
 there is a continuous path $p:[0,1]\to\CC$ such that $p$ maps the
half-open interval $(0,1]$ into $\HO$, and such that $p(0)=q^2_0$.
Define a \textbf{\textit{channel}} to $q^2_0$ within $\HO$ to be a
function ${\mathfrak c}$ 
which assigns to each
open neighborhood $U$ of $q^2_0$ in $\CC$ a connected component
${\mathfrak c}(U)$ of the intersection $U\cap\HO$  satisfying the
condition that 
$$ U\supset U'\quad\Longrightarrow\quad {\mathfrak c}(U)\supset
{\mathfrak c}(U').$$ 
(If $U_1\supset U_2\supset\cdots$ is a basic set of neighborhoods shrinking
down to $q^2_0$, then clearly the channel $\mathfrak c$ is uniquely
determined by the sequence ${\mathfrak c}(U_1)\supset\mathfrak c(U_2)\supset
\cdots$. Here, we can choose the $U_j$ to be simply connected, so that the
${\mathfrak c}(U_j)$ will also be simply connected.)
\end{remark}\medskip

By definition, the access path $p$ lands on $q^2_0$ 
\textbf{\textit{through}}
 the channel ${\mathfrak c}$ if for every neighborhood $U$ there is
an $\eps>0$ so that $p(\tau)\in{\mathfrak c}(U)$ for $\tau<\eps$.
It is not hard to see that every access path determines a unique
channel, and that for every channel $\mathfrak c$
 there exits one and only one homotopy class of access paths which land at
 $q^2_0$ through $\mathfrak c$.
\smallskip

The word \textbf{\textit{fjord}} will be reserved for a channel
to a point on the circle at infinity
which is determined by an internal ray which lands at that point.
\medskip 

\begin{remark}{\bf Are there Irrational Fjords?.} \label{r-irrfjord} It seems 
likely,
that there are uncountably many irrational rotation numbers $\bt\in\RZ$
such that the ray  $\pR_{\bth_\bt}\subset{\pazocal H}_0$
lands at the point $\binf_\bt$
 on the circle at infinity. In this case, we would say that
the ray lands through an \textbf{\textit{irrational fjord}}.
 What we can actually prove is the following statement.
\end{remark}\bigskip

\begin{lem}\label{l-ir-fj}
 { There exist countably many dense open sets
$W_n\subset\RZ$ such that, for any $\bt\in\bigcap W_n$,
the  ray $\pR_{\bth_\bt}$
accumulates at the point
$\binf_\bt$ on the circle at infinity.}\end{lem}
\medskip

However, we do not know how to prove that this ray actually lands at
$\binf_\bt$. The set of all accumulation points for the
 ray is
necessarily a compact connected subset of $\tri_\bt\cap\partial\HO$,
but it could contain more than one point.
In that case, all neighboring fjords and tongues would have to undergo very
wild oscillations as they approach the circle at infinity.\bigskip

\begin{proof}[\bf Proof of Lemma \ref{l-ir-fj}]
Let $V_n$ be the open set consisting
 of all  internal angles $\bth$ such that the internal ray $\pR_{\bth}$
contains points $q^2$ with $|q^2|>n$. 
A corresponding open set $W_n$ of formal rotation numbers
can be constructed as follows. For each
rational $r\in\RZ$ with odd denominator,
choose an open neighborhood which is small enough so that every associated
internal angle is contained in $V_n$. The union of these open
neighborhoods will be the required dense open set $W_n$. For any $\bt$ in
 $\bigcap W_n$, the associated internal ray  $\pR_{\bth_\bt}$
will come arbitrarily close
to the circle at infinity. Arguing as in Corollary~ \ref{c-unique}, we see that
$\binf_\bt$ is the only point of $\partial \CC$
 where this ray can accumulate.\end{proof}\bigskip

If the following is true, then we can prove a sharper result.
 It will be convenient to use the notation $\L$
for the compact set $\CC\ssm\HO$, and the notation
$\L_\bt$ for the intersection $\L\cap\tri_\bt$.
Thus $\L$ is the disjoint union of the compact connected
sets $\L_\bt$. (It may be convenient to refer to 
 $\L_\bt$ as the $\bt$-\textbf{\textit{limb}} of $\L$.)
\smallskip

\begin{conj}\label{c-inf} If $\bt$ is rational with odd denominator, then
the set $\L_\bt$ consists of the single point
 $\binf_\bt$.
 \end{conj}\medskip

Assuming this statement, we can prove that a generic ray actually
 lands, as follows. Let $\overline\D_n$ be the disk consisting of
all $s\in\C$ with
 $|s|\le n$, and let  $W_n'\subset\RZ$ be the  set consisting 
of all $\bt$ for which $\L_\bt\cap\overline\D_n=\emptyset$.
 The set $W_n'$ is open for all $n$,  since  $\L$ is
 compact.
Then $W_n'$ is dense since it contains all rational numbers with odd
denominator. For any $\bt\in\bigcap_n W_n'$, the set $\L_\bt$
evidently consists of the single point $\binf_\bt$, and it
follows easily that the associated ray $\pR_{\bth_\bt}$ lands at this
 point.\qed\bigskip

\appendix
\section{Dynamics of monotone circle maps}

\begin{definition}
Any continuous map $g:\RZ\to\RZ$ will be called a {\bf circle map}.
For any circle map, there exists a continuous \textbf{\textit{lift}},
$\widehat g:\R\to\R$, unique up to addition of an integer constant, such that the
square
$$\xymatrix@=2.8pc{
\R \ar[d]\ar[r]^{\widehat g} & \R\ar[d]\\
\RZ \ar[r]^g & \RZ}$$
is commutative. Here the vertical arrows stand for the natural map
from $\R$ to $\RZ$.  This lift satisfies the identity
$$\widehat g(y+1)=\widehat g(y) +d,$$
where $d$ is an integer constant called the \textbf{\textit{degree}}
 of $g$.
We will only consider the case of circle maps of positive degree. 
The map $g$ will be called \textbf{\textit{monotone}} if
$$ y<y'\quad\Longrightarrow\quad \widehat g(y)\;\le\; \widehat g(y').$$
It is not hard to see that $f$ has degree $d$ if and only
 if it is homotopic to the standard map
\[\m_d:\RZ\to\RZ\qquad\text{defined by}\qquad \m_d(x)=dx.\]
In this appendix, we will study some compact subsets of $\RZ$ which are invariant under $\m_d$. 
If $\I\subset \RZ$ is an open set, we define
\[X_d(\I)=\bigl\{x\in \RZ\text{ such that }\m_d^n(x)\not \in \I\text{ for all }n\geq 0\bigr\}.\]
\end{definition}\medskip

\subsection{Degree one circle maps}

We first review the classical theory of rotation numbers for monotone
degree one circle maps.

It is well known that for any monotone circle map $g$ of degree $d=1$, we can define a
\textbf{\textit{rotation number}} $\rot(g)\in\RZ$, as follows.
Choose some lift $\widehat g$. We will first
prove that the \textbf{\textit{translation number}}
$$ \trans(\widehat g)= \lim_{n\to\infty} \frac{\widehat g^{\circ n}(x)-x}{n} \in\R$$
is well defined and independent of $x$. For any circle map of degree one,
setting  $\underline a_n
=\min_x\big(\widehat g^{\circ n}(x)-x\big)$, the inequality
$\underline a_{m+n}\ge\underline a_m+\underline a_n$ is easily verified.
It follows that $\underline a_{km+n}\ge k\underline a_m+\underline a_n$.
Dividing by $h=km+n$ and passing to the limit as $k\to\infty$ for fixed
$m>n$, it follows that $\liminf \underline a_h/h\ge \underline a_m/m$.
Therefore the {\it lower translation number}
$$ \underline\trans(\widehat g)= \lim_{n\to\infty} \underline a_n/n
=\sup_n \underline a_n/n $$
is well defined. Similarly, the {\it upper translation number}
$$ \overline\trans(\widehat g)= \lim_{n\to\infty} \overline a_n/n
=\inf_n \overline a_n/n $$
is well defined, where 
$\overline a_n={\rm max}_x\big(\widehat g^{\circ n}(x)-x\big)$.
Finally, in the monotone case it is not hard to check that
$0\le \overline a_n-\underline a_n\le 1$, and hence that
 these upper and lower translation numbers are equal. By definition,
the \textbf{\textit{rotation number}} $\rot(g)$ is equal to
the residue class of $\trans(\widehat g)$ modulo $\Z$.
\medskip

\begin{remark}\label{R-rn} Note that the rotation number is rational if and
only if $g$ has a periodic orbit. First consider the rotation number zero
case. Then the translation number $\trans(\widehat g)$ for a lift $\widehat g$ will be
 an integer. The periodic function 
 $x\mapsto \widehat g(x)-x-\trans(\widehat g)$  has translation number zero. 
We claim that it must have a zero. Otherwise, if
for example its values were all greater than  $\eps>0$, 
then it would follow
easily that the translation number would be greater than or equal to $\eps$.
The class modulo $\Z$ of any such zero will be the required fixed point of $g$. 

To prove
the corresponding statement when the rotation number is rational with
 denominator $n$, simply apply the same argument to the $n$-fold iterate
$g^{\circ n}$. Each fixed point of $g^{\circ n}$ will be a periodic point
 of $g$. In addition,  the cyclic order of each cycle \footnote{
Three points $x_1,x_2,x_3\in\RZ$ are in  
\textbf{\textit{positive cyclic order}} if we can choose lifts 
$\widehat x_j
\in\R$ so that $\widehat x_1<\widehat x_2<\widehat x_3<\widehat x_1+1$.}
 around the circle
is determined by the rotation number. Indeed, if $\widehat g:\R\to \R$ is a lift with translation
 number $m/n\in \Q$, then the map 
 \[\frac{1}{n}\sum_{k=0}^{n-1} \left( \widehat g^{\circ k} - \frac{km}{n}\right)\]
 descends to a monotone degree one map $\psi:\RZ\to \RZ$ which
 conjugates $g$ to the rotation of angle $m/n$ on periodic cycles of $g$. 
 
 \smallskip

In the irrational case, one obtains much sharper control. {\it It is not 
hard to see 
that $g$ has irrational rotation number $\bt$ if and only if there
is a degree one monotone semiconjugacy $\psi:\RZ\to\RZ$ satisfying
\[\psi\big(g(x)\big)=\psi(x)+\bt.\]
More precisely, $\widehat g:\R\to \R$ is a lift with translation
 number $\widehat \bt\in \R\ssm \Q$ if and only if the sequence of maps
 \[x\mapsto \frac{1}{n} \sum_{k=0}^{n-1}\bigl(\widehat g^{\circ k}(x) - \widehat g^{\circ k}(0)\bigr)\]
 converges uniformly to a semiconjugacy $\widehat \psi$ between $\widehat g$ and 
 the translation by $\widehat \bt$.}
\end{remark}

\subsection{Rotation sets}

\medskip
\begin{definition}{\bf Rotation Sets.}\label{D-rnth}
Fixing some integer $d\ge 2$, let $X\subset\RZ$ be a non-empty
compact subset which is invariant under multiplication by $d$
 in the sense that  $\m_d(X)\subset X$, where 
$$ \m_d:\RZ\to\RZ\qquad{\rm isthemap}\qquad \m_d(x)=dx.$$
Then $X$ will be called a \textbf{\textit{rotation set}} if the map
$\m_d$ restricted to $X$ extends to a continuous monotone map
$g_X:\RZ\to\RZ$ of degree $1$. 
The \textbf{\textit{rotation number}} of $X$
is then defined to be the rotation number of this extension $g_X$.
It is easy to check that this rotation number does not depend on the choice
of $g_X$ since the lifts of two extensions to $\R$ may always be chosen to coincide on the lift of $X$. 
\end{definition}
\smallskip

Here is a simple characterization of rotation sets.
\smallskip

\begin{theorem}\label{T-rs}
A non-empty compact set $X\supset \m_d(X)$ is a rotation set if and only if
the complement $(\RZ)\ssm X$ contains $d-1$ disjoint open
intervals of length $1/d$.
\end{theorem}
\smallskip

The proof will depend on two lemmas.

\begin{lem}\label{L-XJ} Let $\I\subset\RZ$ be an open set which is
the union of disjoint open subintervals 
$\I_1,\I_2,\cdots,\I_{d-1}$,
each of length precisely $1/d$. Then $X_d(\I)$ is a  rotation
 set for $\m_d$.\end{lem}
\smallskip

\begin{proof}[\bf Proof] The required monotone degree one map $g=g_{_\I}$ will coincide
 with $\m_d$ outside of $\I$,\; and will map each subinterval $\I_i$
 to a single point: namely the image under $\m_d$ of its two endpoints.
The resulting map $g_{_\I}$ is clearly monotone and continuous,
 and has degree one
since the complement $(\RZ)\ssm\I$ has length exactly $1/d$. It follows
easily from Remark \ref{R-rn} that the set  
$X_d(\I)$ is non-empty. In fact, if $g_\I$ has an attracting periodic orbit,
then it must also have a repelling one, which will be contained in $X_d(\I)$;
while if there is no periodic orbit, then the rotation number is irrational,
and there will be uncountably many orbits in $X_d(\I)$.
\end{proof}\bigskip

Thus the rotation number of $X_d(\I)$ is well defined.
If the rotation number is non-zero, note that each of the $d-1$
 intervals $\I_i$ must contain exactly one of the $d-1$ fixed points of
$\m_d$.
\smallskip

\begin{remark}\label{R-mono}
If we move one or more of the intervals $\I_i$ to the right 
by a small  
distance  $\eps$, as illustrated in Figure~\ref{f7a},
then it is not hard to check that for each $x\in\RZ$ the image
$g(x)$ increases by at most $\eps d$. Hence the rotation number
$\rot\big(X_d(\I)\big)$ also increases
by a number in the interval $[0,\eps d]$.  
\end{remark}
\medskip

\begin{figure}[ht]\label{GX-graf}
\centerline{\includegraphics[height=2.5in]{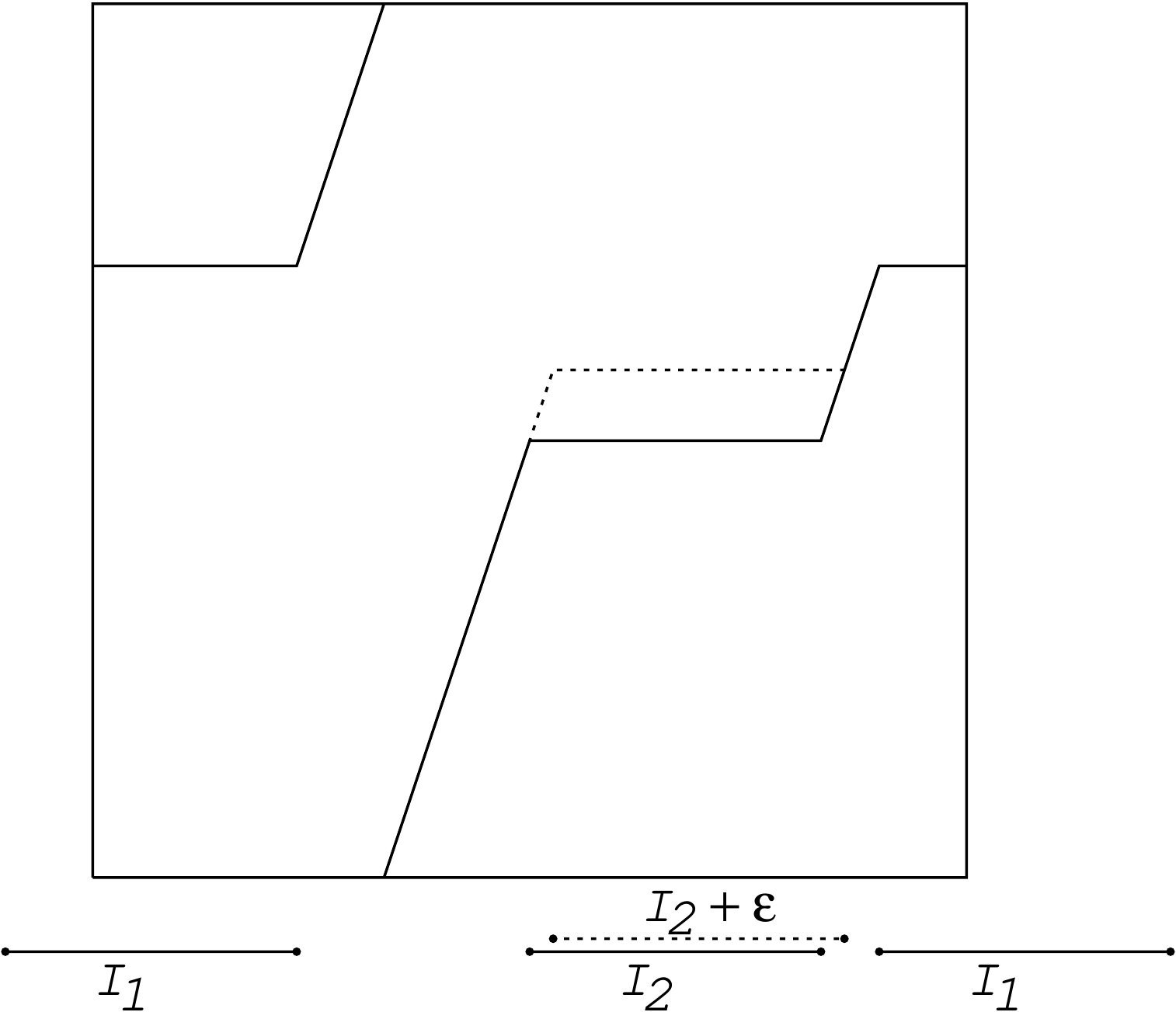}}
\caption{\label{f7a} \it A typical graph for $g_{_{\I}}$ in the case $d=3$.}
\end{figure}

\medskip
By a \textbf{\textit{gap}} $G$ in $X$ will be
 meant any connected component of the complement $(\RZ)\ssm X$.
The notation $0<\ell(G)\le 1$ will be used for the length of a gap.
The \textbf{\textit{multiplicity of a gap}} $G$
 is defined as the integer part of the product $d\ell(G)$, denoted by
$$\mult(G)= \floor\big(d \ell(G)\big).$$
 If this multiplicity is non-zero, or in other words if
$\ell(G)\ge 1/d$, then 
$G$ will be called  a \textbf{ \textit{major gap}}.
\medskip

\begin{lem}\label{L-gaps} Suppose that $X\supset \m_d(X)$ is a non-empty compact set. 
Then $X$ is a rotation set if and only if
it has exactly $d-1$ major gaps, counted with multiplicity.\footnote
{Compare \cite{BMMOP}}
\end{lem}

\begin{proof}[\bf Proof]  First note that every rotation set  must have
 measure zero. This follows from the fact that
 the map $m_d$ is ergodic, so that almost every orbit is everywhere dense.
Let $\{\ell_i\}$ be the (finite or countably infinite)
sequence of lengths of the various gaps. Thus $\sum\ell_i=1$.
The $i$-th gap must map
under $g_X$ either to a point (if the product $d\ell_i$ is an integer),
 or to an interval of length equal to the fraction part 
$$\fract(d\ell_i)=d\ell_i-\floor(d\ell_i).$$
It follows easily that $\sum\fract(d\ell_i)=1$. Therefore
$$ \sum\ell_i=1=\sum(d\ell_i) - \sum\floor(d\ell_i),$$
which implies that $\sum\floor(d\ell_i)=d-1$.
This proves Lemma~\ref{L-gaps}; and Theorem~\ref{T-rs} follows easily.
\end{proof}
\medskip


\begin{definition}{\bf Reduced Rotation Sets.}\label{D-red}
A  rotation set $X$ will be called
\textbf{\textit{reduced}} if there are no wandering points in $X$,
that is, no points with a
neighborhood $U$ such that $U\cap X$  is disjoint from all 
of its iterated forward images under $\m_d$.

 For any 
rotation set $X$, the subset $X_{\rm red}\subset X$
consisting of all non-wandering points in $X$, forms a reduced rotation
set with the same rotation number. (In particular, $X_{\rm red}$ is always
compact and non-vacuous.)
Here is an example: If we take  $d=2$ and $\I=(1/2,1)$ in Lemma \ref{L-XJ},
then the  rotation set $X=X_2(\I)$ is given by
$$X=\{1/2,1/4, 1/8,1/16,\ldots, 0\}.$$
The associated reduced rotation set is just the single point
$X_{\rm red}=\{0\}$.\smallskip
\end{definition}

\begin{lem}\label{L-red}
For any  rotation set $X$ with rational rotation number, the associated
reduced set $X_{\rm red}$ is  a finite union of periodic orbits. 
In the irrational case, $X_{\rm red}$ is a Cantor set, and is minimal
 in the sense that every orbit is dense in $X_{\rm red}$.
\end{lem}

\begin{proof}
First consider the case of zero rotation number. Then $X$
necessarily contains a fixed point, which we may take to be zero.
 (Compare the proof of Lemma \ref{L-XJ}.) 
Then the associated monotone circle map $g_X$
 can be thought
of as a monotone map from the interval $[0,1]$ onto itself. Clearly the only
non-wandering points of $g_X$ are fixed points, so it follows that
$X_{\rm red}$ contains only fixed points under the map $\m_d$.
For any  rotation set with rotation number $m/n$, we can apply
this same argument using the $n$-fold iterate $\m_{d^n}$,
 and conclude that 
the reduced set $X_{\rm red}$ must be a finite union of periodic
orbits. 

Now suppose that the rotation number $\bt$ is irrational, and that
$X=X_{\rm red}$ is reduced. Then $X$ 
can have no isolated points, since an isolated point would necessarily be
wandering. Hence it must be a Cantor set. 
Now form a topological circle $X/\!\!\simeq$ by setting
 $x\simeq y$ if and only if
$x$ and $y$ are the endpoints of some gap. Then the map $\m_d|_X$ gives
rise to a degree one circle homeomorphism $g:X/\!\!\simeq\to X/\!\!\simeq$
with the same irrational rotation number. As in Remark~\ref{R-rn},
there is a monotone degree one semiconjugacy $h:X/\!\simeq\to\RZ$ 
satisfying
$h\big(g(x)\big)=h(x)+\bt$. In fact $h$ must be a homeomorphism. For if
some open interval  $I\subset X/\!\!\simeq$ maps to a point, then all of the
points in $I$ would be wandering points for $g$, and hence all of the
associated points of $X$ would also be wandering points. Thus every orbit
in $X/\!\simeq$ is dense, and it follows easily that every orbit in $X$ is dense. \end{proof}
\medskip

\begin{remark}{\bf Rotation Sets under Doubling.}\label{R4}  
This construction is particularly transparent in the case $d=2$. In fact, let
$\I_c$ be the interval $(c,c+1/2)$ in $\RZ$.
According to Remark  \ref{R-mono},  the correspondence
$$ c \mapsto \rot\bigl(X_2(\I_c)\bigr) \in \RZ $$
can itself be considered as a monotone circle map of degree one. Hence for
each  $\bt\in\RZ$ the collection of $c$ with
 $\rot\bigl(X_2(\I_c)\bigr)=\bt$ is
non-vacuous: either a point or a closed interval.
Let $X$ be a reduced rotation set with rotation number $\bt$. If 
$\bt$ is rational with denominator $n$, then each gap in $X$ must
 have period $n$ under the associated monotone map $g_X$. Since there is
 only one major gap there can be only one cycle of gaps, hence $X$ 
must consist
of a single orbit of period $n$. If $\ell_0$ is the shortest gap length,
then the sum of the gap lengths is
$$ 1=\ell_0(1+2+2^2+\cdots+2^{n-1})=(2^n-1)\ell_0.$$
It follows that the shortest gap has length $\ell_0=1/(2^n-1)$;
hence the longest gap has length $2^{n-1}/(2^n-1).$ 

For each rational $\bt\in\RZ$ there is one and only one periodic orbit
 $X=X_\bt$ which has rotation number $\bt$ under doubling. 
In fact we can compute $X_\bt$ explicitly as follows. Evidently each point
$x\in X_\bt$ is
equal to the sum of the gap lengths between $x$ and $2x$. But we have
computed all of the gap lengths, and their cyclic order around the circle
is determined by the rotation number $\bt$.

 Note that an interval $\I_c$ of length $1/2$
has rotation number $\bt$ if and only if it
is contained in the longest gap for $X_\bt$ which has length 
 $2^{n-1}/(2^n-1)>1/2,$ where $n$ is the denominator of $\bt$.
Thus the interval of allowed $c$-values has
length equal to
\begin{equation}\label{e-fracsum}
\ell_\bt= \frac{ 2^{n-1}}{2^n-1}-\frac{1}{2}=\frac{1}{2(2^n-1)}.
\end{equation}
One can check\footnote{Note that $2\ell_\bt=1/2^n+1/2^{2n}+1/2^{3n}+\cdots$,
where $n$ is the denominator of $\bt$. In other words, we can write 
$2\ell_\bt$ as the sum of $1/2^k$ over {\it all} integers $k>0$
 for which the product $j=k\bt$ is an integer. Therefore 
 $2\sum_{0\le \bt<1}\ell_\bt=\sum_{k>0}\sum_{0\le j<k} 1/2^k=\sum_{k>0} k/2^k$.
But this last expression can be computed by  differentiating the identity
$\sum_{k\ge 0} x^k=1/(1-x)$ to obtain $\sum kx^{k-1}=1/(1-x)^2$,
then multiplying 
by $x$, and setting $x$ equal to $1/2$.}
that the sum of this quantity $\ell_\bt$ over all rational numbers 
\hbox{$0\le \bt= k/n<1$} in lowest terms 
 is precisely equal to $+1$.
It follows that the set of $c$ for which $\rot\bigl(X_2(\I_c)\bigr)$ is irrational
has measure zero.

Since the function $c\mapsto \rot\bigl(X_2(\I_c)\bigr)$ is monotone,
it follows that for each irrational $\bt$ there is exactly one corresponding
$c$-value. In particular, there is just one corresponding rotation set
 $X_\bt$. In the irrational case,
it is not difficult to check that $X_\bt$ has exactly one gap of length
$1/2^i$ for each $i\ge 1$.\end{remark}
\medskip

\begin{remark}{\bf Classifying Reduced Rotation Sets.}\label{R-gold}
Goldberg \cite{G} has given a complete classification of reduced rotation
 sets with rational rotation number, and Goldberg and Tresser \cite{GT}
have given a similar classification in the irrational case. For our purposes, 
the following partial description of these
 results will suffice. First let $X$ be a reduced rotation
set under multiplication by $d$ with rational rotation number
(so that $X$ is a finite union of periodic orbits).
Divide the
circle $\RZ$ into $d-1$ half-open intervals $\displaystyle\left[\frac{j}{d-1},
\frac{j+1}{d-1}\right)$, where the endpoints of these intervals
 are the $d-1$ 
fixed points of $\m_d$. Define the \textbf{\textit{deployment
sequence}}\footnote{This is a mild modification of Goldberg's definition,
 which counts the number of points in $\big[0,j/(d-1)\big)$.}
 of $X$ to be the $(d-1)$-tuple
 $(n_0,n_1,\ldots, n_{d-2})$ , where
$n_j\ge 0$ is the number of points in $\displaystyle X\cap\left[\frac{j}{d-1},
\frac{j+1}{d-1}\right)$.
\bigskip

\begin{theorem}[\bf Goldberg]\label{T-gold}
A reduced rotation set $X$ with rational rotation number is
uniquely determined by its deployment sequence and rotation number.
Furthermore, for rotation sets consisting
 of a single periodic orbit of period $n$,
 every one of the $\displaystyle \frac{(n+d-2)!}{n!(d-2)!}$
 possible deployment sequences with $n_j\ge 0$
and $\displaystyle \sum_j n_j=n$ actually occurs.
\end{theorem}

Note that a finite union of periodic orbits
 $\cO_1\cup\cdots\cup\cO_k$, all with the same rotation
number, is a rotation set if and only if the $\cO_j$ are pairwise
\textbf{\textit{compatible}}, in the sense that any gap in one of
these orbits contains exactly one point of each of the other orbits.
{\bf Caution:} Compatibility is not an equivalence relation.\footnote{
Similarly,  compatibility between humans is clearly 
not an equivalence relation.}
 For example,
the orbit $\{1/4,3/4\}$ under tripling is compatible with both
$\{1/8,3/8\}$ and $\{5/8,7/8\}$, but the last two are not compatible
with each other.

In fact, Goldberg  gives a complete classification of such
 rational rotation sets with more than one periodic orbit.
In particular, she shows that at most $d-1$ periodic orbits can be pairwise
compatible.\medskip

In the case of irrational rotation number $\bt$, a slightly modified definition
is necessary. Note that any irrational rotation set $X$ has a unique invariant
probability measure. The semi-conjugacy
$$ h: X\to\RZ\qquad{\rm satisfying}
\qquad h\big(\m_d(x)\big)=h(x)+\bt,$$
of Remark~\ref{R-rn} is one-to-one except on the countable subset consisting
of endpoints of gaps. 
(Compare the proof of Lemma~\ref{L-red}.) 
Therefore Lebesgue measure on $\RZ$ can be pulled back to a measure on
$X$ which is invariant under $\m_d|_X$.
The deployment sequence can now be defined as
 $(p_0,p_1,\ldots,p_{d-2})$, where $p_j\ge 0$ is the measure
of $X\cap\big[j/(d-1),(j+1)/(d-1)\big)$.

\begin{theorem}[\bf Goldberg and Tresser\footnote{See also \cite{Z}. We are
indebted to Zakeri for help in the preparation of this section}]
 A reduced irrational rotation set is uniquely
determined by its rotation number and deployment sequence. Furthermore,
for any irrational rotation number, any $(p_0,\ldots, p_{D-2})$
with $p_j\ge 0$ and $\sum p_j=1$ can actually occur.
\end{theorem}
\end{remark}
\medskip


\subsection{Semiconjugacies\label{R-sc}}

By a 
 \textbf{\textit{semiconjugacy}} from a circle map $g$ to a circle map
$f$ we will always mean a monotone degree one map $h$ satisfying
$f\circ h(x)=h\circ g(x)$ for all $x$.
Whenever such a semiconjugacy exists, it follows that $f$ and $g$ have
the same degree.  As an example, if $g$ is any monotone circle map of degree
$d>1$, and if $x_0$ is a fixed point of $g$, then it is not hard to
show that there is a unique semiconjugacy from  $g$ to $\m_D$ which
maps $x_0$ to zero: the monotone degree one map induced by 
\[\widehat h = \lim_{n\to +\infty} \widehat h_n \quad\text{with}\quad \widehat h_n (x)= \frac{1}{d^n}  \widehat g^{\circ n}(x),\] 
where $\widehat g:\R\to \R$ 
is a lift of $g:\RZ\to \RZ$ fixing a point above $x_0$.

%
%

Lemmas~\ref{L-XJ} and \ref{L-gaps}  have precise analogs in the study of
 semiconjugacy. For example, 
if $\I$ is a disjoint union of $\delta$ disjoint
open intervals of length exactly $1/d$, then there is a unique circle map
$g_{_\I}$ of degree $d-\delta$ 
 which agrees with $\m_d$ outside of $\I$
but maps each component of $\I$ to a point. Hence, as noted above,
there exists a semiconjugacy from the map
$g_{_\I}$ to $\m_{d-\delta}$, unique up to $d-1$ choices. 

Similarly, let $X\subset\RZ$ be any nonempty compact $\m_d$-invariant set.
Then there exists a monotone continuous map $g_X:\RZ \rightarrow \RZ$
 of lowest possible degree
which satisfies $g_X(x)=\m_d(x)$ whenever $x\in X$. Thus any gap in
$X$ of length $\ell$ must map to an interval (or point) of length equal
to the fractional part of $d\ell$.
 If $\delta$ is the number of gaps, counted with multiplicity
as in Lemma~\ref{L-gaps}, then it is not hard to see that $g_X$ has degree
$d-\delta$.

\end{document}